\documentclass[11pt]{article}
\usepackage{amsmath,amsfonts,amssymb,amsthm,graphics,epsfig,cite,color,indentfirst}
\usepackage{graphics}
\usepackage{subcaption}
\usepackage{mathtools}
\usepackage{commath}
\usepackage{mathrsfs}
\usepackage{verbatim}
\usepackage[sc,osf]{mathpazo}
\usepackage[notref,notcite]{showkeys} 
\usepackage[colorlinks=true]{hyperref} 

\newtheorem{theorem}{Theorem}[section]
\newtheorem{assumption}[theorem]{Assumption}
\newtheorem{corollary}[theorem]{Corollary}
\newtheorem{definition}[theorem]{Definition}
\newtheorem{lemma}[theorem]{Lemma}
\newtheorem{claim}[theorem]{Claim}
\newtheorem{proposition}[theorem]{Proposition}
\newtheorem{remark}[theorem]{Remark}
\theoremstyle{plain}

\newcommand{\al}{\alpha}

\newcommand{\la}{\lambda}
\newcommand{\ga}{\gamma}
\newcommand{\sig}{\sigma}
\newcommand{\de}{\delta}

\newcommand{\RR}{{\rm I\kern -1.6pt{\rm R}}}

\newcommand{\N}{\mathbb{N}}
\newcommand{\R}{\mathbb{R}}

\numberwithin{equation}{section}
\numberwithin{figure}{section}

\allowdisplaybreaks

\begin{document}
\title{\bf Age-structured Models with Nonlocal Diffusion of Dirichlet Type, I: Principal Spectral Theory and Limiting Properties\thanks{Research was partially supported by National Science Foundation (DMS-1853622). H. Kang would like to acknowledge the region Normandie for the financial support of his postdoctoral study.}}
	
\author{{\sc Arnaud Ducrot$^{a}$, Hao Kang$^{a, b}$ and Shigui Ruan$^{c}$}\\[2mm]
{\small $^a$Normandie Univ, UNIHAVRE, LMAH, FR-CNRS-3335, ISCN, 76600 Le Havre, France}\\
{\small $^b$Center for Applied Mathematics, Tianjin University, Tianjin 300072, China}\\
{\small Emails: arnaud.ducrot@univ-lehavre.fr and hao.kang@univ-lehavre.fr}\\
{\small $^c$Department of Mathematics, University of Miami, Coral Gables, FL 33146, USA} \\
{\small Email: ruan@math.miami.edu}
}
\maketitle
	
\begin{abstract}
Age-structured models with nonlocal diffusion arise naturally in describing the population dynamics of biological species and the transmission dynamics of infectious diseases in which individuals disperse nonlocally and interact each other and the age structure of individuals matters. In the first part of our series papers, we study the principal spectral theory of age-structured models with nonlocal diffusion of Dirichlet type. First, we provide two criteria on the existence of principal eigenvalues by using the theory of resolvent positive operators with their perturbations. Then we define the generalized principal eigenvalue and use it to investigate the influence of diffusion rate on the principal eigenvalue. In addition, we establish the strong maximum principle for age-structured nonlocal diffusion operators. In the second part \cite{Ducrot2022Age-structuredII} we will investigate the effects of principal eigenvalues on the global dynamics of the model with monotone nonlinearity in the birth rate and show that the principal eigenvalue being zero is critical. 
		
\vspace{0.3cm} \noindent {\em Key words:} Age structure; Nonlocal diffusion; Resolvent positive operators; Spectral bound; Generalized principal eigenvalue; Maximum principle.
		
\vskip 0.2cm
		
\noindent {\em AMS subject classifications:}~  35K57, 47A10, 92D25
		
\end{abstract}	
	
\tableofcontents
	
\section{Introduction}
In this paper, we study the existence of principal eigenvalue and asymptotic behavior of the principal eigenvalue with respect to diffusion rate and diffusion range of linear age-structured models with nonlocal diffusion of Dirichlet type. The motivation comes from investigating the following age-structured model with nonlocal diffusion under Dirichlet boundary condition:
\begin{equation}\label{originallogistic}
\begin{cases}
\!\!(\partial_t \!+\!\partial_a) u(t, a, x)\!\!=\!\! D \!\left[\!\int_{\Omega}\! J(x \!-\! y)u(t, a, y)dy \!-\! u(t, a, x)\!\right]\!\!-\!\!\mu(a, x)u(t, a, x),&\!\!\!\!\!(t, a, x)\!\in\!(0, \infty)\!\times\!(0, \hat{a})\!\times\!\overline{\Omega},\\
u(t, 0, x)=f\left(\int_{0}^{\hat{a}}\beta(a, x)u(t, a, x)da\right),&\!\!\!\!\!(t, x)\!\in\!(0, \infty)\!\times\!\overline{\Omega},\\
u(0, a, x)=u_0(a, x),&\!\!\!\!\!(a, x)\!\in\!(0, \hat{a})\!\times\!\overline{\Omega},
\end{cases}
\end{equation}
where $u(t, a, x)$ denotes the density of population at time $t$, with age $a$ and at position $x$, $J$ is a dispersal kernel and $f$ is a monotone type nonlinearity describing the birth rate of the population. Such equations appear naturally in describing some ecological problems when in addition to the dispersion of the individuals in the environment, the birth and death of these individuals are also modeled, see Fife \cite{fife2017integrodifferential}, Garc\'ia-Meli\'an and Rossi \cite{garcia2009principal},  Hutson et al. \cite{hutson2003evolution}, Medlock and Kot \cite{medlock2003spreading}, and Murray \cite{murray2007mathematical}. It could be used to characterize the spatio-temporal dynamics of biological species and transmission dynamics of infectious diseases in which the age structure of the population is a very important factor and the dispersal is in long distance. We mention that the nonlocal diffusion operator in \eqref{originallogistic} corresponds to zero Dirichlet boundary condition, which indicates that the region outside their habitat, $\mathbb{R}^N\setminus\overline{\Omega}$, is hostile that the population cannot survive there, see Hutson et al. \cite{hutson2003evolution}.
	
Here $\hat{a}\in(0, \infty]$ represents the maximum age and $\Omega\subset\mathbb{R}^N$ is a bounded domain. Moreover, $D>0$ is the diffusion rate and the diffusion kernel $J$ satisfies the following assumption.

\begin{assumption}\label{J}
{\rm The kernel $J\in C(\mathbb{R}^N)$ is nonnegative and supported in $B(0, r)$ for some $r>0$, where $B(0, r)\subset\mathbb{R}^N$ is the open ball centered at $0$ with radius $r$. In addition, $J$ satisfies $J(0)>0$ and $\int_{\mathbb{R}^N}J(x)dx=1$.}
\end{assumption}

Next we provide assumptions on the birth rate $\beta=\beta(a, x)$ and the death rate $\mu=\mu(a, x)$. Define
\begin{eqnarray*}
&&\underline{\mu}(a):=\min_{x\in\overline{\Omega}}\mu(a, x),\; \overline{\mu}(a):=\max_{x\in\overline{\Omega}}\mu(a, x), \\
&&\underline{\beta}(a):=\min_{x\in\overline{\Omega}}\beta(a, x),\; \overline{\beta}(a):=\max_{x\in\overline{\Omega}}\beta(a, x).
\end{eqnarray*} 

\begin{assumption}\label{Ass1.1}
{\rm We assume that the birth rate $\beta=\beta(a, x)$ and the death rate $\mu=\mu(a, x)$ satisfy the following conditions,
\begin{itemize}
\item [(i)] $\beta\in C(\R^N, L^\infty_+(0, \hat{a}))$;
				
\item [(ii)] $\mu\in C(\R^N, L^\infty_{{\rm loc}, +}[0, \hat{a}))$;
				
\item [(iii)] There exists $\widetilde{\mu}>0$ such that $\underline{\mu}(a)\ge \widetilde{\mu}>0$ a.e. $a\in(0, \hat{a})$;

\item [(iv)] For any $x\in\R^N$ and almost every $a\in (0, \hat{a}),$ 
\begin{equation*}
\underline{\beta}(a)\leq \beta(a,x)\text{ and }\mu(a,x)\leq \overline{\mu}(a).
\end{equation*}				
\end{itemize}
}
\end{assumption}

\begin{remark}
{\rm Note that from the biological modeling point of view, one usually also assumes $\int_0^{\hat{a}}\underline{\mu}(a)da=+\infty$ to guarantee that the population density reaches to zero at maximum age. To be able to consider such a situation, we only assume (ii) for $\mu$ so that $\underline{\mu}\in L^\infty_{{\rm loc}, +}[0, \hat a)$ and the integral can be infinite when $\hat a<\infty$. In addition, we mention that Assumption \ref{Ass1.1}-(iv) is employed to study the limiting properties of principal eigenvalue with respect to diffusion rate. In fact, the existence of principal eigenvalue is only dependent on the behavior of $\mu$ and $\beta$ in $\overline{\Omega}$. While for the limiting properties, in particular under the kernel scaling (see Theorem \ref{gamma}), we will use the behavior of $\mu$ and $\beta$ in the larger domain, for example in $\R^N$.
			}
\end{remark}
%
%
%
%
	Note that Assumptions \ref{J}-\ref{Ass1.1} will be required throughout the whole paper, thus in the following we will assume them hold everywhere without repeating them. However, we will indicate the additional assumptions that will be needed.
		
	The important tool for studying the global dynamics of \eqref{originallogistic} is to investigate the spectrum set of the linearized operator of \eqref{originallogistic} at some equilibrium, and then use the information of spectrum set (for example the sign of spectral bound or principal eigenvalue, if exists) to study the long time behavior of \eqref{originallogistic}. In this paper, we are interested in the principal eigenvalue problem of the linearization of \eqref{originallogistic} at zero. We will provide two criteria on the existence of principal eigenvalues by using the theory of resolvent positive operators with their perturbations. We will study the global dynamics of \eqref{originallogistic} in the second part \cite{Ducrot2022Age-structuredII}.
	
	The difficulty in establishing the existence of principal eigenvalues for \eqref{originallogistic} mainly comes from the nonlocal diffusion, since as Donsker and Varadhan \cite{donsker1975variational} have already noticed that there may not exist a principal eigenvalue associated with a positive eigenfunction for the nonlocal diffusion operator in the regular function spaces. Thus, in the following we first briefly recall the history of studying the principal eigenvalues of nonlocal operators; then point out the differences between the previous models and ours; and finally explain our ideas to treat the eigenvalue problem of \eqref{originallogistic}. 
	
	In 2010, Coville \cite{coville2010simple} employed the concept of generalized principal eigenvalues from Berestycki et al. \cite{berestycki1994principal} to study the existence of principal eigenvalues of nonlocal operators and gave a non-locally-integrable condition based on the generalized Krein-Rutman theorem (Edmunds et al. \cite{edmunds1972non} and Nussbaum \cite{nussbaum1981eigenvectors}). Later Berestycki et al. \cite{berestycki2016definition} further studied the problem in both bounded and unbounded domains, then investigated the asymptotic behavior of the generalized principal eigenvalue with respect to the diffusion rate. Related studies along this direction include Coville et al. \cite{coville2013pulsating,coville2013singular,coville2013nonlocal}, Li et al. \cite{li2017eigenvalue}, Su et al. \cite{su2020asymptotic}, Sun et al. \cite{sun2011entire}, Yang et al. \cite{yang2016principal,yang2019dynamics}, Sun et al \cite{sun2017periodic}, Brasseur \cite{brasseur2021role} and the references cited therein. On the other hand, Rawal and Shen \cite{rawal2012criteria}, Rawal et al. \cite{rawal2015spreading}, Shen and Xie \cite{Shen2015principal, shen2015approximations}, and Shen and Zhang \cite{shen2010spreading} investigated the existence of principal eigenvalues for autonomous and time periodic cases respectively, where they gave sufficient and necessary conditions for both cases by using the idea of perturbation of positive operators from B\"urger \cite{burger1988perturbations}. Combining these two directions, recently Shen and Vo \cite{shen2019nonlocal} and Su et al. \cite{su2020generalised} discussed the asymptotic behavior of the generalized principal eigenvalue with respect to the diffusion rate in the time-periodic case. There are also many other studies on the analysis of principal eigenvalues for nonlocal diffusion equations in different situations, the interested readers can refer to Liang et al. \cite{liang2019principal}, Smith \cite{smith2014sufficient}, Coville and Hamel \cite{coville2020generalized}, De Leenheer et al. \cite{de2020persistence}, Onyido and Shen \cite{onyido2021nonlocal} and the references cited therein.
	
	Although it seems to be natural to follow the idea of Rawal and Shen \cite{rawal2012criteria} (where they studied the time-periodic situation) to deal with our case, since there is also a derivative term $\partial_a$ in our model \eqref{originallogistic}, however, it is quite different from ours because there is an initial integral condition in \eqref{originallogistic}. We cannot directly choose the space of functions satisfying the integral condition for $a=0$ as in \cite{rawal2012criteria} where the authors worked on the space of time periodic functions, since such a function space is unknown and heavily depends on the birth rate $\beta$. Thus it forces us to treat the problem in a new and different way. First we define the operator $\mathcal{A}$ given in \eqref{A} containing the integral boundary condition. Then we recall the theory of resolvent positive operators with their perturbations to study the existence of principal eigenvalues. The concepts of resolvent positive operators were proposed long time ago, and then were further studied and developed by Arendt \cite{arendt1987resolvent}, Engel and Nagel \cite{engel2006short}, Kato \cite{kato1982superconvexity}, and Thieme \cite{thieme1998remarks,thieme2009spectral}. Next observing our case, since it contains the $\partial_a$ term, which is like a parabolic type of nonlocal operator, it does not admit the usual $L^2$ variational structure in the elliptic type case. This fact suggests us to follow the idea of Berestycki et al. \cite{berestycki2016definition} to define and study the generalized principal eigenvalue of \eqref{originallogistic} when investigating their asymptotic behavior with respect to the diffusion rate.
	
	We would like to mention that here we propose a kind of new method to study the eigenvalue problem of age-structured models with nonlocal diffusion. Such a method is different from but closely related to the ideas from both Coville \cite{coville2010simple} and Rawal and Shen \cite{rawal2012criteria}. In this method, we first characterize the spectral bounds of two related operators and then employ the theory of resolvent positive operators with their perturbations to obtain the existence of principal eigenvalues by comparing the two spectral bounds. We point out that such an idea is basically similar to B\"urger's perturbation of positive semigroups \cite{burger1988perturbations} which Rawal and Shen \cite{rawal2012criteria} employed, but they required that the operator has a dense domain and generates a positive semigroup of contractions, which seems to be restrictive and in general not satisfied in our case. Further, we will bridge a close link between our idea and Coville's \cite{coville2010simple,coville2013pulsating}, see Remark \ref{GKR} in Section \ref{Principal Spectral Theory}. In addition, this work of combining these two features (age structure and nonlocal diffusion) is also an extension of our previous studies in Kang et al. \cite{Kang2020Age} and Kang and Ruan \cite{kang2021nonlinear}. 
	
	The paper is organized as follows. In Section \ref{Notations}, we introduce our fundamental setting including operators and function spaces. In Section \ref{Preliminaries}, we present some necessary propositions and lemmas in order to prove the main theorems later. It includes two important propositions for characterizing the spectral bounds of two key operators respectively. More precisely, we analyze the spectral bound $s(\mathcal{B}_1+\mathcal{C})$ of $\mathcal{B}_1+\mathcal{C}$ which corresponds to the age-structured model without nonlocal diffusion and the one $s(\mathcal{A})$ of $\mathcal{A}$ defined in \eqref{A} which corresponds to the age-structured model with nonlocal diffusion respectively, and then obtain a non-strict size relation between them. In Section \ref{Principal Spectral Theory}, we provide the main result on the existence of principal eigenvalues based on the strict size relation between $s(\mathcal{B}_1+\mathcal{C})$ and $s(\mathcal{A})$, and find two relatively easily verifiable and sufficient conditions for $s(\mathcal{A})$ being the principal eigenvalue. Further, we show by giving a counterexample that such conditions are sharp in the sense that if they are not satisfied, $\mathcal{A}$ admits no principal eigenvalue. In Section \ref{LP}, we study the effects of the diffusion rate and diffusion range on the generalized principal eigenvalue of $\mathcal{A}$ and discuss the continuous dependence of the principal eigenvalue on the birth and death rates $\beta$ and $\mu$. In Section \ref{Strong Maximum Principle}, we give the strong maximum principle which is of fundamental importance and independent interest. For the readers' convenience, we recall the theory of resolvent positive operators with their perturbations in Appendix. In addition, we mention that our work here is different from the previous one \cite{kang2021principal} where we considered the Neumann boundary condition and in particular, the limiting properties are quite different.
	
	Finally, we want to mention that the conditions introduced at the beginning that $J$ has a compact support and $\Omega$ is bounded can be relaxed. For the principal spectral theory, we only need $\Omega$ to be bounded without requiring that $J$ has a compact support. However, the boundedness of $\Omega$ seems necessary due to the lack of Harnack's inequality for such parabolic problems, see Shen and Vo \cite{shen2019nonlocal}. Moreover, in order to study the limiting properties of principal eigenvalues with kernel scaling, $J$ is needed to be compactly supported for Taylor expansion later. In addition, the condition that $\Omega$ is bounded can be removed if one only defines the generalized principal eigenvalue, see Berestycki \cite{berestycki2016definition}. Here to give a unified presentation of the results, we assume both of them.
	
	\section{Notations}\label{Notations} 
	In this section we introduce our notations and some preparatory results. Denote by $X$ and $X_+$ respectively the Banach space $X=C(\overline{\Omega})$ and its positive cone or the Banach space $X=L^1(\Omega)$ and its positive cone. Here $\Omega\subset\mathbb{R}^N$ is a given bounded domain. Recall that for both cases $X_+$ is a normal and generating cone. In addition, we denote by $I$ the identity operator.
	
	Then we define the following function spaces
	\begin{equation*}
		\mathcal{X}=X\times L^1((0, \hat{a}), X), \; \mathcal{X}_0=\{0_X\}\times L^1((0, \hat{a}), X),
	\end{equation*}
	endowed with the product norms and the positive cones:
	\begin{equation*}
		\mathcal{X}^+=X_+\times L^1_+((0, \hat{a}), X)=X_+\times\{u\in L^1((0, \hat{a}), X): u(a, \cdot)\in X_+,\; \text{a.e. in }(0, \hat{a})\}, \; \mathcal{X}^+_0=\mathcal{X}^+\cap \mathcal{X}_0.
	\end{equation*}
	We also define the linear positive and bounded operator $K\in\mathcal L(X)$ by
	\begin{equation}\label{op-K}
		[K\varphi](\cdot)=\int_\Omega J(\cdot-y)\varphi(y)dy,\;\forall \varphi\in X.
	\end{equation}	
	Note that due to Assumption \ref{J} one has
	\begin{equation}\label{norm-K}
		\|K\|_{\mathcal L(X)}\leq \begin{cases} \sup_{y\in\Omega}\int_{\Omega} J(x-y)dx\text{ if $X=L^1(\Omega)$}\\ \sup_{x\in\overline{\Omega}}\int_{\Omega} J(x-y)dy\text{ if $X=C(\overline\Omega)$}\end{cases}\leq \int_{\R^N}J(z)dz=1.
	\end{equation}

	\subsection{Evolution Family Without Diffusion}
	
	We consider the following problem posed in $X$ for $0\leq\tau\leq a<\hat a$:
	\begin{equation}\label{B_1}
		\begin{cases}
			\partial_av(a)=-\mu(a, \cdot)v(a), \;\tau<a<\hat a,\\
			v(\tau)=\eta\in X.
		\end{cases}
	\end{equation}
	This problem generates an evolution family on $X$, denoted by $\Pi$, that is explicitly given for $0\leq\tau\leq a<\hat a$ and $\eta\in X$ by 
	\begin{equation}\label{def-pi}
		\begin{split}
			&\Pi(\tau,a)\eta=\pi(\tau,a,\cdot)\eta\\
			&\text{with  $\pi(\tau, a, x):=\exp\left(-\int_\tau^a\mu(s, x)ds\right)$ for $0\leq\tau\leq a<\hat a$ and $x\in\overline\Omega$}.
		\end{split}
	\end{equation}		
	Observe that one has 
	\begin{equation}\label{expbdd}
		\|\Pi(\tau,a)\|_{\mathcal L(X)}\leq \exp\left(-\int_\tau^a\underline{\mu}(s)ds\right)\leq e^{-\widetilde{\mu}(a-\tau)}\leq 1,\;\;\forall\,0\leq\tau\leq a<\hat a. 
	\end{equation}
	We also define the following family of bounded linear operators $\{W_\lambda\}_{\lambda>-\widetilde{\mu}}\subset \mathcal L\left(\mathcal{X},\mathcal{X}_0\right)$ for $(\eta,g)\in \mathcal X$ by 
	\begin{eqnarray}
		W_\lambda(\eta, g)&=&(0, h)\nonumber\\
		\text{with } h(a)&=&e^{-\lambda a}\Pi(0, a)\eta+\int_0^a e^{-\lambda(a-s)}\Pi(s, a)g(s)ds.\label{W_alpha}
	\end{eqnarray}
	We will show that this provides a family of positive pseudoresolvents. To this aim, one can make some computations to obtain
	\begin{eqnarray}
		&&W_\nu W_\lambda(\eta, g)\nonumber\\
		&& \quad = \int_0^ae^{-\nu(a-s)}\Pi(s, a)e^{-\lambda s}\Pi(0, s)\eta ds+\int_{0}^{a}e^{-\nu(a-s)}\Pi(s, a)\int_{0}^{s}e^{-\lambda(s-\tau)}\Pi(\tau, s)g(\tau)d\tau ds\nonumber\\
		&& \quad = \int_0^ae^{-\nu a}e^{-(\lambda-\nu)s}ds\Pi(0, a)\eta+\int_0^a\int_0^se^{\lambda\tau-\nu a}e^{-(\lambda-\nu)s}\Pi(\tau, a)g(\tau)d\tau ds.\nonumber
	\end{eqnarray}
	Hence for $\nu\neq\lambda$, we have
	\begin{eqnarray}
		&&W_\nu W_\lambda(\eta, g)\nonumber\\
		&& \quad = \frac{1}{\nu-\lambda}\left(e^{-\lambda a}-e^{-\nu a}\right)\Pi(0, a)\eta+\frac{1}{\nu-\lambda}\left(e^{-(\lambda-\nu) a}-e^{-(\lambda-\nu)\tau}\right)\int_0^a e^{\lambda\tau-\nu a}\Pi(\tau, a)g(\tau)d\tau\nonumber\\
		&& \quad = \frac{1}{\nu-\lambda}\left(W_\lambda-W_\nu\right)(\eta, g).\nonumber
	\end{eqnarray}
	Moreover, one see (for example Magal and Ruan \cite[Lemma 3.8.3]{magal2018theory}) that for all $\lambda>-\widetilde{\mu}$,
	$$ 
	W_\lambda(\eta, g)=0 \text{ only occurs if }\eta=0, g=0
	$$ 
	and 
	$$
	\lim\limits_{\lambda\to\infty}\lambda W_\lambda(0, g)=(0, g), \;\forall (0, g)\in \mathcal{X}_0. 
	$$
	Moreover, one has
	$$
	\norm{W_\la}_{\mathcal{L}(\mathcal{X}, \mathcal{X}_0)}\le\frac{1}{\la+\widetilde{\mu}}.
	$$
	Thus, by Pazy \cite[Section 1.9]{pazy2012semigroups} there exists a unique closed Hille-Yosida operator $\widetilde B_1: {\rm dom}(\widetilde B_1)\subset \mathcal{X}\to\mathcal{X}$ with $\overline{{\rm dom}(\widetilde B_1)}=\mathcal X_0$ such that 
	\begin{equation}\label{def B_1}
		\left(\lambda I-\widetilde B_1\right)^{-1}=W_\lambda\text{ for all $\lambda>-\widetilde{\mu}$.} 
	\end{equation}
	Recalling \eqref{op-K} we also define the bounded linear operator $\mathcal B_2\in \mathcal L(\mathcal X_0)$ by  
	$$
	\mathcal{B}_2(0, g)=\left(0, DKg(\cdot)\right),\; \forall (0, g)\in\mathcal{X}.
	$$
	
	\subsection{Evolution Family With Diffusion}  
	Consider now the following evolution equation for $\eta\in X$ and $0\leq\tau\leq a<\hat a$:
	\begin{equation}\label{steadystatequation}
		\begin{cases}
			\partial_a u(a)=D(K-I_X) u(a)-\mu(a, \cdot)u(a),\; \tau< a<\hat a,\\
			u(\tau)=\eta\in X.
		\end{cases}
	\end{equation}
	Define the evolution family $\{\mathcal{U}(\tau, a)\}_{0\leq\tau\leq a< \hat{a}}\subset\mathcal{L}(X)$ associated with \eqref{steadystatequation}. Using the constant of variation formula, $\mathcal U$ becomes for all $0\leq\tau\leq a<\hat a$ the solution of the equation 
	\begin{equation}\label{U}
		\begin{cases}
		\mathcal{U}(\tau, a)=e^{-D(a-\tau)}\Pi(\tau,a)+D\int_\tau^a e^{-D(a-l)}\Pi(l,a)K\,\mathcal{U}(\tau, l)dl,\\
		\mathcal U(\tau, \tau)=I_X.
		\end{cases}
	\end{equation}
	Here $I_X$ is the identity operator in $X$. Note that the right hand side of \eqref{steadystatequation} is linear and bounded with respect to $u$, thus the existence and uniqueness of $\{\mathcal{U}(\tau, a)\}_{0\le\tau\le a<\hat{a}}$ can be obtained from the general semigroup theory (see Pazy \cite{pazy2012semigroups}). Next let us prove that $\{\mathcal{U}(\tau, a)\}_{0\le\tau\le a<\hat{a}}$ is exponentially bounded. 
	
	To this aim fix $\eta\in X$, $\tau\in [0,\hat a)$ and set $u(a)=\mathcal U(\tau,a)\eta$. Then one has
	$$
	\norm{u(a)}_X\leq e^{-(D+\widetilde \mu)(a-\tau)}\|\eta\|_X+D\|K\|_{\mathcal L(X)}\int_\tau^a e^{-(D+\widetilde \mu)(a-l)}\|u(l)\|_Xdl.
	$$	
	Next Gronwall's inequality yields
	$$
	\norm{u(a)}_X e^{(D+\widetilde \mu)(a-\tau)}\le \norm{\eta}_Xe^{D\norm{K}_{\mathcal{L}(X)}(a-\tau)},
	$$
	which implies due to \eqref{norm-K} that
	$$
	\norm{\mathcal U(\tau,a)}_{\mathcal L(X)}\le e^{-\widetilde{\mu}(a-\tau)}.
	$$
	As a consequence $\{\mathcal{U}(\tau, a)\}_{0\le \tau\le a<\hat{a}}$ is positive and exponentially bounded in $X$ and satisfies
	\begin{equation}\label{EB}
		\norm{\mathcal{U}(a, a+t)}_{\mathcal{L}(X)}\le e^{-\widetilde{\mu} t},\;\forall t\ge0, 0\le a<\hat{a}-t.
	\end{equation}
	Now we define the family of bounded linear operators $\{R_\lambda\}_{\lambda>-\widetilde{\mu}}\subset\mathcal{L}(\mathcal{X}, \mathcal{X}_0)$ as follows:
	\begin{eqnarray}\label{R_lambda}
		R_\lambda(\eta, g)&=&(0, h)\nonumber\\
		\text{ with } h(a)&=&e^{-\lambda a}\mathcal{U}(0, a)\eta+\int_0^ae^{-\lambda(a-s)}\mathcal{U}(s, a)g(s)ds.
	\end{eqnarray}
	Moreover, for any $\la>-\widetilde{\mu}$, one has
	$$
	\norm{R_\la}_{\mathcal{L}(\mathcal{X}, \mathcal{X}_0)}\le\frac{1}{\la+\widetilde{\mu}}.
	$$
	Then by the same procedure as in the case without diffusion, we can prove that this provides a family of positive pseudoresolvents. Thus again by Pazy \cite[Section 1.9]{pazy2012semigroups} there exists a unique closed Hille-Yosida operator $\mathcal{B}: {\rm dom}(\mathcal B)\subset\mathcal{X}\to\mathcal{X}$ with $\overline{{\rm dom}(\mathcal B)}=\mathcal X_0$ such that 
	\begin{equation*}
		\left(\lambda I-\mathcal{B}\right)^{-1}=R_\lambda\text{ for all $\lambda>-\widetilde{\mu}$}. 
	\end{equation*}
	Next we define the part of $\mathcal{B}$ in $\mathcal{X}_0$, denoted by $\mathcal{B}_0$. That is, 
	$$ 
	\mathcal{B}_0x=\mathcal{B}x, \, \forall x\in D(\mathcal{B}_0),  \text{ with }  {\rm dom}(\mathcal{B}_0):=\{x\in {\rm dom}(\mathcal{B}): \mathcal{B}x\in \mathcal{X}_0\}. 
	$$
	Note that $\mathcal{B}_0$ is the infinitesimal generator of a strongly continuous semigroup of bounded linear operators on $\mathcal{X}_0$, denoted by $\{T_{\mathcal{B}_0}(t)\}_{t\geq0}$. Moreover, it satisfies the following estimate
	$$
	\left\|T_{\mathcal{B}_0}(t)\right\|_{\mathcal{L}(\mathcal{X}_0)}\le e^{-\widetilde{\mu} t},\;\forall t\ge0.
	$$
	Observe now that we have $\widetilde B_1+\mathcal B_2-DI=\mathcal B$. From now on for the sake of convenience, we denote $\mathcal B_1:=\widetilde B_1-DI$.
	
	On the other hand, we define $\mathcal{C}\in\mathcal L(\mathcal X_0,\mathcal X)$ by
	$$
	\mathcal{C}(0, h)=\left(\int_{0}^{\hat{a}}\beta(a, \cdot)h(a)da, \;  0\right),\; (0, h)\in\mathcal{X}_0,
	$$
	and  $\mathcal{A}:{\rm dom}(\mathcal{A})\subset \mathcal{X}\to \mathcal{X}$ by
	\begin{equation}\label{A}
		\begin{cases}
			{\rm dom}(\mathcal{A})={\rm dom}(\mathcal{B})\subset \mathcal X_0,\\
			\mathcal{A}=\mathcal{B}+\mathcal{C}.
		\end{cases}
	\end{equation}
	This shows that $\mathcal{A}$ is not densely defined in $\mathcal{X}$.
	
\begin{remark}\label{age-operator}
{\rm In addition, for each fixed $x\in\overline{\Omega}$, following the above procedures, one can obtain the age-structured operator, denoted by $\mathcal{B}_1^x+\mathcal{C}^x$, defined on $\R\times L^1(0, \hat a)$.}
	\end{remark}  
	
	Now define the map $F: \mathcal{X}_0\to \mathcal{X}$ by
	$$
	F\begin{pmatrix}
		0, \psi
	\end{pmatrix}
	=\begin{pmatrix}
		f\left(\int_0^{\hat{a}}\beta(a, \cdot)\psi(a)da\right), 0
	\end{pmatrix}.
	$$
	Then by identifying $U(t)=\begin{pmatrix}
		0, u(t)
	\end{pmatrix}$, one can rewrite problem \eqref{originallogistic} as the following abstract Cauchy problem:
	\begin{equation}\label{Cauchy}
		\begin{cases}
			\frac{dU}{dt}=\mathcal{B}U+F(U),\\
			U(0)=U_0,
		\end{cases}\text{ with }U_0=\begin{pmatrix}
			0, u_0
		\end{pmatrix}\in\mathcal{X}_0.
	\end{equation}
	As mentioned before, we will study the principal spectral theory of the linearized problem corresponding to \eqref{Cauchy}; that is, the principal spectral theory of $\mathcal{A}=\mathcal{B}+f'(0)\mathcal{C}$. For the sake of convenience, we first ignore the constant $f'(0)$ before investigating the global dynamics of \eqref{originallogistic}, which is left in our forthcoming paper \cite{Ducrot2022Age-structuredII}.
	
	Finally, let us introduce briefly our idea to the existence of principal eigenvalues. Observe that if $\alpha\in\rho(\mathcal{B}_1+\mathcal{C})$, then the equation
	$$ 
	\mathcal{A}u=(\mathcal{B}_2+\mathcal{B}_1+\mathcal{C})u=\alpha u 
	$$
	has nontrivial solutions in $\mathcal{X}_0$ is equivalent to the euqattion 
	$$ 
	\mathcal{B}_2(\alpha I -\mathcal{B}_1-\mathcal{C})^{-1}v=v 
	$$
	has nontrivial solutions in $\mathcal{X}$. Next on one hand, we will prove that $\mathcal{A}$ is a positive and compact perturbation of $\mathcal{B}_1+\mathcal{C}$ (see Appendix for precise definitions). On the other hand, we will provide some relatively easy to verify and general sufficient conditions for $s(\mathcal{A})>s(\mathcal{B}_1+\mathcal{C})$. Finally we apply the theory of resolvent positive operators with their perturbations to study the existence of principal eigenvalues of our problem.
	
	Before ending this section, we would like to mention that the well-posedness of the Cauchy problem \eqref{Cauchy} has been investigated in an abstract setting by the theory of integrated semigroups, see Thieme \cite{thieme1998positive,thieme2009spectral}, and in particular Magal and Ruan \cite{magal2007integrated} in a more general framework where the operators are neither densely-defined nor of Hille-Yosida type, for example in $L^p$ spaces with $p>1$. Here we focus on the principal spectral theory and global dynamics.
	
	\section{Preliminaries}\label{Preliminaries}
	In this section we present some necessary propositions and lemmas to 1) establish the existence of the spectral bounds of $\mathcal{B}_1+\mathcal{C}$ and $\mathcal{A}$ which correspond to the evolution families without diffusion ($\{\Pi(\tau, a)\}_{0\le\tau\le a<\hat{a}}$) and with diffusion ($\{\mathcal{U}(\tau, a)\}_{0\le\tau\le a<\hat{a}}$) respectively; 2) show that $\mathcal{A}$ is a positive and compact perturbation of $\mathcal{B}_1+\mathcal{C}$. We emphasize that the following results hold for both $X=C(\overline{\Omega})$ and $X=L^1(\Omega)$ if we do not indicate what $X$ is exactly. Moreover, we define the intervals $V$ and $\widetilde V$ as follows,
	\begin{equation}\label{I interval}
		V=\begin{cases}
			\R, &\text{ if }\hat{a}<\infty,\\
			(-D-\widetilde{\mu}, \infty), &\text{ if }\hat{a}=\infty.
		\end{cases} \text{ and }\widetilde V=\begin{cases}
			\R, &\text{ if }\hat{a}<\infty,\\
			(-D\la^0-\widetilde{\mu}, \infty), &\text{ if }\hat{a}=\infty.
		\end{cases}
	\end{equation}
	Here $0<\lambda^0<1$ is the principal eigenvalue of $-K+I$ associated with a positive eigenfunction $\phi_0\in C(\overline{\Omega})$ (see Garc\'ia-Meli\'an and Rossi \cite[Theorem 2.1]{garcia2009principal}). 
	
	In order to deal with the case $\hat{a}=+\infty$, we provide the following additional assumption throughout this section.
	\begin{assumption}\label{mu}
		{\rm Denote $\overline{\pi}(a)=e^{-\int_0^a\overline{\mu}(s)ds}$. If $\hat a=\infty$, we assume that there exists a real number $\widehat{\lambda}\in V$ such that 
			$$
			\widehat{R}:=\int_0^\infty\underline{\beta}(a)e^{-(\widehat{\lambda}+D)a}\overline{\pi}(a)da>1.
			$$
		} 
	\end{assumption}
	
\begin{remark}\label{rkk}
{\rm Note that coupling the above Assumption \ref{mu} together with Assumption \ref{Ass1.1}-(iv) this ensures that when $\hat a=\infty$ then we have for all $x\in\R^N$:
$$
1<\widehat{R}\leq \int_0^\infty \beta(a,x)e^{-(\widehat{\lambda}+D)a}\pi(0,a,x)da.
$$ 
This property will be used below to construct principal eigenvectors for any given and fixed $x\in\R^N$.}
\end{remark}

	\begin{remark}\label{mu2}
		{\rm 
			Assumption \ref{mu} is employed to guarantee the existence of spectral bound of $s(\mathcal{B}_1+\mathcal{C})$. Moreover, by setting 
			$$
			\widetilde{\lambda}:=\widehat{\lambda}+D-D\lambda^0,
			$$
			Assumption \ref{mu} reads as follows: there exists a number $\widetilde{\lambda}\in \widetilde V$ such that 
			$$
			\widehat{R}=\int_0^\infty\underline{\beta}(a)e^{-(\widetilde{\lambda}+D\lambda^0)a}\overline{\pi}(a)da>1.
			$$
			Later, we will see that this new condition $\widehat{R}>1$ is used to obtain the existence of spectral bound of $s(\mathcal{A})$ when $\hat{a}=+\infty$.
		}
	\end{remark}
	
	\subsection{Characterization of $s(\mathcal{B}_1+\mathcal{C})$}
	Now recalling that the functions $\{\pi(\tau, a,x)\}_{0\le\tau\le a<\hat{a}, x\in\overline\Omega}$ given in \eqref{def-pi}, we define for $\alpha\in V$ a continuous function $G_\alpha: \overline{\Omega}\to\R$ by
	\begin{eqnarray}\label{Galpha}
		G_\alpha(x)&=&\int_{0}^{\hat{a}}\beta(a, x)e^{-(\alpha+D)a}\pi(0, a, x)da,\;\forall x\in \overline\Omega.\label{Galpha(x)}
	\end{eqnarray}
	We also consider for $\alpha\in V$ a multiplication operator $\mathcal{G}_{\alpha}\in \mathcal L(X)$ given by
	\begin{eqnarray}
		[\mathcal{G}_\alpha g](x)&=&G_\alpha(x)g(x), \; g\in X.\label{Galphadef}
	\end{eqnarray}
	
	Then the following proposition holds.
	\begin{proposition}\label{rGalpha}
		Under Assumption \ref{mu}, there exists $\alpha^{**}\in V$ satisfying the equation 
		\begin{equation}\label{max}
			\max_{x\in\overline{\Omega}}G_{\alpha^{**}}(x)=\max_{x\in\overline{\Omega}}\int_{0}^{\hat{a}}\beta(a, x)e^{-(\alpha^{**}+D)a}\pi(0, a, x)da=1.
		\end{equation} 				
		Moreover, $\mathcal{B}_1+\mathcal{C}$ is a resolvent positive operator with $s(\mathcal{B}_1+\mathcal{C})=\alpha^{**}$ and
		\begin{equation}\label{rG}
			r\left(\mathcal{G}_{\alpha^{**}}\right)=r\left(\int_{0}^{\hat{a}}\beta(a, \cdot)e^{-(\alpha^{**}+D)a}\Pi(0, a)da\right)=1.
		\end{equation}	
	\end{proposition}
	
	\begin{proof}
		Observe that for $\al\in\rho(\mathcal B_1)$ the operator $\alpha I-\mathcal{B}_1-\mathcal{C}$ is invertible if and only if the operator $I-\mathcal{C}(\alpha I-\mathcal{B}_1)^{-1}$ is invertible. In that case, we have
		$$
		(\alpha I-\mathcal{B}_1-\mathcal{C})^{-1}=(\alpha I-\mathcal{B}_1)^{-1}\left[I-\mathcal{C}(\alpha I-\mathcal{B}_1)^{-1}\right]^{-1}.
		$$
		We now compute the inverse of $I-\mathcal{C}(\alpha I-\mathcal{B}_1)^{-1}$. To this aim choose $\alpha\in\rho(\mathcal{B}_1)$ and for some $(\eta, \varphi), (\widehat\eta, \widehat\varphi)\in\mathcal{X}$ consider 
		$$
		(\widehat{\eta}, \widehat{\varphi})=\left[I-\mathcal{C}(\alpha I-\mathcal{B}_1)^{-1}\right](\eta, \varphi).
		$$
		First we define
		\begin{eqnarray}
			(0, \phi)=(\alpha I-\mathcal{B}_1)^{-1}(\eta, \varphi).\nonumber
		\end{eqnarray}
		It follows that
		$$
		\widehat{\varphi}=\varphi \text{ and }\widehat{\eta}=\eta-\int_0^{\hat{a}}\beta(a, \cdot)\phi(a)da.
		$$
		Next recall from \eqref{W_alpha} that one has
		$$
		\phi(a)=e^{-(\alpha+D)a}\Pi(0, a)\eta+\int_0^ae^{-(\alpha+D)(a-s)}\Pi(s, a)\varphi(s)ds.
		$$
		It follows that 
		\begin{eqnarray}
			&& \eta-\int_{0}^{\hat{a}}\beta(s, \cdot)e^{-(\alpha+D)s}\Pi(0, s)\eta ds \nonumber\\
			&&\quad\quad =\int_{0}^{\hat{a}}\beta(s, x)\int_{0}^{s}e^{-(\alpha+D)(s-\tau)}\Pi(\tau, s)\widehat{\varphi}(\tau)d\tau ds+\widehat{\eta},\nonumber
		\end{eqnarray}
		which is equivalent to 
		\begin{equation}\label{I-G}
			(I-\mathcal{G}_\alpha)\eta=\int_{0}^{\hat{a}}\beta(s, \cdot)\int_{0}^{s}e^{-(\alpha+D)(s-\tau)}\Pi(\tau, s)\widehat{\varphi}(\tau)d\tau ds+\widehat{\eta},
		\end{equation} 
		where $\mathcal{G}_\alpha$ is defined in (\ref{Galphadef}). Thus if $1\in\rho(\mathcal{G}_\alpha)$ for $\al\in V$, then
		\begin{equation}\label{(I-G)^{-1}}
			\eta=(I-\mathcal{G}_\alpha)^{-1}\left[\int_{0}^{\hat{a}}\beta(s, \cdot)\int_{0}^{s}e^{-(\alpha+D)(s-\tau)}\Pi(\tau, s)\widehat{\varphi}(\tau)d\tau ds+\widehat{\eta}\right], 
		\end{equation} 
		which implies that 
		\begin{eqnarray}\label{resolventsolution1}
			(\eta, \varphi)&=&\left[I-\mathcal{C}(\alpha I-\mathcal{B}_1)^{-1}\right]^{-1}(\widehat{\eta}, \widehat{\varphi})\nonumber\\
			&=&\left((I-\mathcal{G}_\alpha)^{-1}\left[\int_{0}^{\hat{a}}\beta(s, \cdot)\int_{0}^{s}e^{-(\alpha+D)(s-\tau)}\Pi(\tau, s)\widehat{\varphi}(\tau)d\tau ds+\widehat{\eta}\right], \widehat{\varphi}\right).
		\end{eqnarray} 
		It follows that $\alpha\in\rho(\mathcal{B}_1+\mathcal{C})$ and thus $(\alpha I -\mathcal{B}_1-\mathcal{C})^{-1}$ exists. Now we have shown that 
		$$
		\al\in\rho(\mathcal{B}_1+\mathcal{C})\cap V\Leftrightarrow \al\in V \text{ and } 1\in\rho(\mathcal{G}_\al),
		$$
		thus the problem is inverted into finding such $\alpha\in V$ satisfying $1\in\rho(\mathcal{G}_\alpha)$. 
		
		Observe that $\mathcal{G}_\alpha$ is actually a positive multiplication operator in $X$, thus its spectrum is quite clear (for example see Liang et al. \cite[Proposition 2.7]{liang2019principal}), that is 
		$$
		\sigma(\mathcal{G}_\alpha)=\left[\min_{x\in\overline{\Omega}}G_\alpha(x), \;\max_{x\in\overline{\Omega}}G_\alpha(x)\right], 
		$$
		where $\min_{x\in\overline{\Omega}}G_\alpha(x)$ and $\max_{x\in\overline{\Omega}}G_\alpha(x)$ satisfy the following equations respectively 
		$$
		\min_{x\in\overline{\Omega}}G_\alpha(x)=\min_{x\in\overline{\Omega}}\int_{0}^{\hat{a}}\beta(a, x)e^{-(\alpha+D)a}\pi(0, a, x)da
		$$
		and
		$$
		\max_{x\in\overline{\Omega}}G_\alpha(x)=\max_{x\in\overline{\Omega}}\int_{0}^{\hat{a}}\beta(a, x)e^{-(\alpha+D)a}\pi(0, a, x)da.
		$$
		Observe that for any $x\in\overline{\Omega}$, $\alpha\to G_\alpha(x)$ is decreasing with respect to $\alpha\in V$ from $\infty$ to $0$ when $\hat{a}<\infty$, and from $\lim\limits_{\al\to-D-\widetilde{\mu}}G_\al(x)\ge\widehat{R}>1$ to $0$ when $\hat{a}=\infty$ respectively, due to Assumption \ref{mu}. It follows from $\alpha\in\rho(\mathcal{B}_1+\mathcal{C})\cap V\Leftrightarrow \al\in V \text{ and } 1\in\rho(\mathcal{G}_\alpha)$ that $(\alpha^{**}, \infty)\subset\rho(\mathcal{B}_1+\mathcal{C})$, where $\alpha^{**}$ satisfies the following equation
		$$
		\max_{x\in\overline{\Omega}}G_{\alpha^{**}}(x)=\max_{x\in\overline{\Omega}}\int_{0}^{\hat{a}}\beta(a, x)e^{-(\alpha^{**}+D)a}\pi(0, a, x)da=1.
		$$
		Hence the result \eqref{max} is desired, $s(\mathcal{B}_1+\mathcal{C})=\alpha^{**}>-D-\widetilde{\mu}$. In addition, $\mathcal{B}_1+\mathcal{C}$ is resolvent positive due to the fact that $(\alpha I-\mathcal{B}_1-\mathcal{C})^{-1}$ is positive by \eqref{resolventsolution1} and $\rho(\mathcal{B}_1+\mathcal{C})$ contains a ray $(\alpha^{**}, \infty)$. Hence the proof is done.
	\end{proof}
	
\begin{remark}\label{omegaT}
{\rm Since $\mathcal{B}_1+\mathcal{C}$ is resolvent positive, we have $r((\al I -\mathcal{B}_1-\mathcal{C})^{-1})=(\al -s(\mathcal{B}_1+\mathcal{C}))^{-1}$ for all $\al>s(\mathcal{B}_1+\mathcal{C})$ by Corollary \ref{spr}. Now let $(\mathcal{B}_1+\mathcal{C})_0$ be the part of $\mathcal{B}_1+\mathcal{C}$ in $\mathcal{X}_0$. It generates a positive $C_0$-semigroup $\{T_{(\mathcal{B}_1+\mathcal{C})_0}(t)\}_{t\geq0}$ on $\mathcal{X}_0$. Since $\mathcal{B}_1+\mathcal{C}$ is resolvent positive, by Thieme \cite[Proposition 2.4]{thieme1998positive}, we know that $s(\mathcal{B}_1+\mathcal{C})=s((\mathcal{B}_1+\mathcal{C})_0)=\omega(T_{(\mathcal{B}_1+\mathcal{C})_0})$ when $X=L^1(\Omega)$, since now $\mathcal{X}$ is an abstract L space, where $\omega(T)$ denotes the growth bound of $\{T(t)\}_{t\ge0}$, see Definition \ref{growth bound} in Appendix. 
		}
	\end{remark}
	
	\subsection{Characterization of $s(\mathcal{A})$}
	Next we will prove that $\mathcal{A}$ is resolvent positive and provide a precise characterization of its spectral bound $s(\mathcal{A})$. Recall that $\{\mathcal{U}(\tau, a)\}_{0\le\tau\le a<\hat{a}}$ is defined in \eqref{U} and let us define for $\lambda\in \widetilde V$ (see \eqref{I interval} for the definition of $\widetilde V$) a operator $\mathcal M_\lambda\in \mathcal L(X)$ by
	\begin{equation}\label{Glambda} 
		\mathcal{M}_\lambda\eta=\int_{0}^{\hat{a}}\beta(a, \cdot)e^{-\lambda a}\mathcal{U}(0, a)\eta\, da,\; \forall\eta\in X. 
	\end{equation}			    
	Then the following proposition holds.		    
	\begin{proposition}\label{sA}
		Under Assumption \ref{mu}, there exists $\lambda_0\in \widetilde V$ such that			    
		\begin{equation}\label{rM}
			r(\mathcal{M}_{\lambda_0})=r\left(\int_{0}^{\hat{a}}\beta(a, \cdot)e^{-\lambda_0a}\mathcal{U}(0, a)\,da\right)=1.
		\end{equation}		    
		Moreover, the operator $\mathcal{A}$ is resolvent positive and its spectral bound satisfies $s(\mathcal{A})=\lambda_0$.
	\end{proposition}
	
	\begin{proof}
		Consider the resolvent equation
		$$ 
		(0, \phi)=(\lambda I -\mathcal{A})^{-1}(\eta, \varphi),\; \forall\,(\eta, \varphi)\in\mathcal{X}, \;\lambda\in\rho(\mathcal{A}),
		$$
		following the same procedure in Proposition \ref{rGalpha}, we can obtain
		\begin{eqnarray}\label{(lambda-mathcal{A})^{-1}}
			&&[(\lambda I - \mathcal{A})^{-1}(\eta, \varphi)](a)\nonumber\\
			&& \quad =\big(0, \quad e^{-\lambda a}\mathcal{U}(0, a)(I-\mathcal{M}_\lambda)^{-1}\left[\int_{0}^{\hat{a}}\beta(s, \cdot)\int_{0}^{s}e^{-\lambda(s-\tau)}\mathcal{U}(\tau, s)\varphi(\tau)d\tau ds+\eta\right]\nonumber\\
			&&\quad\quad +\int_{0}^{a}e^{-\lambda(a-\tau)}\mathcal{U}(\tau, a)\varphi(\tau)d\tau\big). 
		\end{eqnarray}
		It follows that $\lambda\in\rho(\mathcal{A})\cap \widetilde V\Leftrightarrow \la\in \widetilde V \text{ and } 1\in\rho(\mathcal{M}_\lambda)$. Now define (recalling Assumption \ref{mu} and Remark \ref{mu2})
		$$ 
		\mathcal{C}_\lambda=\int_{0}^{\hat{a}}\underline{\beta}(a)e^{-\lambda a}\overline{\pi}(a)e^{D(K-I)a}da\in \mathcal{L}(X).
		$$ 
		Then we have $\mathcal{M}_\lambda\geq\mathcal{C}_\lambda$ in the positive operator sense.
		
		Now we claim that $r(\mathcal{M}_\lambda)$ is decreasing and log-convex (and thus continuous) with respect to the parameter $\lambda\in \widetilde V$. 
		
		\begin{claim}\label{convex}
			$r(\mathcal{M}_\lambda)$ is decreasing and log-convex with respect to $\lambda\in \widetilde V$. 
		\end{claim} 
		
		For now let us assume that the claim is true. On the other hand, from Theorem 2.2(v) in \cite{Kang2020Age} there exists a unique simple real value $\lambda_1$ such that $r(\mathcal{C}_{\lambda_1})=1$ for $\hat{a}<\infty$, and for $\hat{a}=\infty$ due to Assumption \ref{mu} and Remark \ref{mu2}. Moreover, one also has
		$$
		\lambda_1=\varpi-D\lambda^0 \;\text{ with }\int_0^{\hat{a}}\underline{\beta}(a)e^{-\varpi a}\overline{\pi}(a)da=1.
		$$
		Therefore, by the theory of positive operators (see Marek \cite{marek1970frobenius}), 
		$$ 
		r(\mathcal{M}_{\lambda_1})\geq r(\mathcal{C}_{\lambda_1})=1. 
		$$
		Moreover, one has $\lim\limits_{\lambda\rightarrow\infty}r(\mathcal{M}_\lambda)=0$. Since $\la\to r(\mathcal{M}_\lambda)$ is continuous and decreasing by Claim \ref{convex}, there exists a real $\lambda_0\ge\la_1$ such that $r(\mathcal{M}_{\lambda_0})=1$. 
		
		Next we prove that $\lambda_0$ is unique. Assume that there is $\lambda_{\varsigma}<\lambda_{\mu}$ such that $r(\mathcal{M}_{\lambda_{\varsigma}})=r(\mathcal{M}_{\lambda_{\mu}})=1$. Since $\lambda\rightarrow r(\mathcal{M}_\lambda)$ is decreasing and log-convex by Claim \ref{convex}, it follows that $r(\mathcal{M}_\lambda)=1$ for all $\lambda\geq\lambda_{\varsigma}$. This contradicts the fact that $r(\mathcal{M}_\lambda)\rightarrow0$ as $\lambda\rightarrow\infty$. Thus there is a unique $\lambda_0\in\mathbb{R}$ such that $r(\mathcal{M}_{\lambda_0})=1$. This is equivalent to the uniqueness of $\lambda_0$. Moreover, we have shown that the mapping $\lambda\rightarrow r(\mathcal{M}_\lambda)$ is either strictly decreasing on the interval $\widetilde V$ or strictly decreasing on some interval $(-D\la^0-\widetilde{\mu}, \la_2)$ with $r(\mathcal M_\la)=0$ for all $\la\ge\la_2$. In addition, since $\mathcal{M}_\lambda$ is positive, $1=r(\mathcal{M}_{\lambda_0})\in\sigma(\mathcal{M}_{\lambda_0})\neq\emptyset$, which implies that $\lambda_0\in\sigma(\mathcal{A})$, thus $\sigma(\mathcal{A})\neq\emptyset$.
		
		In addition, for any $\lambda\in\mathbb{R}$, when $\la>\la_0$ we have $r(\mathcal{M}_{\la})<r(\mathcal{M}_{\la_0})=1$, and thus $(I-\mathcal{M}_{\la})^{-1}$ exists. It follows that $\la\in\rho(\mathcal{A})$ when $\la>\la_0$, which implies that $\rho(\mathcal{A})$ contains a ray $(\la_0, \infty)$ and $(\la I -\mathcal{A})^{-1}$ is obviously a positive operator by \eqref{(lambda-mathcal{A})^{-1}} for all $\la>\la_0$. Thus $\mathcal{A}$ is a resolvent positive operator. 
		
		Finally $\la_0$ is larger than any other real spectral value in $\sigma(\mathcal{A})$. It follows that 
		$$
		\la_0=s_\mathbb{R}(\mathcal{A}):=\sup\{\lambda\in\mathbb{R}; \lambda\in\sigma(\mathcal A)\}. 
		$$
		Now we have known that $\mathcal{A}$ is a resolvent positive operator. But since $\mathcal{X}$ is a Banach space with a normal and generating cone $\mathcal{X}^+$ defined in Section \ref{Notations} and $s(\mathcal{A})\geq\la_0>-\infty$ due to $\la_0\in\sigma(\mathcal{A})$, we can conclude from Theorem \ref{sR} that $s(\mathcal{A})=s_\mathbb{R}(\mathcal{A})=\la_0$. Hence the proof is complete.
	\end{proof}
	
	Now let us prove the above claim.							
	\begin{proof}[Proof of Claim \ref{convex}]
		We use the generalized Kingman's theorem from Kato \cite{kato1982superconvexity} to show it. First claim that $\lambda\rightarrow \mathcal{M}_\lambda$ is completely monotonic. Then, $\lambda\rightarrow r(\mathcal{M}_\lambda)$ is decreasing and super-convex by Thieme \cite[Theorem 2.5]{thieme1998remarks} and hence log-convex. By the definition from Thieme \cite{thieme1998remarks}, an infinitely often differentiable function $f:(\Delta, \infty)\rightarrow Z_+$ is called \textit{completely monotonic} if 
		$$ 
		(-1)^nf^{(n)}(\lambda)\in Z_+, \;\forall \lambda>\Delta, n\in\mathbb{N}, 
		$$
		where $Z_+$ is a normal and generating cone of an ordered Banach space $Z$ and $(\Delta, \infty)$ is the domain of $f$. A family $\{F_\lambda\}_{\lambda\in\Lambda}$ of positive operators on $Z$ is called \textit{completely monotonic} if $f(\lambda)=F_\lambda x$ is completely monotonic for every $x\in Z_+$. 
		
		For our case, $\mathcal{M}_\lambda$ is indeed infinitely often differentiable with respect to $\lambda\in\widetilde V$ and 
		$$ 
		(-1)^n\mathcal{M}_\lambda^{(n)}\phi=\int_{0}^{\hat{a}}\beta(a, \cdot)a^ne^{-\lambda a}\mathcal{U}(0, a)\phi da\in X_+, \;\lambda\in\widetilde V, n\in\mathbb{N}, \phi\in X_+.
		$$	
		Thus, our result follows.
	\end{proof}
	
	\begin{remark}\label{sasb}
		{\rm (i) As Remark \ref{omegaT} shown, for all $\la>s(\mathcal{A})$ one has $r((\la I -\mathcal{A})^{-1})=(\la -s(\mathcal{A}))^{-1}$. Let $\mathcal{A}_0$ be the part of $\mathcal{A}$ in $\mathcal{X}_0$. Then it generates a positive $C_0$-semigroup $\{T_{\mathcal{A}_0}(t)\}_{t\geq0}$ on $\mathcal{X}_0$. Since $\mathcal{A}$ is also resolvent positive, we have $s(\mathcal{A})=s(\mathcal{A}_0)=\omega(T_{\mathcal{A}_0})$ when $X=L^1(\Omega)$.
			
		(ii) Note that we have $s(\mathcal{A})\geq s(\mathcal{B}_1+\mathcal{C})$ now. In fact, $\mathcal{A}$ is resolvent positive from Proposition \ref{sA} and thus Theorem \ref{rFlambda} applies to rule out the case $s(\mathcal{A})< s(\mathcal{B}_1+\mathcal{C})$. But we cannot obtain the strict relation, i.e. $s(\mathcal{A})> s(\mathcal{B}_1+\mathcal{C}),$ even if $e^{-Da}\Pi(0, a)f$ is strictly smaller than $\mathcal{U}(0, a)f$ in $L^1((0, \hat a), X)$ for any $f\in X$, because $\alpha^{**}$ and $\lambda_0$ are obtained by taking the spectral radii of the operators equal to $1$, where a limit process occurs in which the strict relation may not be preserved. However, if $r(\mathcal{G}_{\alpha})$ and $r(\mathcal{M}_\lambda)$ are eigenvalues of $\mathcal{G}_\alpha$ and $\mathcal{M}_\lambda$ respectively, we could obtain the strict relation, see Marek \cite[Theorem 4.3]{marek1970frobenius}, which is the Frobenius theory for positive operators.
		
%
       }
	\end{remark}
	
	\subsection{A Special Case:  $s(\mathcal{A})>s(\mathcal{B}_1+\mathcal{C})$}
	
	Next, we give a special case where $s(\mathcal{A})>s(\mathcal{B}_1+\mathcal{C})$ holds.
	
	\begin{proposition}\label{sAsB}
		Assume that $\mu(a, x)\equiv\mu(a)$ and $\beta(a, x)\equiv\beta(a)$, then one has $s(\mathcal{A})>s(\mathcal{B}_1+\mathcal{C})$.
	\end{proposition}
	
	\begin{proof}
		Note that when $\mu(a, x)\equiv\mu(a)$ and $\beta(a, x)\equiv\beta(a)$, 
		$s(\mathcal{B}_1+\mathcal{C})=\alpha^{**}$ and $s(\mathcal{A})=\lambda_0$ satisfies the following equations
		\begin{equation}\label{alpha**}
			\int_{0}^{\hat{a}}\beta(a)e^{-\alpha^{**}a}e^{-Da}e^{-\int_{0}^{a}\mu(s)ds}da=1
		\end{equation}  
		and 
		\begin{equation}\label{lambda0}
			r(\mathcal{M}_{\lambda_0})=r\left(\int_{0}^{\hat{a}}\beta(a)e^{-\lambda_0a}e^{-\int_{0}^{a}\mu(s)ds}e^{D(K-I)a}da\right)=1,
		\end{equation} 
		respectively. Further, in Kang et al. \cite[Theorem 2.2 (v)]{Kang2020Age} we have shown that $\mathcal{M}_{\varpi-D\lambda^0}$ has an eigenvalue associated of $1$ with an eigenfunction $\phi_0\in X_+\setminus\{0\}$ and 
		\begin{equation}\label{nu}
			r(\mathcal{M}_{\varpi-D\lambda^0})=1, 
		\end{equation}  
		where $\varpi$ is the principal eigenvalue of an age-structured operator; i.e. $\varpi$ satisfies the following characteristic equation
		\begin{equation}\label{gamma0}
			\int_{0}^{\hat{a}}\beta(a)e^{-\varpi a}e^{-\int_{0}^{a}\mu(s)ds}da=1.
		\end{equation} 
		Now comparing \eqref{alpha**} with \eqref{gamma0} and \eqref{nu} with \eqref{lambda0}, we have $\alpha^{**}=\varpi-D$ while $\lambda_0=\varpi-D\lambda^0$. Thanks to $0<\lambda^0<1$ one has that $\lambda_0>\alpha^{**}$, which implies that $s(\mathcal{A})>s(\mathcal{B}_1+\mathcal{C})$.
	\end{proof}
	
	\begin{remark}
		{\rm In the above case when $\beta$ and $\mu$ are independent of the spatial variable, age structure and nonlocal diffusion are completely decoupled, thus the spectrum becomes simpler. For instance, $\mathcal{G}_\alpha$ becomes a scalar value instead of a multiplication operator, while for $\mathcal{M}_\lambda$ by estimating the essential growth rate of nonlocal diffusion semigroup, one could obtain the spectral gap; i.e. $r_e(\mathcal{M}_\lambda)<r(\mathcal{M}_\lambda),$ and then use the generalized Krein-Rutman theorem (Edmunds et al. \cite{edmunds1972non}, Nussbaum \cite{nussbaum1981eigenvectors}, Zhang \cite{zhang2016generalized}) to show that $r(\mathcal{M}_\lambda)$ becomes an eigenvalue of $\mathcal{M}_\lambda$, the interested readers can refer to Kang et al. \cite[Theorem 2.2 (v)]{Kang2020Age}. Here $r_e(A)$ denotes the essential spectral radius of $A$.}
	\end{remark}
	
	\subsection{A Key Proposition}
	Next we give a key proposition on the solvability of some equation without diffusion, which is important in studying the effects of the diffusion rate and diffusion range on the principal eigenvalues later. Consider the problem
	\begin{equation}\label{key}
		\begin{cases}
			\partial_a u(a, x)=-(\alpha+D)u(a, x)-\mu(a, x)u(a, x),&\;(a, x)\in(0, \hat{a})\times\R^N,\\
			u(0, x)=\int_{0}^{\hat{a}}\beta(a, x)u(a, x)da,&\; x\in\R^N.
		\end{cases}
	\end{equation}
	
	\begin{proposition}\label{Galphax}
		Let Assumption \ref{mu} hold. Then there exists a continuous function $x\rightarrow\alpha(x): \R^N\rightarrow\mathbb{R}$ such that for any $x\in\R^N$, equation \eqref{key} with $\alpha=\alpha(x)$ has a positive solution $a\to u(a, x)\in W^{1, 1}(0, \hat{a})$ and 
		\begin{equation}\label{G_alpha} 
			\int_{0}^{\hat{a}}\beta(a, x)e^{-(\alpha(x)+D)a}\pi(0, a, x)da=1, \;\forall x\in\R^N. 
		\end{equation}
		Moreover, $\alpha(x)\leq\alpha^{**}$ for all $x\in\overline{\Omega}$, where $\alpha^{**}$ is defined in \eqref{max}.
	\end{proposition}
	
	\begin{proof}
		Solving the first equation of \eqref{key} explicitly, we obtain a positive solution 
		$$ 
		u(a, \cdot)=e^{-(\alpha+D)a}\pi(0, a, \cdot)u(0, \cdot) 
		$$
		provided $u(0, x)>0$. Then plugging it into the integral initial condition we get that
		$$ 
		\int_{0}^{\hat{a}}\beta(a, x)e^{-(\alpha+D)a}\pi(0, a, x)da=1. 
		$$
		Now define 
		$$
		G(\alpha, x):=\int_{0}^{\hat{a}}\beta(a, x)e^{-(\alpha+D)a}\pi(0, a, x)da. 
		$$
		Observe that $G: V\times\R^N\rightarrow(0, \infty)$ is continuously differentiable with respect to $\alpha$ and continuous with respect to $x$ due to assumptions of $\beta$ and $\mu$, respectively, where $V$ is defined in \eqref{I interval}. 
	
		Next using Assumption \ref{mu} and the subsequent Remark \ref{rkk}, for any $x\in\R^N$, one has that
		\begin{eqnarray}
			\lim\limits_{\alpha\to -\infty}G(\alpha, x)=\infty, &\; \lim\limits_{\alpha\to \infty}G(\alpha, x)=0, \;\text{ when }\hat{a}<\infty,\nonumber\\
			\lim\limits_{\alpha\to \widehat{\lambda}}G(\alpha, x)>1, &\; \lim\limits_{\alpha\to \infty}G(\alpha, x)=0, \;\text{ when }\hat{a}=\infty.\nonumber
		\end{eqnarray}
		Thus, for any $x\in\R^N$, thanks to the monotonicity of $G$ with respect to $\alpha$, there always exists a unique $\alpha(x)$ such that \eqref{G_alpha} holds. Moreover,
		\begin{eqnarray}\label{>0}
			\frac{\partial G(\alpha, x)}{\partial\alpha}=-\int_{0}^{\hat{a}}\beta(a, x)ae^{-(\alpha+D)a}\pi(0, a, x)da<0, \;\forall x\in\R^N.
		\end{eqnarray} 
		The continuity of $\alpha$ comes from implicit function theorem. In addition, one has that $\alpha(x)\leq\alpha^{**}$ for $x\in\overline\Omega$ by \eqref{max}, since $\alpha^{**}=\max_{x\in\overline{\Omega}}\alpha(x)$ due to the monotonicity of $G_\alpha$ with respect to $\alpha$. Thus the proposition is proved.
	\end{proof}
	
	\subsection{Compact Perturbation}		
	In this subsection we will show that $\mathcal{A}$ is a compact and positive perturbation of $\mathcal{B}_1+\mathcal{C}$.
	\begin{proposition}\label{compact}
		For any real $\alpha>\alpha^{**}$, $\mathcal{B}_2(\alpha I -\mathcal{B}_1-\mathcal{C})^{-1}$ is a compact operator in $\mathcal{X}$.
	\end{proposition}
	
	\begin{proof}
		We only prove the proposition in the case $X=C(\overline{\Omega})$, since $L^1((0, \hat{a}), C(\overline{\Omega}))\subset L^1((0, \hat{a}), L^1(\Omega))$. Let us choose a sequence $\{(\eta_n, \psi_n)\}_{n\in\N}\subset\mathcal{X}$ with 
$$
\norm{(\eta_n, \psi_n)}_\mathcal{X}:=\norm{\psi_n}_{L^1((0, \hat{a}), X)}+\norm{\eta_n}_X\leq1, \text{ for any $n\in\N$ }. 
$$
		By \eqref{resolventsolution1} we have for $\alpha>\alpha^{**}$ that
		\begin{eqnarray}
			&&\mathcal{B}_2(\alpha I -\mathcal{B}_1-\mathcal{C})^{-1}(\eta_n, \psi_n)=(0, \phi_n)=\left(0, DKg_{1n}+DKg_{2n}\right),\nonumber
		\end{eqnarray}
		where
		\begin{eqnarray}\label{g_12}
			g_{1n}(a) &=& e^{-(\alpha+D)a}\Pi(0, a)(1-\mathcal{G}_\alpha)^{-1}\big[\int_{0}^{\hat{a}}\beta(s, \cdot)\int_{0}^{s}e^{-(\alpha+D)(s-\tau)}\Pi(\tau, s)\psi_n(\tau)d\tau ds+\eta_n\big],\nonumber\\
			g_{2n}(a) &=& \int_{0}^{a}e^{-(\alpha+D)(a-\tau)}\Pi(\tau, a)\psi_n(\tau)d\tau. \label{B_2}
		\end{eqnarray}
		Note that $g_{1n}$ and $g_{2n}$ are continuous with respect to $a\in[0, \hat{a})$, so is $\phi_n$. We split the proof into two parts: (a) $\hat{a}<\infty$ and (b) $\hat{a}=\infty$.
		
		\textbf{(a) $\hat{a}<\infty$.} Thanks to the presence of the continuous kernel $J$, one can obtain that the functions $\{[\phi_n(a)](x)\}_{n\in\N}$ are equicontinuous with respect to $x\in\overline{\Omega}$ for any $a\in[0, \hat{a}]$. It follows by Arzela-Ascoli theorem that $\{\phi_n(a)\}_{n\in\N}$ is relatively compact in $C(\overline{\Omega})$ for any $a\in[0, \hat{a}]$. Thus, for any $0\le a_1<a_2\le\hat a$, we have that $\{\int_{a_1}^{a_2}\phi_n(a)da\}_{n\in\N}$ is relatively compact in $C(\overline\Omega)$.  
		
		Next observe that when $\alpha>\alpha^{**}$, one has
		\begin{eqnarray}\label{1-G_alpha}
			\norm{g(\eta_n, \psi_n)}_X:
			&=&\norm{(1-G_\alpha)^{-1}\big[\int_{0}^{\hat{a}}\beta(s, \cdot)\int_{0}^{s}e^{-(\alpha+D)(s-\tau)}\Pi(\tau, s)\psi_n(\tau)d\tau ds+\eta_n\big]}_X\nonumber\\
			&\le& C_\alpha\big[\int_{0}^{\hat{a}}\overline{\beta}(s)\int_{0}^{s}e^{-(\alpha+D)(s-\tau)}\norm{\Pi(\tau, s)}_{\mathcal{L}(X)}\norm{\psi_n(\tau)}_Xd\tau ds+\norm{\eta_n}_X\big]\nonumber\\
			&\le& C_\alpha\big[\norm{\overline{\beta}}_{L^\infty(0, \hat{a})}\int_{0}^{\hat{a}}\norm{\psi_n(\tau)}_Xd\tau\int_{\tau}^{\hat{a}}e^{-(\alpha+D+\widetilde{\mu})(s-\tau)}ds+\norm{\eta_n}_X\big]\nonumber\\
			&\le& C_\alpha\left[\frac{\norm{\overline{\beta}}_{L^\infty(0, \hat{a})}}{\alpha+D+\widetilde{\mu}}\norm{\psi_n}_{L^1((0, \hat{a}), X)}+\norm{\eta_n}_X\right]\nonumber\\
			&\le& C_\alpha\left[\frac{\norm{\overline{\beta}}_{L^\infty(0, \hat{a})}}{\alpha+D+\widetilde{\mu}}+1\right]=:\widetilde{C}_\alpha,
		\end{eqnarray}
		where we used the fact that $0\le(1-G_\alpha(y))^{-1}\le C_\alpha$ for all $y\in\overline{\Omega}$ with $C_\alpha>0$ being a constant, due to $\alpha>\alpha^{**}$. Here $\widetilde{C}_\alpha>0$ is another constant. It follows that
		\begin{eqnarray}
			\norm{g_{1n}(a)}_X&\le& \widetilde{C}_\alpha e^{-(\alpha+D+\widetilde{\mu})a} \text{ uniformly in $n\in N$, }\nonumber\\ 
			\norm{g_{2n}(a)}_X&\le& \int_0^ae^{-(\alpha+D+\widetilde{\mu})(a-\tau)}\norm{\psi_n(\tau)}_Xd\tau, \text{ uniformly in $n\in\N$. }\nonumber
		\end{eqnarray}
		This implies 
		\begin{eqnarray}\label{phi_n}
			\norm{\phi_n(a)}_X\le D\widetilde{C}_\alpha e^{-(\alpha+D+\widetilde{\mu})a}+D\int_0^ae^{-(\alpha+D+\widetilde{\mu})(a-\tau)}\norm{\psi_n(\tau)}_Xd\tau\text{ uniformly in $n\in\N$. }
		\end{eqnarray}
		Note that the right hand side of the above inequality is an integrable function in $L^1(0, \hat{a})$. 
		
		Next, let us show that $g_{1n}$ and $g_{2n}$ are equi-integrable respect to $a$. Observe by \eqref{B_2} that for any $n\in\N$ and $l>0$, one has
		\begin{eqnarray}
			&&|g_{2n}(a+l)-g_{2n}(a)|\nonumber\\
			&& \quad \le\int_a^{a+l} e^{-(\alpha+D)(a+l-\tau)}\pi(\tau, a+l, \cdot)\psi_n(\tau)d\tau\nonumber\\ 
			&& \qquad +\int_0^a\left[e^{-(\alpha+D)(a+l-\tau)}\pi(\tau, a+l, \cdot)-e^{-(\alpha+D)(a-\tau)}\pi(\tau, a, \cdot)\right]\psi_n(\tau)d\tau\nonumber\\
			&& \quad \le \int_a^{a+l} e^{-(\alpha+D)(a+l-\tau)}\pi(\tau, a+l, \cdot)\psi_n(\tau)d\tau+\int_0^ae^{-(\alpha+D)(a-\tau)}\pi(\tau, a+l, \cdot)\left[1-e^{-(\alpha+D)l}\right]\psi_n(\tau)d\tau\nonumber\\
			&& \qquad +\int_0^ae^{-(\alpha+D)(a-\tau)}\pi(\tau, a, \cdot)\left[1-\pi(a, a+l, \cdot)\right]\psi_n(\tau)d\tau\nonumber\\
			&& \quad \le \int_a^{a+l}e^{-(\alpha+D+\widetilde{\mu})(a+l-\tau)}\psi_n(\tau)d\tau+\int_0^ae^{-(\alpha+D+\widetilde{\mu})(a-\tau)}e^{-\widetilde{\mu}l}\left[1-e^{-(\alpha+D)l}\right]\psi_n(\tau)d\tau\nonumber\\
			&& \qquad +\int_0^ae^{-(\alpha+D+\widetilde{\mu})(a-\tau)}\left[1-e^{-\widetilde{\mu}l}\right]\psi_n(\tau)d\tau.\nonumber
		\end{eqnarray}
		It follows by setting $k=\alpha+D+\widetilde{\mu}$ that
		\begin{eqnarray}
			&&\int_0^{\hat{a}}\norm{g_{2n}(a+l)-g_{2n}(a)}_Xda\nonumber\\
			&& \quad \le\int_0^{\hat{a}}\int_a^{a+l}e^{-k(a+l-\tau)}\norm{\psi_n(\tau)}_Xd\tau da+\int_0^{\hat{a}}\int_0^ae^{-k(a-\tau)}e^{-\widetilde{\mu}l}\left[1-e^{-(\alpha+D)l}\right]\norm{\psi_n(\tau)}_Xd\tau da\nonumber\\
			&& \qquad +\int_0^{\hat{a}}\int_0^ae^{-k(a-\tau)}\left[1-e^{-\widetilde{\mu}l}\right]\norm{\psi_n(\tau)}_Xd\tau da\nonumber\\
			&& \quad :={\rm I_1+I_2+I_3}. \nonumber
		\end{eqnarray}
		Via integration by parts, one has 
		\begin{eqnarray}
			{\rm I_2}&\le& e^{-\widetilde{\mu}l}\left[1-e^{-(\alpha+D)l}\right]\int_0^{\hat{a}}\int_\tau^{\hat{a}}e^{-k(a-\tau)} da \norm{\psi_n(\tau)}_X d\tau\nonumber\\
			&\le& \frac 1 k \left[1-e^{-(\alpha+D)l}\right]\norm{\psi_n}_{L^1((0, \hat{a}), X)}\stackrel{l\to0}{\longrightarrow} 0, \text{ uniformly in $n\in\N$.}\nonumber
		\end{eqnarray}
		Similarly, one also obtains ${\rm I_3}\to0$ as $l\to0$ uniformly in $n\in \N$. 
		
		Next let us deal with ${\rm I_1}$. To this aim, we split it into two cases: $0\le a\le\tau\le a+l\le\hat{a}$ and $0\le a\le\tau\le \hat{a}\le a+l$.
		
		\textbf{Case $0\le a\le\tau\le a+l\le\hat{a}$.} Via integration by parts, one has
		\begin{eqnarray}
			{\rm I_1}\le \int_0^{\hat{a}}\int_{\tau-l}^\tau e^{-k(a+l-\tau)}da \norm{\psi_n(\tau)}_Xd\tau\stackrel{l\to0}{\longrightarrow} 0 \text{ uniformly in $n\in\N$.}\nonumber
		\end{eqnarray}
		
		\textbf{Case $0\le a\le\tau\le\hat{a}\le a+l$.} Via integration by parts, one has
		\begin{eqnarray}
			{\rm I_1}&\le& \int_0^{a+l}\int_a^{a+l}e^{-k(a+l-\tau)}\norm{\psi_n(\tau)}_Xd\tau da\nonumber\\
			&\le&\int_0^{a+l}\int_{\tau-l}^\tau e^{-k(a+l-\tau)}da \norm{\psi_n(\tau)}_Xd\tau\stackrel{l\to0}{\longrightarrow} 0 \text{ uniformly in $n\in\N$.}\nonumber
		\end{eqnarray}
		In summary, we have shown that $\int_0^{\hat{a}}\norm{g_{2n}(a+l)-g_{2n}(a)}_Xda\to0$ as $l\to0$ uniformly in $n\in\N$. Similarly, one can show by \eqref{1-G_alpha} that $\int_0^{\hat{a}}\norm{g_{1n}(a+l)-g_{1n}(a)}_Xda\to0$ as $l\to0$ uniformly in $n\in\N$. It follows that $\int_0^{\hat{a}}\norm{\phi_{n}(a+l)-\phi_{n}(a)}_Xda\to0$ as $l\to0$ uniformly in $n\in\N$. Combining with \eqref{phi_n}, Simon's compactness theorem in $L^1((0, \hat a), X)$ (see Simon \cite[Theorem 1]{simon1986compact}) ensures that the sequence $\{a\to\phi_n(a)\}_{n\in\N}$ is relatively compact in $L^1((0, \hat{a}), X)$. Hence there exists a limit function $\phi\in L^1((0, \hat a), X)$ such that, up to a subsequence, $\phi_n\to\phi$ in $L^1((0, \hat{a}), X)$ and the linear operator $\mathcal{B}_2(\alpha I-\mathcal{B}_1-\mathcal{C})^{-1}$ is compact on $\mathcal{X}$.
		
		\textbf{(b) $\hat{a}=\infty$.}
		Define the characteristic function $\chi_{[0, n]}, n\in\N$, which is defined as follows,
		$$
		\chi_{[0, n]}(a)=1, \text{ if }a\in[0, n] \text{ and } \chi_{[0, n]}=0, \text{ otherwise}.
		$$ 
		Define $g: \mathcal{X}\to X$ as follows,
		\begin{equation*}
			[g(\eta, \psi)](y):=(1-G_\alpha(y))^{-1}\int_{0}^{\infty}\beta(s, y)\int_{0}^{s}e^{-(\alpha+D)(s-\tau)}\pi(\tau, s, y)\psi(\tau, y)d\tau ds+\eta(y),\;\forall(\eta, \psi)\in \mathcal{X}.
		\end{equation*}
		Note that for any $\alpha>\alpha^{**}$ one has $(1-G_\alpha(y))^{-1}\le C_\alpha$ for all $y\in\overline{\Omega}$, where $C_\alpha>0$ is a constant. 
		It follows by \eqref{1-G_alpha} and \eqref{norm-K} that for any $(\eta, \psi)\in\mathcal{X}$,
		\begin{eqnarray}
			&&\int_0^\infty D\norm{K}_{\mathcal{L}(X)}e^{-(\alpha+D)a}\norm{\Pi(0, a)g(\eta, \psi)}_X\left[\chi_{[0, n]}(a)-1\right]da\nonumber\\
			&& \quad \le D\norm{g(\eta, \psi)}_X\int_0^\infty e^{-(\alpha+D+\widetilde{\mu})a}|\chi_{[0, n]}(a)-1|da\to0, \text{ as }n\to\infty,\nonumber\\
			\text{and}\nonumber\\
			&& \int_0^\infty D\norm{K}_{\mathcal{L}(X)}\int_{0}^{a}e^{-(\alpha+D)(a-\tau)}\norm{\Pi(\tau, a)}_{\mathcal{L}(X)}\norm{\psi(\tau)}_X d\tau \left[\chi_{[0, n]}(a)-1\right]da\nonumber\\
			&& \quad \le D\int_0^\infty\norm{\psi(\tau)}_X d\tau\int_\tau^\infty e^{-(\alpha+D+\widetilde{\mu})(a-\tau)}|\chi_{[0, n]}(a)-1| da\nonumber\\
			&& \quad \le D\norm{\psi}_{L^1((0, \infty), X)}\int_0^\infty e^{-(\alpha+D+\widetilde{\mu})a}|\chi_{[0, n]}(a)-1|da\to0 \text{ as }n\to\infty.\nonumber
		\end{eqnarray}
		It follows that 
		$$
		\chi_{[0, n]}\mathcal{B}_2(\alpha I -\mathcal{B}_1-\mathcal{C})^{-1}\to \mathcal{B}_2(\alpha I -\mathcal{B}_1-\mathcal{C})^{-1} \text{ as }n\to\infty,\text{ in }\mathcal{L}(\mathcal{X}, \mathcal{X}_0).
		$$
		Since we have known that $\chi_{[0, n]}\mathcal{B}_2(\alpha I -\mathcal{B}_1-\mathcal{C})^{-1}$ is compact from the first paragraph, it implies that $\mathcal{B}_2(\alpha I -\mathcal{B}_1-\mathcal{C})^{-1}$ is also compact. Thus the proof is complete.
	\end{proof}
	
	\begin{corollary}\label{compactperturbator}
		The operator $\mathcal{B}_2$ is a compact perturbator of $\mathcal{B}_1+\mathcal{C}$ and thus the operator $\mathcal{A}=\mathcal{B}_1+\mathcal{B}_2+\mathcal{C}$ is a compact perturbation of $\mathcal{B}_1+\mathcal{C}$.
	\end{corollary}
	
	\begin{proof}
		$(\alpha I -\mathcal{B}_1-\mathcal{C})^{-1}\mathcal{B}_2(\alpha I -\mathcal{B}_1-\mathcal{C})^{-1}$ is compact for any $\alpha>s(\mathcal{B}_1+\mathcal{C})$ since $\mathcal{B}_2(\alpha I -\mathcal{B}_1-\mathcal{C})^{-1}$ is compact by Proposition \ref{compact}.
	\end{proof}
	
	\section{Principal Spectral Theory}\label{Principal Spectral Theory}
	In this section we state and prove the main results on the existence/nonexistence of principal eigenvalues. Recall that Assumption \ref{mu} is used to guarantee the existence of $s(\mathcal{A})$ and $s(\mathcal{B}_1+\mathcal{C})$ when $\hat{a}=\infty$ in the previous section. Now we assume that $s(\mathcal{A})$ and $s(\mathcal{B}_1+\mathcal{C})$ exist for $\hat a\le\infty$ throughout this section.
	
	\subsection{Principal Eigenvalue}
	We first provide a sufficient condition to make the spectral bound $s(\mathcal{A})$ become the principal eigenvalue. Here we say that $\la\in\sigma(T)\cap\R$ is the principal eigenvalue of a linear operator $T$, if it is larger than the real part of other eigenvalues of $T$ and associated with a positive eigenfunction.
	
	\begin{theorem}\label{principle}
		Assume that $s(\mathcal{A})>s(\mathcal{B}_1+\mathcal{C})$, then $s(\mathcal{A})$ is the principal eigenvalue of $\mathcal{A}$. 
	\end{theorem}
	
	\begin{proof}
		Denote 
		\begin{equation}\label{Flambdainverse}
			\mathcal{F}_\lambda=\mathcal{B}_2(\lambda I -\mathcal{B}_1-\mathcal{C})^{-1}, \;\lambda>\alpha^{**}.
		\end{equation} 
		Note that $\mathcal{A}=\mathcal{B}_1+\mathcal{C}+\mathcal{B}_2$ is a compact perturbation of $\mathcal{B}_1+\mathcal{C}$ by Corollary \ref{compactperturbator}. We will use Theorem \ref{compactperturbation} in Appendix to prove the conclusion. First, we know that $\mathcal{A}$ is resolvent positive by Proposition \ref{sA}. It follows that case (i) in Theorem \ref{rFlambda} in Appendix will be ruled out. Secondly, by assumption that $s(\mathcal{A})>s(\mathcal{B}_1+\mathcal{C})$ we know that only case (iii) in Theorem \ref{rFlambda} will happen, otherwise $s(\mathcal{A})=s(\mathcal{B}_1+\mathcal{C})$ which is a contradiction, if case (ii) in Theorem \ref{rFlambda} would happen. Hence, there exists $\lambda_2>\lambda_1>s(\mathcal{B}_1+\mathcal{C})$ such that $r(\mathcal{F}_{\lambda_1})\geq1>r(\mathcal{F}_{\lambda_2})$. Now the hypothesis in Theorem \ref{compactperturbation} holds, so $s(\mathcal{A})$ is an eigenvalue of $\mathcal{A}$ with a positive eigenfunction and has finite algebraic multiplicity and is a pole of the resolvent of $\mathcal{A}$. It follows that $s(\mathcal{A})$ is the principal eigenvalue of $\mathcal{A}$.	
	\end{proof}
	
	Combining the above theorem with Proposition \ref{sAsB}, one can immediately obtain the following conclusion.
	\begin{corollary}
		Assume that $\mu(a, x)\equiv\mu(a)$ and $\beta(a, x)\equiv\beta(a)$, then $s(\mathcal{A})$ is the principal eigenvalue of $\mathcal{A}$. 
	\end{corollary}				
	
	Next, we give a sufficient and necessary condition to ensure that $s(\mathcal{A})>s(\mathcal{B}_1+\mathcal{C})$.
	\begin{corollary}\label{corollary}
		The inequality $s(\mathcal{A})>s(\mathcal{B}_1+\mathcal{C})$ holds if and only if there is $\lambda^*>s(\mathcal{B}_1+\mathcal{C})$ such that $r(\mathcal{F}_{\lambda^*})\geq1$, where $\mathcal{F}_\lambda$ is defined in \eqref{Flambdainverse}.
	\end{corollary}
	
	\begin{proof}
		If there exists $\lambda^*>s(\mathcal{B}_1+\mathcal{C})$ such that $r(\mathcal{F}_{\lambda^*})\geq1$, then case (iii) in Theorem \ref{rFlambda} will happen which implies that $s(\mathcal{A})>s(\mathcal{B}_1+\mathcal{C}),$ because we can always find $\vartheta$ large enough such that $r(\mathcal{F}_\vartheta)<1$ regarding to \eqref{B_2}. Conversely, if $s(\mathcal{A})>s(\mathcal{B}_1+\mathcal{C})$, by the same argument as in Theorem \ref{principle}, we have the desired result.
	\end{proof}
	
	Note that Theorem \ref{principle} is valid for both $X=L^1(\Omega)$ and $X=C(\overline{\Omega})$, as long as $s(\mathcal{A})>s(\mathcal{B}_1+\mathcal{C})$. Next we will show that $s(\mathcal{A})$ is also algebraically simple under the additional assumption on $\beta$. Once it is true, the eigenfunctions in $X=L^1(\Omega)$ and $X=C(\overline{\Omega})$ respectively associated with $s(\mathcal{A})$ are the same, due to the fact that $C(\overline{\Omega})\subset L^1(\Omega)$.
	\begin{assumption}\label{irreducible}
		There exists no $a_0$ such that $\underline{\beta}(a)=0$ a.e. $[a_0, \hat{a})$.
	\end{assumption}
	
	\begin{remark}\label{assump}
		{\rm Before proceeding, let us make some comments on Assumption \ref{irreducible}. It is motivated by Engel and Nagel \cite[Theorem 4.4]{engel2006short} to show that the semigroup generated by the age-structured operator is irreducible. In our situation, we will prove a similar property, which is called conditionally strictly positive (see Definition \ref{CSP} in Appendix), under this assumption. Note that this assumption is equivalent to 
			\begin{eqnarray}\label{s>0}
				\int_s^{\hat a}e^{-\sigma}\underline{\beta}(\sigma)d\sigma>0, \;\forall s\in [0, \hat a).
			\end{eqnarray}
			}
	\end{remark}
	
	\begin{theorem}\label{sim}
		Let Assumption \ref{irreducible} hold and assume that $s(\mathcal{A})>s(\mathcal{B}_1+\mathcal{C})$, then the principal eigenvalue of $\mathcal{A}$, i.e. $s(\mathcal{A})$ is algebraically simple. 
	\end{theorem}
	\begin{proof}
		We will show that all positive nonzero fixed points of $\mathcal{F}_\lambda$ are conditionally strictly positive (see Definition \ref{CSP} in Appendix), and then employ Theorem \ref{compactperturbation} again to conclude the result.  
		
		First observe that $\mathcal{F}_\lambda$ maps $\mathcal{X}$ into $\mathcal{X}_0$, then we introduce the restriction of $\mathcal{F}_\lambda$ to $\mathcal{X}_0$ and the associated operator $L_\lambda, \lambda>\al^{**}$ in $L^1((0, \hat{a}), X)$, see \eqref{B_2}, 
		\begin{eqnarray}\label{Llambda}
			[L_\lambda \psi](a, x)&=&D\int_{\Omega}J(x-y)e^{-(\lambda+D)a}\pi(0, a, y)[(1-G_\lambda)^{-1}\tilde{g}\psi](y)dy\nonumber\\
			&&+D\int_{\Omega}J(x-y)\int_{0}^{a}e^{-(\lambda+D)(a-\gamma)}\pi(\gamma, a, y)\psi(\gamma, y)d\gamma dy,
		\end{eqnarray}
		where $\tilde{g}: L^1((0, \hat{a}), X)\to X$ is given by
		$$
		[\tilde{g}\psi](y):=\int_{0}^{\hat{a}}\beta(s, y)\int_{0}^{s}e^{-(\lambda+D)(s-\gamma)}\pi(\gamma, s, y)\psi(\gamma, y)d\gamma ds.
		$$
		Now we have $\mathcal F_\la(0, \psi)=(0, L_\la\psi)$. Observe that $(a, x)\to [L_\lambda\psi](a, x)$ is continuous. Thus $L_\lambda$ is strictly positive in the sense that for $\psi\in L^1_+((0, \hat{a}), X)$ being a fixed point of $L_\lambda$, if there exists some point $(a_0, x_0)\in[0, \hat{a})\times\overline{\Omega}$ such that $[L_\lambda\psi](a_0, x_0)=0$, then $\psi\equiv0$ in $[0, \hat{a})\times\overline{\Omega}$. 
		
		In fact, $[L_\lambda\psi](a_0, x_0)=0$ implies that 
		$$
		D\int_{\Omega}J(x_0-y)e^{-(\lambda+D)a_0}\pi(0, a_0, y)[(1-G_\lambda)^{-1}\tilde{g}\psi](y)dy=0,
		$$
		which follows by the positivity of $\int_{\Omega}J(x_0-y)dy$ and $(1-G_\lambda)^{-1}, \lambda>\alpha^{**}$, along with exponential functions that
		\begin{equation}\label{bet}
			\int_0^{\hat{a}}\beta(s, y)\int_{0}^{s}\psi(\gamma, y)d\gamma ds =0 \text{ for all }y\in B(x_0, r).
		\end{equation}
		Now denote 
		$$
		H(s, y):=\int_s^{\hat{a}}e^{-\sigma}\beta(\sigma, y)d\sigma.
		$$ 
		Note that $H$ is well-defined. Then \eqref{bet} can be transformed by using integration by parts into 
		\begin{eqnarray}
			0=\int_0^{\hat{a}}e^{-s}\beta(s, y)\int_{0}^{s}\psi(\gamma, y)d\gamma ds &=&-H(s, y)\int_0^s\psi(\gamma, y)d\gamma\big|_{s=0}^{s=\hat{a}}+\int_0^{\hat{a}}H(s, y)\psi(s, y)ds\nonumber\\
			&=&\int_0^{\hat{a}}H(s, y)\psi(s, y)ds, \text{ for all }y \in B(x_0, r).\nonumber
		\end{eqnarray}
		But by Assumption \ref{irreducible} and Remark \ref{assump}, one has $H(s, y)\ge\int_s^{\hat{a}}e^{-\sigma}\underline{\beta}(\sigma)d\sigma>0$ for all $(s, y)\in[0, \hat{a})\times\overline{\Omega}$. This will give us $\psi\equiv0$ in $[0, \hat{a})\times B(x_0, r)$. 
		
		Next, noticing that $\psi$ is a fixed point of $L_\lambda$ and considering the second term of \eqref{Llambda}, we iterate $L_\lambda$ for $n-$times to obtain
		\begin{eqnarray}
			0&=&[L_\lambda\psi](a_0, x_0)=[L_\lambda^n\psi](a_0, x_0)\nonumber\\
			&\ge& \!\!\!\! D^n \!\! \int_{\Omega}\!\!\cdots\!\!\int_{\Omega}\prod_{m=1}^n \left[J(x_{m-1}\!-\! x_m)\int_0^{a_{m-1}}\!\! e^{-(\lambda+D)(a_{m-1}-a_m)}\pi(a_m, a_{m-1}, x_m)da_m\right]\psi(a_n, x_n)dx_n\cdots dx_1.\nonumber
		\end{eqnarray}
		It follows that $\psi(\cdot, x)\equiv0$ in $B(x_0, nr)\cap\overline{\Omega}$. Now when $n$ is sufficiently large, $B(x_0, nr)\cap\overline{\Omega}$ will cover $\overline{\Omega}$, then $\psi\equiv0$ in $[0, \hat{a})\times\overline{\Omega}$. Thus $L_\lambda$ is strictly positive.
		
		Now for any positive nonzero fixed point of $L_\lambda$, denoted by $\psi\in L^1_+((0, \hat{a}), X)$, and any $\psi^*\in L^\infty_+((0, \hat{a}), X^*)$ with $L_\lambda^*\psi^*\neq0$, where $X^*$ denotes the dual space of $X$, one has
		$$
		\langle \psi, \psi^*\rangle=\langle L_\lambda\psi, \psi^*\rangle>0.
		$$
		Here $\langle \;,\; \rangle$ denotes the duality paring between $L^1((0, \hat a), X)$ and $L^\infty((0, \hat a), X^*)$. It follows that all positive nonzero fixed points of $L_\lambda$ are conditionally strictly positive and so is $\mathcal{F}_\lambda$.
	\end{proof}
	
	\subsection{Criteria}			
	Since the condition $s(\mathcal{A})>s(\mathcal{B}_1+\mathcal{C})$ turns out to be hard to check, we now provide relatively easily verifiable and general sufficient conditions ensuring that $s(\mathcal{A})$ is the principal eigenvalue of $\mathcal{A}$ for the sake of applications. This leads to our main theorems on the existence of principal eigenvalues of $\mathcal{A}$ in this section.
	
	\subsubsection{Criterion I}
	Before giving the first criterion, we first provide a lower bound for $(\alpha I -\mathcal{B}_1-\mathcal{C})^{-1}$ with $\alpha>\alpha^{**}$.
	
	\begin{proposition}\label{lower}
		Assume that $\mu_{\max}:=\sup_{a\in(0, \hat a)}\overline{\mu}(a)<\infty$. For any $\al>\alpha^{**}$, any $\psi\in X_+$ and any $\theta>0$, the inverse $(\alpha I -\mathcal{B}_1-\mathcal{C})^{-1}: \mathcal{X}\rightarrow\mathcal{X}_0$ satisfies the following estimate, 
		$$ 
		\left[(\alpha I -\mathcal{B}_1-\mathcal{C})^{-1}(0, e^{-\theta\cdot}\psi)\right](a, x)\geq\frac{M(\alpha, D, \theta)}{1-G_\alpha(x)}\left(0, e^{-(\alpha+D+\mu_{\max})a}\psi(x)\right) \; \text{ a.e.} \; (a, x)\in(0, \hat{a})\times\overline{\Omega}, 
		$$
where $M(\alpha, D, \theta)>0$ is a constant that will be determined in the proof.
\end{proposition}
	
	\begin{proof}
		First note that $e^{-\theta\cdot}\psi\in L^1_+((0, \hat{a}), X)$ for any $\psi\in X_+$ and then define 
		\begin{eqnarray}
			&&I_1(\alpha, D, \theta, x):=\int_{0}^{\hat{a}}\beta(a, x)\int_{0}^{a}e^{-(\alpha+D)(a-\gamma)}\pi(\gamma, a, x)e^{-\theta \gamma}d\gamma da,\nonumber\\
			&&I_2(\alpha, D, a, x)=e^{-(\alpha+D)a}\pi(0, a, x).\nonumber
		\end{eqnarray} 
Also notice that $\min_{x\in\overline{\Omega}}I_1(\alpha, D, \theta, x)<\infty$ for $\hat{a}\le\infty$ due to $\alpha>\alpha^{**}>-D-\widetilde{\mu}$. Next observe that
		$$
		I_2(\alpha, D, a, x)\ge e^{-(\alpha+D+\mu_{\max})a} \text{ and }  \min_{x\in\overline{\Omega}}I_2(\alpha, D, \cdot, x)\in L^1(0, \hat{a})
		$$ 
		for $\hat{a}\le\infty$ again due to $\alpha>\alpha^{**}>-D-\widetilde{\mu}$. It follows from \eqref{g_12} that
		$$ 
		\left[(\alpha I -\mathcal{B}_1-\mathcal{C})^{-1}(0, e^{-\theta\cdot}\psi)\right](a, x)\geq\big(0,\; \frac{I_1(\alpha, D, \theta, x)}{1-G_\alpha(x)}e^{-(\alpha+D+\mu_{\max})a}\psi(x)\big) \text{ in }\mathcal X_0
		$$
		for any $\theta>0$ and $\psi\in X_+$. Thus $M(\alpha, D, \theta)$ is given by
		$$ 
		M(\alpha, D, \theta):=\min_{\overline{\Omega}}I_1(\alpha, D, \theta, \cdot)\ge \int_0^{\hat a}\underline{\beta}(a)\int_0^ae^{-(\al+D)(a-\ga)-\int_\ga^a\overline{\mu}(s)ds-\theta\ga}d\ga da. 
		$$
		Then the result follows.
	\end{proof}
	
	\begin{theorem}[Existence of principal eigenvalues - I]\label{I}
		Assume that $\mu_{\max}:=\sup_{a\in(0, \hat a)}\overline{\mu}(a)<\infty$ and
		\begin{eqnarray}\label{criterion}
			x\to\frac{1}{1-G_{\alpha^{**}}(x)}\notin L^1_{loc}(\overline{\Omega}),
		\end{eqnarray}
		then $s(\mathcal{A})$ is the principal eigenvalue of $\mathcal{A}$, where $G_\alpha(x)$ is defined in \eqref{Galpha(x)}.
	\end{theorem}
	
	\begin{proof}
		The idea of the proof below traced back to Coville \cite{coville2010simple} (see also Shen and Vo \cite{shen2019nonlocal}). For the completeness and reader's convenience, we include the necessary modifications and provide a detailed proof.
		
		In the following proof, we only show the result in the case $X=C(\overline{\Omega})$ and put the case $X=L^1(\Omega)$ in the Appendix. By contradiction, assume that $s(\mathcal{A})$ is not the principal eigenvalue of $\mathcal{A}$, by the contrapositive statement of Theorem \ref{principle} and Remark \ref{sasb}, we have  $s(\mathcal{A})=s(\mathcal{B}_1+\mathcal{C})$. It follows from Corollary \ref{corollary} that
		\begin{eqnarray}\label{r<1}
			r(\mathcal{F}_\alpha)=r(\mathcal{B}_2(\alpha I -\mathcal{B}_1-\mathcal{C})^{-1})<1, \;\forall\,\alpha>s(\mathcal{B}_1+\mathcal{C})=\al^{**}.
		\end{eqnarray} 		
		Now we choose $\theta=\alpha+D+\mu_{\max}$. Then Proposition \ref{lower} implies that 
		$$ 
		[(\alpha I -\mathcal{B}_1-\mathcal{C})^{-1}(0, e^{-\theta \cdot})](a, x)\geq\left(0,\; \frac{M(\alpha, D, \theta)}{1-G_\alpha(x)}e^{-\theta a}\right)\ge(0, 0), \;(a, x)\in(0, \hat{a})\times\overline{\Omega}. 
		$$
		Applying $\mathcal{B}_2$ to both sides of the above estimates, we find for $(a, x)\in(0, \hat a)\times\overline\Omega$ that
		\begin{eqnarray}\label{1iteration}
			&&[\mathcal{B}_2(\alpha I -\mathcal{B}_1-\mathcal{C})^{-1}(0, e^{-\theta \cdot})](a, x)\nonumber\\
			&& \quad = D\int_{\Omega}J(x-y)[(\alpha I -\mathcal{B}_1-\mathcal{C})^{-1}(0, e^{-\theta\cdot})](a, y)dy\nonumber,\\
			&& \quad \geq \left(0,\; e^{-\theta a}\int_{\Omega}J(x-y)\frac{DM(\alpha, D, \theta)}{1-G_\alpha(y)}dy\right).
		\end{eqnarray} 
		By \eqref{1iteration} and Proposition \ref{lower}, we find for $(a, x)\in(0, \hat{a})\times\overline{\Omega}$ that
		\begin{eqnarray}
			&&\left((\alpha I -\mathcal{B}_1-\mathcal{C})^{-1}\mathcal{B}_2(\alpha I -\mathcal{B}_1-\mathcal{C})^{-1}(0, e^{-\theta\cdot})\right)(a, x)\\
			&&\quad \ge\left((\alpha I -\mathcal{B}_1-\mathcal{C})^{-1}\left(0,\; e^{-\theta \cdot}\int_{\Omega}J(\cdot-y)\frac{DM(\alpha, D, \theta)}{1-G_\alpha(y)}dy\right)\right)(a, x)\nonumber,\\
			&&\quad \geq\left(0, \; e^{-\theta a}\frac{M(\alpha, D, \theta)}{1-G_\alpha(x)}\int_{\Omega}J(x-y)\frac{DM(\alpha, D, \theta)}{1-G_\alpha(y)}dy \right).\nonumber
		\end{eqnarray}
		Applying $\mathcal{B}_2$ to both sides of the above estimates, we have
		\begin{eqnarray}
			&&\left((\mathcal{B}_2(\alpha I -\mathcal{B}_1-\mathcal{C})^{-1})^2(0, e^{-\theta \cdot})\right)(a, x) \nonumber \\
			&&\quad \geq\left(0, \; e^{-\theta a}\int_{\Omega}J(x-y)\frac{DM(\alpha, D, \theta)}{1-G_\alpha(y)}\int_{\Omega}J(y-z)\frac{DM(\alpha, D, \theta)}{1-G_\alpha(z)}dzdy\right).
		\end{eqnarray}
		Repeating the above arguments, we find for $(a, x_0)\in(0, \hat a)\times\overline{\Omega}$ and $n\ge1$ the following estimate
		$$ 
		\left(\!(\mathcal{B}_2(\alpha I \!-\!\mathcal{B}_1-\mathcal{C})^{-1})^n(0, e^{-\theta \cdot})\!\right)\!(a, x_0)\!\geq\! \left(\!0, \; e^{-\theta a}\int_{\Omega}\!\!\cdots\!\!\int_{\Omega}\!\prod_{m=1}^n\left[J(x_{m-1}\!-\! x_m)\frac{DM(\alpha, D, \theta)}{1-G_\alpha(x_m)}\right]dx_n \! \cdots  dx_1\!\right). 
		$$
		As a result, 
		\begin{eqnarray}\label{com}
			\norm{(\mathcal{B}_2(\alpha I -\mathcal{B}_1-\mathcal{C})^{-1})^n}\ge  \max_{x_0\in\overline{\Omega}}\int_{\Omega}\cdots\int_{\Omega}\prod_{m=1}^{n}\left[J(x_{m-1}-x_m)\frac{DM(\alpha, D, \theta)}{1-G_\alpha(x_m)}\right]dx_n\cdots dx_1,
		\end{eqnarray}
		which implies that for any $x_0\in\overline{\Omega}$ and $\delta>0$, 
		\begin{eqnarray}
			&&\norm{(\mathcal{B}_2(\alpha I -\mathcal{B}_1-\mathcal{C})^{-1})^n}\nonumber\\
			&&\quad \geq\int_{\Omega\cap B(x_0, \delta)}\cdots\int_{\Omega\cap B(x_0, \delta)}\prod_{m=1}^{n}\left[J(x_{m-1}-x_m)\frac{DM(\alpha, D, \theta)}{1-G_\alpha(x_m)}\right]dx_n\cdots dx_1\nonumber\\
			&&\quad \geq\left[\inf_{x\in\Omega\cap B(x_0, \delta)}\int_{\Omega\cap B(x_0, \delta)}J(x-y)\frac{DM(\alpha, D, \theta)}{1-G_\alpha(y)}dy\right]^n,
		\end{eqnarray}
		where $B(x_0, \delta)$ is the open ball in $\mathbb{R}^N$ centered at $x_0$ with radius $\delta$. We can use \eqref{r<1} and Gelfand's formula for the spectral radius of a bounded linear operator to find that 
		\begin{eqnarray}\label{1}
			1\geq\inf_{x\in\Omega\cap B(x_0, \delta)}\int_{\Omega\cap B(x_0, \delta)}J(x-y)\frac{DM(\alpha, D, \theta)}{1-G_\alpha(y)}dy:=I(x_0, \delta, \alpha, D)
		\end{eqnarray}
		for all $x_0\in\overline{\Omega}$ and $\delta>0$.
		
		Since $J$ is continuous and $J(0)>0$, there exist $r>0$ and $c_0>0$ such that $J\geq c_0$ on $B(0, r)$, the open ball in $\mathbb{R}^N$ centered at $0$ with radius $r$. Hence,
		\begin{eqnarray}
			I(x_0, \delta, \alpha, D) &\geq& \inf_{x\in\Omega\cap B(x_0, \delta)}\int_{\Omega\cap B(x_0, \delta)\cap B(x, r)}J(x-y)\frac{DM(\alpha, D, \theta)}{1-G_\alpha(y)}dy\nonumber\\
			&\geq& c_0\inf_{x\in\Omega\cap B(x_0, \delta)}\int_{\Omega\cap B(x_0, \delta)\cap B(x, r)}\frac{DM(\alpha, D, \theta)}{1-G_\alpha(y)}dy\nonumber\\
			&=& c_0\int_{\Omega\cap B(x_0, \delta)}\frac{DM(\alpha, D, \theta)}{1-G_\alpha(y)}dy
		\end{eqnarray}
		provided $2\delta\leq r$ so that $B(x_0, \delta)\subset B(x, r)$ whenever $x\in\overline{B(x_0, \delta)}$. In particular, for any $x_0\in\overline{\Omega}$, 
		$$ 
		I(x_0, r/2, \alpha, D)\geq c_0\int_{\Omega\cap B(x_0, r/2)}\frac{DM(\alpha, D, \theta)}{1-G_\alpha(y)}dy. 
		$$
		Since $\frac{1}{1-G_{\alpha^{**}}}\notin L^1_{loc}(\overline{\Omega})$, there exists $x_*\in\overline{\Omega}$ such that
		$$ \frac{1}{1-G_{\alpha^{**}}}\notin L^1(\overline{\Omega}\cap B(x_*, r/2)), $$
		which implies the existence of some $\epsilon>0$ small enough such that
		$$ 
		c_0\int_{\overline{\Omega}\cap B(x_*, r/2)}\frac{DM(\alpha, D, \theta)}{1-G_{\alpha^{**}+\epsilon}(y)}dy\geq 2. 
		$$
		In particular, $I(x_*, r/2, \alpha^{**}+\epsilon, D)\geq 2$, which contradicts \eqref{1}.
	\end{proof}
	
\subsubsection{Criterion II}
In this subsection, we give the second non-locally-integrable condition similar as in \eqref{criterion} to check the existence of the principal eigenvalue of $\mathcal{A}$. 	
Before proceeding, we first provide an assumption on $\beta$ to make sure that the principal eigenfunction $\phi$ can attain its positive maximum and minimum in $[0, a_2]\times\overline{\Omega}$ for some $a_2\in(0, \hat{a})$.

\begin{assumption}\label{beta}
There exists $a_2\in(0, \hat a)$ such that $\overline\beta(a)=0$ for $a\in[a_2, \hat a)$ or equivalently $\beta(a, x)=0$ for $(a, x)\in [a_2, \hat{a})\times\overline{\Omega}$.
	\end{assumption}
	
We would like to mention that the above assumption is somehow reasonable for applications. It means that the birth rate becomes zero when the age of the individuals approaches the maximal age $\hat a$. 
	
		
		Now, let us rewrite the function space $\mathcal{X}$ as follows:
		$$
		\mathcal{X}=X\times L^1((0, \hat{a}), X)=X\times \left(L^1((0, a_2), X)\times L^1((a_2, \hat{a}), X)\right)
		$$
		with a function $\psi\in L^1((0, \hat{a}), X)$ identified to $(\psi|_{(0, a_2)}, \psi|_{(a_2, \hat{a})})\in L^1((0, a_2), X)\times L^1((a_2, \hat{a}), X)$. Define the operator $\mathcal{\widehat{B}}$ in $X\times L^1((0, a_2), X)$ by
		$$
		\mathcal{\widehat{B}}(0, \psi)=\left(-\psi(0), \; -\partial_a\psi+D[K-I]\psi(a)-\mu(a, \cdot)\psi(a)\right),\; \text{ for }(0, \psi)\in \{0\}\times W^{1, 1}((0, a_2), X).
		$$
		Note that $\mathcal{\widehat{B}}$ is a closed operator under Assumption \ref{beta}. Moreover, define the bounded operator $\mathcal{\widehat{C}}$ as follows:
		$$
		\mathcal{\widehat{C}}(0, h)=\left(\int_{0}^{a_2}\beta(a, \cdot)h(a)da, \;  0\right), \text{ for }(0, h)\in \{0\}\times L^1((0, a_2), X),
		$$
		so that $\mathcal{\widehat{C}}\in \mathcal{L}\left(\{0\}\times L^1((0, a_2), X), X\times\{0_{L^1}\}\right)$. Define the operator $\mathcal{\widehat{A}}$ by $\mathcal{\widehat{A}}:=\mathcal{\widehat{B}}+\mathcal{\widehat{C}}$ with ${\rm dom}(\mathcal{\widehat{A}})=\{0\}\times W^{1, 1}((0, a_2), X)$. 
		
		Next recalling that $\widetilde V$ is the interval defined in \eqref{I interval}, let us show $\sig(\mathcal{\widehat{A}})\cap \widetilde V=\sig(\mathcal{A})\cap \widetilde V$. To do so, it suffices to show $\rho(\mathcal{\widehat{A}})\cap \widetilde V=\rho(\mathcal{A})\cap \widetilde V$. Recalling the argument in Proposition \ref{sA}, it says that 
		$$
		\la\in\rho(\mathcal{A})\cap \widetilde V\Leftrightarrow \la\in \widetilde V \text{ and }1\in\rho(\mathcal{M}_\la).
		$$
		Similarly, Proposition \ref{sA} with $\hat{a}=a_2$ applied to $\mathcal{\widehat{A}}$ gives that
		$$
		\la\in\rho(\mathcal{\widehat{A}})\cap \widetilde V\Leftrightarrow \la\in \widetilde V \text{ and }1\in\rho(\mathcal{\widehat{M}}_\la),
		$$
		where $\mathcal{\widehat{M}}_\la\in\mathcal{L}(X)$ is defined for $\la\in \widetilde V$ by
		$$
		\mathcal{\widehat{M}}_\lambda\eta=\int_{0}^{a_2}\beta(a, \cdot)e^{-\lambda a}\mathcal{U}(0, a)\eta\, da,\; \forall\eta\in X. 
		$$
		But under Assumption \ref{beta}, we have $\mathcal{M}_\lambda=\mathcal{\widehat{M}}_\la$ for all $\la\in \widetilde V$. It follows that $\sig(\mathcal{\widehat{A})}\cap \widetilde V=\sig(\mathcal{A})\cap \widetilde V$, thus we can study the principal spectral theory of $\mathcal{\widehat{A}}$ instead of $\mathcal{A}$ in the following, provided Assumption \ref{beta} holds. Further, in order to not introduce too many notations, we still denote $\mathcal{A}$ and $\mathcal{B}$ under Assumption \ref{beta}.   
		
		\begin{remark}\label{mubeta}
			{\rm Under Assumption \ref{beta}, the eigenvalue problem is inverted from infinite maximum age $\hat{a}=\infty$ or finite maximum age $\hat{a}<\infty$, with possibly unbounded death rate $\mu$, into finite maximum age $a_2<\infty$ with bounded death rate $\mu$. Hence we have that
				\begin{itemize}
					
					\item [(i)] Assumption \ref{mu} is for $\hat a=\infty$, while here we replace $\hat a$ by $a_2<\infty$. Thus $s(\mathcal{A})$ and $s(\mathcal{B}_1+\mathcal{C})$ always exist now.
					
					\item [(ii)] Assumption \ref{irreducible} can be modified as follows: there exists no $a_0$ such that $\underline{\beta}(a)=0$ a.e. $[a_0, a_2]$, if needed, see Section \ref{Strong Maximum Principle}. Further, it is equivalent to $\int_a^{a_2}\underline{\beta}(l)dl>0$ for any $a\in[0, a_2)$.
				\end{itemize}  
			}
		\end{remark} 
		
		Observe now that the principal eigenfunction $v$, if exists, satisfies $v>0$ in $[0, a_2]\times\overline{\Omega}$ due to $\mu\in C(\overline{\Omega}, L^\infty_+(0, a_2))$ now. Hence we will indicate the auxiliary eigenvalue problem corresponding to $\mathcal{\widehat{A}}$ instead of $\mathcal{A}$ in the following context, as long as we use the fact that $v>0$ in $[0, a_2]\times\overline{\Omega}$.  
		
		Now we provide the second criteria under Assumption \ref{beta}.
		\begin{theorem}[Existence of principal eigenvalues - II]\label{II}
			Let Assumption \ref{beta} hold. Assume that
			\begin{eqnarray}\label{criteria}
				x\to\frac{1}{\alpha^{**}-\alpha(x)}\notin L^1_{loc}(\overline{\Omega})
			\end{eqnarray}
and that for each $x\in\overline{\Omega}$, the operator $\mathcal{B}_1^x+\mathcal{C}^x$ possesses a positive eigenvector $\phi\in W^{1, 1}(0, a_2)$ corresponding to $\alpha(x)$, then $s(\mathcal{A})$ is the principal eigenvalue of $\mathcal{A}$. Here $\alpha(x)$ is defined in Proposition \ref{Galphax} and $\mathcal{B}_1^x+\mathcal{C}^x$ is defined in Remark \ref{age-operator} in $(0, a_2)$. 
		\end{theorem}
		
\begin{proof}
The idea of the proof below comes from Liang et al. \cite[Lemma 3.8]{liang2019principal} or Bao and Shen \cite[Proposition 3.1]{bao2017criteria}.
			
\textbf{First step.} By assumption, for any $x\in\overline{\Omega}$, $\phi(\cdot, x):=[\phi(x)](\cdot)$ as a principal eigenfunction of $\mathcal{B}_1^x+\mathcal{C}^x$ is belonging to $W^{1, 1}(0, a_2)$. Further, we can normalize the family $\{\phi(\cdot, x)\}_{x\in\overline\Omega}$ such that $\norm{\phi(\cdot, x)}_{L^1(0, a_2)}=1$ for any $x\in\overline\Omega$. Now we will prove that the eigenfunction $\phi(\cdot, x)$ is continuous for all $x\in\overline{\Omega}$. 
			
To this aim, let us first write down the equation that $\phi$ satisfies,
\begin{equation}\label{E}
\begin{cases}
					\partial_a\phi(a, x)=-(D+\mu(a, x))\phi(a, x)-\alpha(x)\phi(a, x), \; a\in(0, a_2),\\
					\phi(0, x)=\int_0^{a_2}\beta(a, x)\phi(a, x)da.
				\end{cases}
			\end{equation}
			Fix $x_0\in\overline{\Omega}$ and let us choose a sequence $\{x_n\}_{n\ge1}\subset\overline\Omega$ satisfying $x_n\to x_0$ as $n\to\infty$. Consider the sequence $\phi(\cdot, x_n)$. Observing the first equation of \eqref{E}, one has 
			$$
			\norm{\partial_a\phi(\cdot, x_n)}_{L^1(0, a_2)}\le C, 
			$$
where $C>0$ denotes some constant that may vary from line to line but is independent of $n\ge0.$ It follows that the sequence $\{\phi(\cdot, x_n)\}_{n\ge0}$ is bounded in $W^{1, 1}(0, a_2)$ which is continuously embedded into $L^\infty(0, a_2)$ so that $\norm{\phi(\cdot, x_n)}_{L^\infty(0, a_2)}\le C$. Again by the first equation of \eqref{E}, one has 
			$$
			\norm{\partial_a\phi(\cdot, x_n)}_{L^\infty(0, a_2)}\le C.
			$$
			Thus we have $\norm{\phi(\cdot, x_n)}_{W^{1, \infty}(0, a_2)}\le C$. By the compact Sobolev embedding, we can find a limit, denoted by $\widehat \phi(\cdot)\in C([0, a_2])$, up to a subsequence such that 
			$$
			\phi(\cdot, x_n)\to \widehat \phi(\cdot)\;\text{uniformly on $[0, a_2]$.}
			$$
			Since $x\to\mu(, x)\in C(\overline{\Omega}, L^\infty(0, a_2))$, one has $\mu(\cdot, x_n)\to\mu(\cdot, x_0)$ in $L^\infty(0, a_2)$, and thus 
			$$
			\mu(\cdot, x_n)\phi(\cdot, x_n)\to\mu(\cdot, x_0)\widehat{\phi}(\cdot) \text{ in }L^\infty(0, a_2).
			$$ 
			Applying the same argument to $\beta$ and then passing to the limit on \eqref{E}, one obtains
			\begin{equation}
				\begin{cases}
					\partial_a\widehat \phi(a)=-(D+\mu(a, x_0))\widehat \phi(a)-\alpha(x_0)\widehat \phi(a), & a\in(0, a_2),\\
					\widehat \phi(0)=\int_0^{a_2}\beta(a, x_0)\widehat \phi(a)da
				\end{cases}
			\end{equation}
with $\norm{\widehat \phi}_{L^1(0, a_2)}=1$ and $\widehat \phi\ge0$. Hence $\widehat \phi$ is the principal eigenfunction of the operator $\mathcal{B}_1^{x_0}+\mathcal{C}^{x_0}$ corresponding to $\alpha(x_0)$. Next thanks to the simplicity of the principal eigenvalue, we have $\widehat \phi(a)=\phi(a, x_0)$. Thus the function $x\to\phi(\cdot, x)$ is continuous from $x\in\overline{\Omega}$ to $C([0, a_2])$. Then we normalize $\phi$ such that $\max_{(a, x)\in[0, a_2]\times\overline{\Omega}}\phi(a, x)=1$.
			
\textbf{Second step.} We will prove the main conclusion. According to Assumption \ref{J} on the kernel $J$, there exist $r>0$ and $c_0>0$ such that $J(x-y)>c_0$ for all $x, y\in\overline{\Omega}$ with $|x-y|<r$. 
			
			Next let 
			$$
			c_1=\min_{(a, x)\in[0, a_2]\times\overline{\Omega}}\phi(a, x). 
			$$
			Due to Assumption \ref{beta}, $c_1>0$ holds. Since $(\alpha^{**}-\alpha)^{-1}\notin L^1_{loc}(\overline{\Omega})$, we can choose $\zeta>\alpha^{**}$, some $\delta>0$ and $x_1\in\Omega$ such that $B(x_1, \delta)\subset B(x_1, 2\delta)\subset\Omega$, 
			$$ 
			\int_{B(x_1, \delta)}\frac{1}{\zeta-\alpha(x)}dx\geq2(Dc_0c_1)^{-1}, 
			$$
			and $3\delta<r$, where $B(x, r)$ is the ball centered at $x$ with radius $r$. Let $p: \overline{\Omega}\to\R$ be a continuous function on $\overline{\Omega}$ such that 
			\begin{eqnarray}
				p(x)=
				\begin{cases}
					1,&x\in B(x_1, \delta),\\
					0,&x\in\overline{\Omega}\setminus B(x_1, 2\delta)
				\end{cases}
			\end{eqnarray}
with $0\le p(x)\le1$ for all $x\in\overline{\Omega}$ and $\widetilde{\phi}(a, x)=[\widetilde{\phi}(x)](a):=p(x)\phi(a, x), \forall(a, x)\in[0, a_2]\times\overline{\Omega}$. It then follows that for any $(a, x)\in[0, a_2]\times(\overline{\Omega}\setminus B(x_1, 2\delta))$, we have
			$$ 
			\int_{\Omega}J(x-y)\frac{dy}{\zeta-\alpha(y)}\widetilde{\phi}(a, y)\geq0. 
			$$
			For any $(a, x)\in(0, a_2)\times B(x_1, 2\delta)$, we see that
			\begin{eqnarray}
				\int_{\Omega}J(x-y)\frac{dy}{\zeta-\alpha(y)}\widetilde{\phi}(a, y)\ge\int_{B(x_1, \delta)}J(x-y)\frac{dy}{\zeta-\alpha(y)}\phi(a, y)\ge 2c_0c_1(Dc_0c_1)^{-1}\geq2D^{-1}\widetilde{\phi}(a, x).\nonumber
			\end{eqnarray}
			Note that for all $x\in\overline{\Omega}$, one has
			\begin{eqnarray}
				&&(\zeta I -\mathcal{B}_1-\mathcal{C})^{-1}(0, \widetilde{\phi})=(0, \psi)\nonumber\\
				\text{with}&&(0, \psi(\cdot, x))=\left[(\zeta I -\mathcal{B}_1^x-\mathcal{C}^x)^{-1}(0, \widetilde{\phi}(x))\right](\cdot)=\left[(\zeta-\alpha(x))^{-1}(0, \widetilde{\phi}(x))\right](\cdot).
			\end{eqnarray} 
			Recalling \eqref{Flambdainverse}, it then follows that 
			\begin{eqnarray}
				\mathcal{F}_\zeta(0, \widetilde{\phi})=\mathcal{B}_2(\zeta I -\mathcal{B}_1-\mathcal{C})^{-1}(0, \widetilde{\phi})\geq2(0, \widetilde{\phi})>(0, \widetilde{\phi}).
			\end{eqnarray}  
			Thus, there exists $\zeta>s(\mathcal{B}_1+\mathcal{C})$ such that $r(\mathcal{F}_\zeta)>1$. Then by Corollary \ref{corollary}, it follows that $s(\mathcal{A})>s(\mathcal{B}_1+\mathcal{C})$ which implies the desired result by Theorem \ref{principle}.
\end{proof}	

\begin{remark}\label{GKR}
{\rm Observe that the criterion for the existence of principal eigenvalues that we provided in \eqref{criterion} and \eqref{criteria} are reasonable and comparable with the ones obtained for nonlocal problems, for instance, see Coville \cite{coville2010simple} who employed generalized Krein-Rutman Theorem (see Edmunds et al. \cite{edmunds1972non}, Nussbaum \cite{nussbaum1981eigenvectors}) to obtain analogue conditions for the existence of principal eigenvalues of a nonlocal diffusion operator. In fact in our case, \eqref{criterion} and \eqref{criteria} imply that $s(\mathcal{A})>s(\mathcal{B}_1+\mathcal{C})$. It follows by Remark \ref{sasb}-(i) and \ref{omegaT} that $(\la -s(\mathcal{A}))^{-1}$, the spectral radius of $(\la I -\mathcal{A})^{-1}$,  tends to $\infty$ and $(\la -s(\mathcal{B}_1+\mathcal{C}))^{-1}$, the spectral radius of $(\la I -\mathcal{B}_1-\mathcal{C})^{-1}$, remains bounded as $\la\downarrow s(\mathcal{A})$. On the other hand, since $\mathcal{B}_2$ is a compact perturbator of $\mathcal{B}_1+\mathcal{C}$ (which implies that $\mathcal{F}_\la$ defined in \eqref{Flambdainverse} is compact), it follows that for $\la>s(\mathcal{A})$,
$$
r_e((\la I -\mathcal{A})^{-1})=r_e\left((\la I -\mathcal{B}_1-\mathcal{C})^{-1}\left(\sum_{j=0}^\infty\mathcal{F}_\la^j\right)\right)= r_e((\la I -\mathcal{B}_1-\mathcal{C})^{-1})\le r((\la I -\mathcal{B}_1-\mathcal{C})^{-1}).
$$
Hence $(\la I -\mathcal{A})^{-1}$ is essentially compact if $\la$ is sufficiently close to $s(\mathcal{A})$. Then the generalized Krein-Rutman theorem can be applied to conclude that $(\la -s(\mathcal{A}))^{-1}$ is the principal eigenvalue of $(\la I -\mathcal{A})^{-1}$. It follows that $s(\mathcal{A})$ is the principal eigenvalue of $\mathcal{A}$ by the spectral mapping theorem. The above argument is just the idea of obtaining the existence of principal eigenvalues combining the theory of resolvent positive operators with their perturbations and generalized Krein-Rutman Theorem, see Thieme \cite{thieme1998remarks}.
				
Let us see the essential compactness in another way. Observing again from Remark \ref{omegaT} and Remark \ref{sasb}-(i) that if $X=L^1(\Omega)$, then $\omega(T_{\mathcal{A}_0})=s(\mathcal{A})>s(\mathcal{B}_1+\mathcal{C})=\omega(T_{(\mathcal{B}_1+\mathcal{C})_0})$, it follows together with Corollary \ref{compactperturbator} that $\{T_{\mathcal{A}_0}(t)\}_{t\geq0}$ is an essentially compact semigroup by Theorem \ref{EC} in Appendix.}
\end{remark}

\subsection{Relation Between $\mathcal{M}_\lambda$ and $\mathcal{A}$}
We next give a proposition to characterize the relation between the eigenvalues of $\mathcal{M}_\lambda$ and those of $\mathcal{A}=\mathcal{B}+\mathcal{C}$, also see Kang and Ruan \cite{kang2021nonlinear} or Walker \cite{walker2013some}.
		
\begin{proposition}\label{eigenvalues}
Under Assumption \ref{beta}, let $\lambda\in\mathbb{C}$ and $m\in\mathbb{N}\setminus\{0\}$. Then $\lambda\in\sigma_p(\mathcal{A})$ with geometric multiplicity $m$ if and only if $1\in\sigma_p(\mathcal{M}_\lambda)$ with geometric multiplicity $m$, where $\sigma_p(A)$ denotes the point spectrum of $A$. 
		\end{proposition}
		
		\begin{proof}
			Let $\lambda\in\mathbb{C}$. Suppose that $\lambda\in\sigma_p(\mathcal{A})$ has geometric multiplicity $m$ so that there are $m$ linearly independent elements
			$$
			\begin{pmatrix}0, \phi_1\end{pmatrix},...,\begin{pmatrix}0, \phi_m\end{pmatrix}\in {\rm dom}(\mathcal{A}) \text{ with } (\lambda I -\mathcal{A})\begin{pmatrix}0, \phi_j\end{pmatrix}=(0, 0) \text{ for } j=1,...,m. 
			$$
			Then by solving the above eigenvalue problem explicitly, we get
			$$ 
			\phi_j(a)=e^{-\lambda a}\mathcal{U}(0, a)\phi_j(0) \text{ with } \phi_j(0)=\mathcal{M}_\lambda\phi_j(0). 
			$$
			Hence, $\phi_1(0),...,\phi_m(0)$ are necessarily linearly independent eigenvectors of $\mathcal{M}_\lambda$ corresponding to the eigenvalue $1$. 
			
			Now suppose that $1\in\sigma_p(\mathcal{M}_\lambda)$ has geometric multiplicity $m$ so that there are linearly independent $\psi_1,...,\psi_m\in X$ with $\mathcal{M}_\lambda\psi_j=\psi_j$ for $j=1,...,m$. Put $(0, \phi_j)=\begin{pmatrix}0, e^{-\lambda a}\mathcal{U}(0, a)\psi_j\end{pmatrix}\in \mathcal{X}_0$ and note that for $j=1,...,m$, we have
			$$ 
			\partial_a\phi_j+\lambda\phi_j-D[K-I]\phi_j+\mu\phi_j=0,\; \int_{0}^{a_2}\beta(a, \cdot)\phi_j(a)da=\mathcal{M}_\lambda\psi_j=\psi_j=\phi_j(0), 
			$$
			which is equivalent to 
			$$ 
			\mathcal{A}(0, \phi_j)=\lambda(0, \phi_j)\text{ and } (0, \phi_j)\in {\rm dom}(\mathcal{A}). 
			$$
			Thus $\lambda\in\sigma_p(\mathcal{A})$. If $\alpha_1,...,\alpha_m$ are any scalars, the unique solvability of the Cauchy problem 
			$$ 
			\partial_a\phi+\lambda\phi-D[K-I]\phi+\mu\phi=0,\; \phi(0, x)=\sum_{j=1}^{m}\alpha_j\psi_j 
			$$
			ensures that $(0, \phi_1),...,(0, \phi_m)$ are linearly independent. Hence, the result follows.   
		\end{proof}
		
\subsection{A Counterexample}
In this subsection we construct an example of kernel $J$ and functions $\beta(a, x), \mu(a, x)$ for which the operator $\mathcal{A}$ admits no eigenvalue with a positive eigenfunction in ${\rm dom}(\mathcal{A})$ when \eqref{criterion} is not satisfied. In particular, $\mathcal{A}$ admits no principal eigenvalue. This implies that our criterion is sharp in the sense that if they are not satisfied, $\mathcal{A}$ may not have a principal eigenvalue. 
		
Let $\beta(a, x)\equiv\beta(x), \mu(a, x)\equiv\mu$ and the maximum age $\hat{a}=\infty$, where $\beta\in C(\overline{\Omega})$ and $\mu>0$ obviously satisfy the assumptions in the Introduction. Let us suppose that $\mathcal{A}$ admits an eigenvalue of $\lambda_1$ with a positive eigenfunction $(0, \phi)\in {\rm dom}(\mathcal{A})$; that is,
		\begin{equation*}
			\begin{cases}
				\partial_a\phi(a, x)=\int_{\Omega}J(x-y)\phi(a, y)dy-\phi(a, x)-\mu\phi(a, x)-\lambda_1\phi(a, x),&(a, x)\in(0, \hat{a})\times\overline{\Omega},\\
				\phi(0, x)=\int_{0}^{\infty}\beta(x)\phi(a, x)da,& x\in\overline{\Omega}.
			\end{cases}
		\end{equation*}
		Integrating the above equation from $0$ to $\hat{a}=\infty$ and using the condition $\phi(\infty, x)\equiv0$, we obtain
		$$ 
		-\phi(0, x)=\int_{\Omega}J(x-y)\int_{0}^{\infty}\phi(a, y)dady-\int_{0}^{\infty}\phi(a, x)da-\mu\int_{0}^{\infty}\phi(a, x)da-\lambda_1\int_{0}^{\infty}\phi(a, x)da. 
		$$
		Now denote $\psi(x)=\int_{0}^{\infty}\phi(a, x)da$, we then have
		$$ 
		\int_{\Omega}J(x-y)\psi(y)dy-\psi(x)+(\beta(x)-\mu)\psi(x)-\lambda_1\psi(x)=0.$$
		Thus by Coville's criterion \cite[Theorem 5.1]{coville2010simple}, we have the following theorem.
		
\begin{theorem}
Let $J\equiv\rho$ on $\Omega$, where $\rho>0$ is a constant, and set $\beta_{\max}=\max_{x\in\overline{\Omega}}\beta(x)$. If 
$$
\rho\int_{\Omega}\frac{1}{\beta_{\max}-\beta(x)}dx<1,
$$ 
then $\mathcal{A}$ admits no eigenvalue with a positive eigenfunction in ${\rm dom}(\mathcal{A})$. In particular, $\mathcal{A}$ admits no principal eigenvalue.
\end{theorem}
		
\begin{remark}
			{\rm In this example, the function $G_\alpha$ reads as
				$$ 
				G_\alpha(x)=\beta(x)\int_{0}^{\infty}e^{-(\alpha+1+\mu)a}da\;\text{ for } \alpha>-\mu-1. 
				$$ 
				Note that the function 
				$$ 
				\alpha\rightarrow\int_{0}^{\infty}e^{-(\alpha+1+\mu)a}da 
				$$ 
				is continuously decreasing from $\infty$ to $0$ on $(-\mu-1, \infty)$. Thus, we can choose $\alpha^{**}>-\mu-1$ such that 
				$$ 
				\int_{0}^{\infty}e^{-(\alpha^{**}+1+\mu)a}da=1/\beta_{\max}. 
				$$ 
				Now for any $\alpha<\alpha^{**}$,  
				$$ \rho/\beta_{\max}\int_{\Omega}\frac{1}{1-G_\alpha(x)}dx<1\Rightarrow\rho\int_{\Omega}\frac{1}{\beta_{\max}-\beta(x)}dx<1. 
				$$
Hence, the criterion for the existence of principal eigenvalues that we gave in \eqref{criterion} is reasonable and comparable with the one for nonlocal problems, see Coville \cite{coville2010simple} and Shen and Vo \cite{shen2019nonlocal}.}
		\end{remark}

\section{Limiting Properties}\label{LP}
In this section we will study the effects of the diffusion rate on the spectral bound $s(\mathcal{A})$ of $\mathcal{A}$. Remembering in the previous section, we have shown that under Assumption \ref{beta}, the eigenvalue problem to $\mathcal{A}$ on $[0, \hat{a})$ is equivalent to the one on $[0, a_2]$ with bounded death rate $\mu$ and further the principal eigenfunction associated with $s(\mathcal{A})$ is positive in $[0, a_2]$. 
		
		Thus in the following context, we will let Assumption \ref{beta} hold throughout the whole section. Before proceeding, let us first clarify the strict positivity in $X$. 
		\begin{eqnarray}
		&&f>0 \text{ in $X=C(\overline{\Omega})$ means that }f(x)>0 \text{ for all $x\in\overline{\Omega}$},\nonumber\\
		&&f>0 \text{ in $X=L^1(\Omega)$ means that $\int_{\Omega}f^*(x)f(x)dx>0$ for any $f^*\in L_+^\infty(\Omega)\setminus\{0\}$}. \nonumber
	    \end{eqnarray}
        Now following Berestycki et al. \cite{berestycki1994principal,berestycki2016definition}, we introduce the following definition. 
		\begin{definition}\label{GPE}
			{\rm Define the \textit{generalized principal eigenvalues} by
				\begin{eqnarray}
					\begin{cases}
						\lambda_p(\mathcal{A}):=\sup\{\lambda\in\mathbb{R}: \exists\, \phi\in W^{1, 1}((0, a_2), X) \text{ s.t. } \phi>0 \text{ and }(-\mathcal{A}+\lambda)(0, \phi)\leq(0, 0) \text{ in } [0, a_2]\},\\
						\lambda_p'(\mathcal{A}):=\inf\{\lambda\in\mathbb{R}: \exists\, \phi\in W^{1, 1}((0, a_2), X) \text{ s.t. } \phi>0 \text{ and }(-\mathcal{A}+\lambda)(0, \phi)\geq(0, 0) \text{ in } [0, a_2]\}.
					\end{cases}
				\end{eqnarray}
			}
		\end{definition} 
		
		
		Note that the sets in Definition \ref{GPE} are nonempty, see the proof of Theorem \ref{Dlambda} in the following. As mentioned before, such ideas are widely used to prove the existence and asymptotic behavior of principal eigenvalues with respect to diffusion rate, see Coville \cite{coville2010simple},  Li et al. \cite{li2017eigenvalue} and Su et al. \cite{su2020asymptotic} for nonlocal diffusion equations, Shen and Vo \cite{shen2019nonlocal} and Su et al. \cite{su2020generalised} for time periodic nonlocal diffusion equations. As Shen and Vo \cite{shen2019nonlocal} highlighted for the time periodic case, we remark that our parabolic-type operator $\mathcal{A}$ containing $\partial_a$ is not self-adjoint, and thus we lack the usual $L^2(\Omega)$ variational formula for the principal eigenvalue $s(\mathcal{A})$. The generalized principal eigenvalues $\lambda_p(\mathcal{A})$ and  $\lambda_p'(\mathcal{A})$ defined in (\ref{GPE}) remedy the situation and play crucial roles in the following text. 
		
		\subsection{Without Kernel Scaling}
		In this subsection first we study the diffusion without kernel scaling and have the following result.
		
\begin{proposition}\label{principleequal}
Let Assumption \ref{beta} hold and in addition, assume that $\lambda_1(\mathcal{A})$ is the eigenvalue of $\mathcal{A}$ associated with $(0, \phi_1)\in {\rm dom}(\mathcal{A})$ with $\phi_1>0$ in $[0, a_2]$, then one has $\lambda_1(\mathcal{A})=\lambda_p(\mathcal{A})=\lambda_p'(\mathcal{A})$.
\end{proposition}
		
\begin{proof}
Denote $\lambda_1(\mathcal{A})$ by $\lambda_1$. First, we prove that $\lambda_1=\lambda_p$. Since $\lambda_1$ is the eigenvalue of $\mathcal{A}$ associated with $(0, \phi_1)\in {\rm dom}(\mathcal{A})$ with $\phi_1>0$ in $[0, a_2]$; that is,
			\begin{eqnarray}\label{phi1}
				\mathcal{A}(0, \phi_1)-\lambda_1(0, \phi_1)=(0, 0) \; \text{ in } [0, a_2],
			\end{eqnarray} 
			and since $\phi_1>0$ in $[0, a_2]$, we have $\lambda_1\leq\lambda_p$. Suppose by contradiction that $\lambda_1<\lambda_p$. From the definition of $\lambda_p$, there are $\lambda\in(\lambda_1, \lambda_p)$ and $(0, \phi)\in {\rm dom}(\mathcal{A})$ with $\phi>0$ in $[0, a_2]$ such that
			\begin{eqnarray}
				-\mathcal{A}(0, \phi)+\lambda(0, \phi)\leq(0, 0) \;\text{ in } [0, a_2]; \nonumber
			\end{eqnarray}
			that is, for $0\le a\le a_2$
			\begin{eqnarray}\label{phi}
				\begin{cases}
					\partial_a\phi(a)-D[K-I]\phi(a)+\mu(a, \cdot)\phi+\lambda\phi\leq0,\\
					\phi(0)-\int_{0}^{a_2}\beta(a, \cdot)\phi(a)da\leq0.
				\end{cases}
			\end{eqnarray}
			Now solving the first inequality in \eqref{phi}, we obtain
			$$ 
			\phi(a)\leq e^{-\lambda a}\mathcal{U}(0, a)\phi(0). 
			$$ 
			Plugging it into the second inequality in \eqref{phi}, we have
			\begin{eqnarray}\label{leq}
				\phi(0)\leq\int_{0}^{a_2}\beta(a, \cdot)e^{-\lambda a}\mathcal{U}(0, a)\phi(0)da.
			\end{eqnarray} 
			It follows that $\mathcal{M}_\lambda\phi(0)\geq\phi(0)$, which implies that $r(\mathcal{M}_\lambda)\geq1$. But we know that $\lambda_1$ is the eigenvalue of $\mathcal{A}$, then by Proposition \ref{eigenvalues} we have $r(\mathcal{M}_{\lambda_1})=1$. Since $\lambda\rightarrow r(\mathcal{M}_\lambda)$ is decreasing by the arguments in Proposition \ref{sA}, one has $\lambda_1\geq\lambda$. This contradiction leads to $\lambda_1=\lambda_p$.
			
			Next, we prove $\lambda_1=\lambda_p'$. Obviously, $\lambda_1\geq\lambda_p'$. Assume by contradiction that $\lambda_1>\lambda_p'$. Then there are $\tilde{\lambda}\in(\lambda_p', \lambda_1)$ and $(0, \tilde{\phi})\in {\rm dom}(\mathcal{A})$ with $\tilde{\phi}>0$ in $[0, a_2]$ such that $-\mathcal{A}(0, \tilde{\phi})+\tilde{\lambda}(0, \tilde{\phi})\geq(0, 0)$. By reversing the above inequalities, we have the desired conclusion via a similar argument as above.
		\end{proof}
		
		Now we give the main theorem in this section about the effects of diffusion rate on $s(\mathcal{A})$. In the next result, we write $s^D(\mathcal{A})$ for $s(\mathcal{A})$ to highlight the dependence on $D$.
		
\begin{theorem}\label{Dlambda}
Let Assumption \ref{beta} hold and, in addition, assume that $s^D(\mathcal{A})$ is the principal eigenvalue of $\mathcal{A}$, then the function $D\rightarrow s^D(\mathcal{A})$ is continuous on $(0, \infty)$ and satisfies 
\begin{equation}
s^D(\mathcal{A})\rightarrow
\begin{cases}
s(B_1+\mathcal{C}) &\;\text{ as } D\rightarrow 0^+,\\
-\infty &\;\text{ as } D\rightarrow\infty,
\end{cases}
\end{equation}
where $B_1$ is defined as follows,
$$ 
B_1(0, f):=\left(-f(0, \cdot), \;-\partial_af-\mu f\right), \; f\in W^{1, 1}((0, a_2), X). 
$$
\end{theorem}
		
\begin{proof}
Since $s^D(\mathcal{A})$ is a simple eigenvalue, the continuity of $D\rightarrow s^D(\mathcal{A})$ follows from the similar argument in Theorem \ref{II} or see Kato \cite[Section IV. 3.5]{kato2013perturbation} for the classical perturbation theory. 
			
			For the limits, we first claim that for every $\epsilon>0$, there exists $D_\epsilon>0$ such that
			\begin{eqnarray}\label{Depsilon}
				s^D(\mathcal{A})\leq s(B_1+\mathcal{C})+\epsilon, \;\forall D\in(0, D_\epsilon).
			\end{eqnarray}	
			Denote $\vartheta=s(B_1+\mathcal{C})$. Consider equation \eqref{E} with $D=0$ which is written as follows:
			\begin{equation}\label{supsolution}
				\begin{cases}
					\partial_a\phi(a, x)=-(\alpha(x)+\mu(a, x))\phi(a, x),&\;(a, x)\in(0, a_2]\times\overline{\Omega},\\
					\phi(0, x)=\int_{0}^{a_2}\beta(a, x)\phi(a, x)da,&\; x\in\overline{\Omega}.
				\end{cases}
			\end{equation}
			By Proposition \ref{Galphax} ($D=0$), we know that for each $x\in\overline{\Omega}$, \eqref{supsolution} has a positive solution $\phi\in W^{1, 1}(0, a_2)$ given by
			$$
			\phi(a, x)=e^{-\alpha(x)a}\pi(0, a, x)\phi(0, x).
			$$ 
			Moreover, by the argument in Theorem \ref{II}, $\phi(\cdot, x)$ is continuous in $x\in\overline{\Omega}$ and thus {$\phi\in C(\overline{\Omega}, W^{1, 1}(0, a_2))$. Next from the first equation of \eqref{supsolution}, one has $\phi\in W^{1, 1}((0, a_2), C(\overline{\Omega}))$. Thus we now have $\phi>0$ in $[0, a_2]\times\overline{\Omega}$ with $(0, \phi)\in {\rm dom}(\mathcal{A})=\{0\}\times W^{1, 1}((0, a_2), X)$.} Further, it is easy to check that for $(a, x)\in[0, a_2]\times\overline{\Omega}$
			\begin{eqnarray}
				&& \left[-\mathcal{A}(0, \phi)+(\vartheta+\epsilon)(0, \phi)\right](a, x) \nonumber\\
				&&\quad\quad = \left( \phi(0, x) - \int_{0}^{a_2} \beta(a, x)\phi(a, x)da, \right.\nonumber\\
				&&\quad\quad\quad \left. \partial_a\phi(a, x)-D\left[\int_{\Omega} J(x-y)\phi(a, y)dy-\phi(a, x)\right]+\mu(a, x)\phi+(\vartheta+\epsilon)\phi \right). \nonumber
			\end{eqnarray} 
			Since $\min_{(a, x)\in[0, a_2]\times\overline{\Omega}}\phi(a, x)>0$ and $\max_{(a, x)\in[0, a_2]\times\overline{\Omega}}\phi(a, x)<\infty$, it is straightforward to check that for each $\epsilon>0$, there exists $D_\epsilon>0$ such that for each $D\in(0, D_\epsilon)$, there holds
			\begin{eqnarray}
				&&\partial_a\phi(a, x)-D\left[\int_{\Omega}J(x-y)\phi(a, y)dy-\phi(a, x)\right]+\mu(a, x)\phi+(\vartheta+\epsilon)\phi\nonumber\\
				&&\quad =-D\left[\int_{\Omega}J(x-y)\phi(a, y)dy-\phi(a, x)\right]+(\vartheta-\alpha(x))\phi+\epsilon\phi\nonumber\\
				&&\quad \geq-D\left[\int_{\Omega}J(x-y)\phi(a, y)dy-\phi(a, x)\right]+\epsilon\phi\nonumber\\
				&&\quad \geq 0 \;\text{ in }[0, a_2]\times\overline{\Omega},\nonumber
			\end{eqnarray} 
			where we used $\vartheta\geq\alpha(x)$ from Proposition \ref{Galphax} where $D=0$. It then follows that 
			$$
			-\mathcal{A}(0, \phi)+(\vartheta+\epsilon)(0, \phi)\geq(0, 0)
			$$ 
			which, by the definition of $\lambda_p'(\mathcal{A}),$ implies that $s^D(\mathcal{A})=\lambda_p'(\mathcal{A})\leq s(B_1+\mathcal{C})+\epsilon$.
			
			Next from Remark \ref{sasb}, we have
			$$ 
			s(B_1+\mathcal{C})-D=s(\mathcal{B}_1+\mathcal{C})\leq s^D(\mathcal{A}). 
			$$
			Setting $D\rightarrow0^+$, we find that
			$$ 
			s(B_1+\mathcal{C})\leq\liminf_{D\rightarrow0^+}s^D(\mathcal{A})\leq\limsup_{D\rightarrow0^+}s^D(\mathcal{A})\leq s(B_1+\mathcal{C})+\epsilon, \;\forall\epsilon>0, 
			$$
			which leads to $s^D(\mathcal{A})\rightarrow s(B_1+\mathcal{C})$ as $D\rightarrow0^+$.
			
			Finally, to show that $s^D(\mathcal{A})\rightarrow-\infty$ as $D\rightarrow\infty$, we consider the operator $K-I$. It is known from Shen and Xie \cite[Theorem 2.1 and Proposition 3.4]{shen2015approximations} that the principal eigenvalue of $-K+I$ exists and is positive. Let $\lambda^0>0$ be the principal eigenvalue of $-K+I$ and $\Psi^0\in C(\overline{\Omega})$ be an associated positive eigenfunction. Let $(\la^1, \Psi^1(a))$ be the principal eigenpair of the age-structured operator; that is, they satisfy the following equation,
			\begin{equation*}
				\begin{cases}
					\partial_a\Psi^1(a)=-(\lambda^1+\underline{\mu}(a))\Psi^1(a),\\
					\Psi^1(0)=\int_{0}^{a_2}\overline{\beta}(a)\Psi^1(a)da,
				\end{cases}
			\end{equation*}
			where $\lambda^1$ satisfies 
			$$ 
			\int_{0}^{a_2}\overline{\beta}(a)e^{-\lambda^1a}e^{-\int_{0}^{a}\underline{\mu}(s)ds}da=1. 
			$$
			Note that $\Psi^1(a)$ is positive. Now let $\lambda_D=-D\lambda^0+\lambda^1$ and $\Psi(a, x)=\Psi^0(x)\Psi^1(a)$. We have that $\Psi>0$ in $[0, a_2]\times\overline{\Omega}, (0, \Psi)\in {\rm dom}(\mathcal{A})$ and that for $(a, x)\in[0, a_2]\times\overline{\Omega}$
			\begin{eqnarray}
				&&\left[-\mathcal{A}(0, \Psi)+\lambda_D(0, \Psi)\right](a, x)\nonumber\\
				&&\quad\quad = \left( \Psi(0, x)-\int_{0}^{a_2}\!\beta(a, x)\Psi(a, x)da, \right. \nonumber \\
				&&\quad\quad\quad \left. \partial_a\Psi(a, x)-D\left[\int_{\Omega}J(x-y)\Psi(a, y)dy-\Psi(a, x)\right]+\mu(a, x)\Psi\!+\!\lambda_D\Psi \right) \nonumber\\
				&&\quad\quad :=\begin{pmatrix}
					{\rm I_1}(x), {\rm I_2}(a, x)
				\end{pmatrix}.\nonumber
			\end{eqnarray}
			Next we have 
			\begin{eqnarray}
				{\rm I_2}(a, x)&=& \partial_a\Psi^1(a) \Psi^0(x) - D \left[ \int_{\Omega} J(x-y)\Psi^0(y) dy - \Psi^0(x) \right] \Psi^1(a) \nonumber \\
				&&\quad + \mu(a, x) \Psi^1(a)\Psi^0(x) + (- D\lambda^0+\lambda^1)\Psi^0(x)\Psi^1(a) \nonumber\\
				&\geq & \left(\partial_a\Psi^1(a)+\underline{\mu}(a)\Psi^1(a)+\lambda^1\Psi^1(a)\right)\Psi^0(x)+D\lambda^0\Psi^0(x)\Psi^1(a)-D\lambda^0\Psi^0(x)\Psi^1(a) \nonumber\\
				&=& 0, \;\text{ in }(0, a_2]\times\overline{\Omega}\nonumber
			\end{eqnarray}
			and 
			\begin{eqnarray}
				{\rm I_1}(x)=\int_{0}^{a_2}\overline{\beta}(a)\Psi^1(a)da\Psi^0(x)-\int_{0}^{a_2}\beta(a, x)\Psi^1(a)\Psi^0(x)da\geq0, \;\text{ in }\overline{\Omega}.\nonumber
			\end{eqnarray}
			Thus, $(\lambda_D, (0, \Psi))$ is a test pair for $\lambda_p'(\mathcal{A})$. It follows that $s^D(\mathcal{A})=\lambda_p'(\mathcal{A})\leq\lambda_D$. Setting $D\rightarrow\infty$, we reach at $s^D(\mathcal{A})\rightarrow-\infty$ as $D\rightarrow\infty$.
		\end{proof}
		
		\begin{remark}\label{alpha1}
			{\rm From Proposition \ref{rGalpha}, we know that $s(B_1+\mathcal{C})$ equals the value $\alpha_1$ which satisfies 
				$$ 
				\max_{x\in\overline{\Omega}}\int_{0}^{a_2}\beta(a, x)e^{-\alpha_1a}\pi(0, a, x)da=1. 
				$$
			} 
		\end{remark}
		
		\begin{theorem}
			Let Assumption \ref{beta} hold and assume that $\mu(a, x)=\mu_1(a)+\mu_2(x), \;\beta(a, x)\equiv\beta(a)$, where $\beta, \mu_1\in L^\infty_+(0, a_2)$ and $\mu_2\in C_+(\overline\Omega)$. In addition, assume that $J$ is symmetric, i.e. $J(x)=J(-x),$ and that the operator
			$$ 
			v\rightarrow D\left[\int_{\Omega}J(\cdot-y)v(y)-v\right]-\mu_2(\cdot)v: C(\overline{\Omega})\rightarrow C(\overline{\Omega}) 
			$$
			admits a principal eigenvalue. Then the function $D\rightarrow s^D(\mathcal{A})$ is decreasing.
		\end{theorem}
		
		\begin{proof}
			We define $L: C(\overline\Omega)\to C(\overline\Omega)$ and $\mathcal T: \{0\}\times W^{1, 1}(0, a_2)\subset \{0\}\times L^1(0, a_2)\to \{0\}\times L^1(0, a_2)$ respectively as follows,
			\begin{eqnarray} 
				&& Lv=D\left[\int_{\Omega}J(\cdot-y)v(y)dy-v\right]-\mu_2(\cdot)v, \; v\in C(\overline{\Omega})\nonumber\\ 
				&&\mathcal{T}(0, \phi)=\left(-\phi(0)+\int_{0}^{a_2}\beta(a)\phi(a)da, \;-\partial_a\phi-\mu_1\phi\right),\; \phi\in W^{1, 1}(0, a_2). \nonumber
			\end{eqnarray}
			Let $(\lambda_1^D(L), v_1)$ be the principal eigenpair of $-L$. Then by the same argument as in Shen and Vo \cite[Theorem C(2)]{shen2019nonlocal}, we have that $D\rightarrow\lambda_1^D(L)$ is increasing. Now let $(\la_1(\mathcal{T}), (0, \phi_1))$ be the principal eigenpair of $\mathcal{T}$. It follows that $s^D(\mathcal{A})=-\lambda_1^D(L)+\lambda_1(\mathcal{T})$ is the principal eigenvalue of $\mathcal{A}$ with the principal eigenfunction $\left(0, v_1\phi_1\right)$. As $D\rightarrow\lambda_1^D(L)$ is increasing, so $D\rightarrow s^D(\mathcal{A})$ is decreasing.
		\end{proof}
		
		\subsection{With Kernel Scaling}
		In this subsection we study the effects of the diffusion rate and diffusion range on the principal eigenvalue. Define $K_{\gamma, \Omega}$ as follows:
		\begin{equation}\label{LsigmamOmega}
			[K_{\gamma, \Omega}f](\cdot)=\int_{\Omega}J_\gamma(\cdot-y)f(y)dy, \; f\in X.
		\end{equation}
		Here the kernel $J_\ga$ satisfies the scaling $J_\gamma(x)=\frac{1}{\gamma^N}J\left(\frac{x}{\gamma}\right)$ for $x\in\mathbb{R}^N$, where $\gamma>0$ represents the diffusion range. Then we introduce the nonlocal diffusion operator $\frac{D}{\gamma^m}\left[K_{\gamma, \Omega}-I\right]$, where $m\ge0$ denotes the cost parameter.
		
		Write $\mathcal{A}_{\gamma, m, \Omega}=\mathcal{B}_{\gamma, m, \Omega}+\mathcal{C}$ for $\mathcal{A}=\mathcal{B}+\mathcal{C}$ to highlight the dependence on $\gamma, m$ and $\Omega$ and further denote $\mathcal{B}_{\gamma, m, \Omega}^\mu,\, \mathcal{C}^\beta$ for $\mathcal{B},\, \mathcal{C}$ to represent the dependence on $\mu$ and $\beta$ respectively. We mainly employ the idea from Shen and Vo \cite[Theorem D]{shen2019nonlocal} to prove the following results.
		
		\begin{proposition}\label{property}
			Let Assumption \ref{beta} hold and let $m\geq0, \gamma>0$. We have the following statements.
			\begin{itemize}
				\item [(i)] 
				$s(\mathcal{B}_{\gamma, m, \Omega}+\mathcal{C}^\beta)$ is non-decreasing with respect to $\beta$ and $s(\mathcal{B}_{\gamma, m, \Omega}^\mu+\mathcal{C})$ is non-increasing with respect to $\mu$; 
				
				\item [(ii)] Let the assumptions in Theorem \ref{I} or Theorem \ref{II} hold, where $D$ is changed into $\frac{D}{\gamma^m}$, then $s(\mathcal{A}_{\gamma, m, \Omega})$ is the principal eigenvalue of $\mathcal{A}_{\gamma, m, \Omega}$. Assume that $\lambda_1(\mathcal{A}_{\gamma, m, \Omega})$ is the eigenvalue of $\mathcal{A}_{\gamma, m, \Omega}$ associated with $(0, \phi)\in {\rm dom}(\mathcal{A}_{\ga, m, \Omega})$ satisfying $\phi>0$ in $[0, a_2]$, then 
				$$ 
				\lambda_1(\mathcal{A}_{\gamma, m, \Omega})=\lambda_p(\mathcal{A}_{\gamma, m, \Omega})=\lambda'_{p}(\mathcal{A}_{\gamma, m, \Omega}); 
				$$
				
				\item [(iii)] Moreover, $\la_p(\mathcal{B}_{\gamma, m, \Omega}^\mu+\mathcal{C})$ is Lipschitz continuous with respect to $\mu$ in $C(\overline{\Omega}, L^\infty_+(0, a_2))$. More precisely, we have
				$$
				|\la_p(\mathcal{B}_{\gamma, m, \Omega}^{\mu_1}+\mathcal{C})-\la_p(\mathcal{B}_{\gamma, m, \Omega}^{\mu_2}+\mathcal{C})|\leq\norm{\mu_1-\mu_2}_{C(\overline{\Omega}, L^\infty_+(0, a_2))}
				$$
				for any $\mu_1, \mu_2\in C(\overline{\Omega}, L^\infty_+(0, a_2))$;
				
				\item [(iv)] If $\Omega_1\subset\Omega_2$, then $\lambda_p'(\mathcal{A}_{\gamma, m, \Omega_1})\leq\lambda_p'(\mathcal{A}_{\gamma, m, \Omega_2})$. Assume that in addition $X=C(\overline{\Omega})$, $s(\mathcal{A}_{\gamma, m, \Omega_1})$ and $s(\mathcal{A}_{\gamma, m, \Omega_2})$ are principal eigenvalues of $\mathcal{A}_{\gamma, m, \Omega_1}$ and $\mathcal{A}_{\gamma, m, \Omega_2}$ respectively, then we have
				$$ 
				|\lambda_p'(\mathcal{A}_{\gamma, m, \Omega_1})-\lambda_p'(\mathcal{A}_{\gamma, m, \Omega_2})|\leq C_0|\Omega_2\setminus\Omega_1|, 
				$$
				where $C_0>0$ depends on $a, \gamma, m, J_\gamma$ and $\Omega_2$;
				
				\item[(v)] Assume that $s(\mathcal{A}_{\gamma, m, \Omega})$ is the principal eigenvalue of $\mathcal{A}_{\gamma, m, \Omega}$, then the function $\gamma\rightarrow s(\mathcal{A}_{\gamma, m, \Omega})$ is continuous.
			\end{itemize}
		\end{proposition}
		
		\begin{proof}
			For $(i)$, if $\beta_1\geq\beta_2$, it follows that $\mathcal{M}_\lambda(\beta_1)\geq\mathcal{M}_\lambda(\beta_2)$ in the positive operator sense which implies that $r(\mathcal{M}_\lambda(\beta_1))\geq r(\mathcal{M}_\lambda(\beta_2))$. Thus by Proposition \ref{sA}, we have $s(\mathcal{B}_{\gamma, m, \Omega}+\mathcal{C}^{\beta_1})\geq s(\mathcal{B}_{\gamma, m, \Omega}+\mathcal{C}^{\beta_2})$ by the monotonicity of $r(\mathcal{M}_\lambda)$ with respect to $\lambda$. 
			
			Similarly, when $\mu_1\geq\mu_2$, since $\mathcal{U}_{\mu_1}(0, a)$ and $\mathcal{U}_{\mu_2}(0, a)$ are positive in $C(\overline{\Omega})$, we have $\mathcal{U}_{\mu_1}(0, a)\leq\mathcal{U}_{\mu_2}(0, a)$ in the positive operator sense, which implies that $\mathcal{M}_\lambda(\mu_1)\leq\mathcal{M}_{\lambda}(\mu_2)$. Then it follows that $r(\mathcal{M}_\lambda(\mu_1))\leq r(\mathcal{M}_\lambda(\mu_2))$, hence $s(\mathcal{B}_{\gamma, m, \Omega}^{\mu_1}+\mathcal{C})\leq s(\mathcal{B}_{\gamma, m, \Omega}^{\mu_2}+\mathcal{C})$ by the above argument.
			
			For $(ii)$, it follows from Theorem \ref{I} or Theorem \ref{II} and Proposition \ref{principleequal}.
			
			For $(iii)$, fix $\lambda<\la_p(\mathcal{B}_{\gamma, m, \Omega}^{\mu_1}+\mathcal{C})$. By Definition \ref{GPE}, there exists $(0, \phi)\in {\rm dom}(\mathcal{A}_{\gamma, m, \Omega})$ with $\phi>0$ in $[0, a_2]$ such that
			$$ 
			-\mathcal{B}_{\gamma, m, \Omega}^{\mu_1}(0, \phi)-\mathcal{C}(0, \phi)+\lambda(0, \phi)\leq(0, 0), \;\text{ in }[0, a_2]. 
			$$
			Clearly, we have
			\begin{eqnarray}
				(0, 0) &\geq&-\mathcal{B}_{\gamma, m, \Omega}^{\mu_1}(0, \phi)-\mathcal{C}(0, \phi)+\lambda(0, \phi)\nonumber\\
				&=& \left(\phi(0)-\int_{0}^{a_2}\beta(a, \cdot)\phi(a)da,\quad \partial_a\phi-\frac{D}{\gamma^m}\left[K_{\gamma, \Omega}-I\right]\phi+\mu_1\phi+\lambda\phi\right)\nonumber\\
				&=&\left(\phi(0)-\int_{0}^{a_2}\beta(a, \cdot)\phi(a)da,\quad \partial_a\phi-\frac{D}{\gamma^m}\left[K_{\gamma, \Omega}-I\right]\phi+[\mu_2+\mu_1-\mu_2]\phi+\lambda\phi\right)\nonumber\\
				&\geq&\left(\phi(0)-\int_{0}^{a_2}\beta(a, \cdot)\phi(a)da,\quad \partial_a\phi-\frac{D}{\gamma^m}\left[K_{\gamma, \Omega}-I\right]\phi+\mu_2\phi+\lambda-\norm{\mu_1-\mu_2}_{C(\overline{\Omega}, L^\infty_+(0, a_2))}\phi\right). \nonumber
			\end{eqnarray}
			Again by Definition \ref{GPE}, we get
			$$ 
			\lambda-\norm{\mu_1-\mu_2}_{C(\overline{\Omega}, L^\infty_+(0, a_2))}\leq \la_p(\mathcal{B}_{\gamma, m, \Omega}^{\mu_2}+\mathcal{C}). 
			$$
			Since this holds for any $\lambda<\la_p(\mathcal{B}_{\gamma, m, \Omega}^{\mu_1}+\mathcal{C})$, we arrive at 
			$$ 
			\la_p(\mathcal{B}_{\gamma, m, \Omega}^{\mu_1}+\mathcal{C})-\la_p(\mathcal{B}_{\gamma, m, \Omega}^{\mu_2}+\mathcal{C})\leq\norm{\mu_1-\mu_2}_{C(\overline{\Omega}, L^\infty_+(0, a_2))}. 
			$$
			Switching the roles of $\mu_1$ and $\mu_2$, we find that
			$$ 
			\la_p(\mathcal{B}_{\gamma, m, \Omega}^{\mu_2}+\mathcal{C})-\la_p(\mathcal{B}_{\gamma, m, \Omega}^{\mu_1}+\mathcal{C})\leq\norm{\mu_1-\mu_2}_{C(\overline{\Omega}, L^\infty_+(0, a_2))}. 
			$$
			Thus the result follows.
			
			For $(iv)$, let $(\la, \psi)$ be a pair for $\lambda_p'(\mathcal{A}_{\gamma, m, \Omega_2})$. Then $\psi\in W^{1, 1}((0, a_2), C(\overline{\Omega}_2))$ satisfies $\psi>0$ in $[0, a_2]\times\overline\Omega$. Define
			$$
			[\psi_{\overline{\Omega}_1}(a)](x)=[\psi(a)](x), \; (a, x)\in[0, a_2]\times\overline{\Omega}_1\subset[0, a_2]\times\overline{\Omega}_2.
			$$
			Then $\psi_{\overline{\Omega}_1}>0$ in $[0, a_2]$ and belongs into $W^{1. 1}((0, a_2), C(\overline\Omega_1))$. Moreover, for any $(a, x)\in[0, a_2]\times\overline{\Omega}_1$, one has
\begin{eqnarray}
				&&(-\mathcal{A}_{\ga, m, \Omega_1}+\la)(0, \psi_{\overline{\Omega}_1})\nonumber\\
				&&\quad = \left(\psi_{\overline{\Omega}_1}(0)-\int_0^{a_2}\beta(a, \cdot)\psi_{\overline{\Omega}_1}(a)da,\quad
				\partial_a\psi_{\overline{\Omega}_1}-\frac{D}{\gamma^m}\left[K_{\ga, \Omega_1}-I\right]\psi_{\overline{\Omega}_1}+\mu(a, \cdot)\psi_{\overline{\Omega}_1}+\lambda\psi_{\overline{\Omega}_1}\right)\nonumber\\
				&& \quad \ge\left(\psi(0)-\int_0^{a_2}\beta(a, \cdot)\psi(a)da,\quad
				\partial_a\psi-\frac{D}{\gamma^m}\left[K_{\ga, \Omega_2}-I\right]\psi+\mu(a, \cdot)\psi+\lambda\psi\right)\nonumber\\
				&& \quad =(-\mathcal{A}_{\ga, m, \Omega_2}+\la)(0, \psi)\nonumber\\
				&&\quad \ge (0, 0).
			\end{eqnarray}
			That is, $(\la, \psi_{\overline{\Omega}_1})$ is a test pair for $\la_p'(\mathcal{A}_{\ga, m, \Omega_1})$ and hence, $\la\ge\la_p' (\mathcal{A}_{\ga, m, \Omega_1})$. Taking the infimum over all such $\la$, we arrive at
			\begin{eqnarray}\label{p'}
				\la_p'(\mathcal{A}_{\ga, m, \Omega_2})\ge \la_p'(\mathcal{A}_{\ga, m, \Omega_1}).
			\end{eqnarray} 
			Note that here we obtain \eqref{p'} which is reversed compared with \cite[Proposition 6.1-(iv)]{shen2019nonlocal} since we are using $\la_p'$ instead of their relation on $\la_p$.
			
			To prove the second statement, first note that $W^{1, 1}((0, a_2), C(\overline{\Omega}))\subset C([0, a_2]\times\overline{\Omega})$, it follows that $\psi>0$ in $[0, a_2]\times\overline{\Omega}$. Thus we can choose a eigenpair $(\la_p'(\mathcal{A}_{\ga, m, \Omega_2}), \psi)$ of $\mathcal{A}_{\ga, m, \Omega_2}$ with normalization $\max_{(a, x)\in[0, a_2]\times\overline{\Omega}_2}\psi=1$. Direct calculations yield
			\begin{eqnarray}
				&&(-\mathcal{A}_{\ga, m, \Omega_1}+\la_p'(\mathcal{A}_{\ga, m, \Omega_2})(0, \psi)\nonumber\\
				&&\quad = \left(\psi(0)-\int_0^{a_2}\beta(a, \cdot)\psi(a)da,\quad
				\partial_a\psi-\frac{D}{\gamma^m}\left[K_{\ga, \Omega_1}-I\right]\psi+\mu(a, \cdot)\psi+\la_p'(\mathcal{A}_{\ga, m, \Omega_2})\psi\right)\nonumber\\
				&&\quad = \left(0,\quad
				\frac{D}{\ga^m}\int_{\Omega_2\setminus\Omega_1}J_\ga(\cdot-y)[\psi(a)](y)dy\right)\nonumber\\
				&&\quad \le\left(0, \quad \frac{D\norm{J_\ga}_\infty}{\ga^m}|\Omega_2\setminus\Omega_1|\right)\nonumber\\
				&&\quad \le \left(0, \quad \frac{D\norm{J_\ga}_\infty}{\ga^m\min_{\overline{\Omega}_1}\psi}|\Omega_2\setminus\Omega_1|\psi\right).
			\end{eqnarray}
			That is, 
			$$
			-\mathcal{A}_{\ga, m, \Omega_1}(0, \psi)+\left[\la_p'(\mathcal{A}_{\ga, m, \Omega_2})-C_0|\Omega_2\setminus\Omega_1|\right](0, \psi)\le(0, 0), \text{ in }[0, a_2],
			$$
			where $C_0=\frac{D\norm{J_\ga}_\infty}{\ga^m\min_{\overline{\Omega}_1}\psi}$. By (ii), one has
			$$
			\la_p'(\mathcal{A}_{\ga, m, \Omega_1})=\la_p(\mathcal{A}_{\ga, m, \Omega_1})\ge\la_p'(\mathcal{A}_{\ga, m, \Omega_2})-C_0|\Omega_2\setminus\Omega_1|.
			$$
			This together with \eqref{p'} leads to the result.
			
			For $(v)$ we can use the same argument in proving the continuity of $D\rightarrow s^D(\mathcal{A})$ in Theorem \ref{Dlambda} combing with the argument in Shen and Vo \cite[Proposition 6.1 (5)]{shen2019nonlocal} and omit it here.
		\end{proof}
		
		\begin{theorem}\label{gamma}
			Let Assumption \ref{beta} hold. Assume that $s(\mathcal{A}_{\gamma, m, \Omega})$ is the principal eigenvalue of $\mathcal{A}_{\gamma, m, \Omega}$, then 
			\begin{itemize}
				\item [(i)] As $\gamma\rightarrow\infty$, there holds 
				\begin{equation}
					s(\mathcal{A}_{\gamma, m, \Omega})\rightarrow
					\begin{cases}
						s(B_1+\mathcal{C})-D, &\; m=0,\\
						s(B_1+\mathcal{C}), &\; m>0, 
					\end{cases}
				\end{equation}
				where 
				$$ 
				B_1(0, f)=\left(-f(0, \cdot), \;-\partial_af-\mu f\right), \; f\in W^{1, 1}((0, a_2), X); 
				$$

				\item[(ii)] Assume, in addition, that $J$ is symmetric, i.e. $J(x)=J(-x)$, $\mu\in C^2(\R^N, L^\infty_+(0, a_2))$ and $\beta\in C^2(\R^N, L^\infty_+(0, a_2))$. As $\gamma\rightarrow0^+$, there holds
				$$ 
				s(\mathcal{A}_{\gamma, m, \Omega})\rightarrow s(B_1+\mathcal{C}),\;\forall m\in[0, 2).
				$$
				
				\item[(iii)] In the case $m=0$, if $\Omega$ contains the origin and $\mu(a, x)$ is radially symmetric and radially non-decreasing with respect to $x$; namely, $\mu(a, x)=\mu(a, y)$ if $|x|=|y|$ and $\mu(a, x)\geq \mu(a, y)$ if $|x|\geq|y|$ for all $a\in[0, \hat{a})$, then $\gamma\rightarrow s(\mathcal{A}_{\gamma, 0, \Omega})$ is non-increasing.
			\end{itemize}
		\end{theorem}
		
		\begin{proof}
			$(i)$ We first prove the result in the case $m>0$. By Remark \ref{sasb}, we find that
			$$ 
			s(\mathcal{A}_{\gamma, m, \Omega})\geq s(\mathcal{B}_1+\mathcal{C})=s(B_1+\mathcal{C})-\frac{D}{\gamma^m}. 
			$$
			It follows that
			\begin{eqnarray}\label{liminf}
				\liminf_{\gamma\rightarrow\infty}s(\mathcal{A}_{\gamma, m, \Omega})\geq s(B_1+\mathcal{C}).
			\end{eqnarray}
			Let us still consider equation \eqref{supsolution} associated with a positive solution $\phi\in C(\overline{\Omega}, W^{1, 1}(0, a_2))$. As before, one has $\phi\in W^{1, 1}((0, a_2), C(\overline{\Omega}))$. Set $\vartheta=s(B_1+\mathcal{C})$ again. For any $\epsilon>0$, we see that for $(a, x)\in[0, a_2]\times\overline{\Omega}$,
			\begin{eqnarray}
				&&\left[-\mathcal{A}_{\gamma, m, \Omega}(0, \phi)+(\vartheta+\epsilon)(0, \phi)\right](a, x) \nonumber\\
				&&\quad\quad =\left(\phi(0, x)-\int_{0}^{a_2}\beta(a, x)\phi(a, x)da, \right. \nonumber\\
				&&\quad\quad\quad \left. \partial_a\phi(a, x)-\frac{D}{\gamma^m}\left[\int_{\Omega}J_\gamma(x-y)\phi(a, y)dy-\phi(a, x)\right]+\mu(a, x)\phi+(\vartheta+\epsilon)\phi \right)\nonumber
			\end{eqnarray}
			and
			\begin{eqnarray}
				&&\partial_a\phi(a, x)-\frac{D}{\gamma^m}\left[\int_{\Omega}J_\gamma(x-y)\phi(a, y)dy-\phi(a, x)\right]+\mu(a, x)\phi+(\vartheta+\epsilon)\phi\nonumber\\
				&&\quad =-\frac{D}{\gamma^m}\left[\int_{\Omega}J_\gamma(x-y)\phi(a, y)dy-\phi(a, x)\right]+\epsilon\phi+(\vartheta-\alpha(x))\phi\nonumber\\
				&&\quad \geq-\frac{D}{\gamma^m}\left[\int_{\Omega}J_\gamma(x-y)\phi(a, y)dy-\phi(a, x)\right]+\epsilon\phi.\label{Dgammam}
			\end{eqnarray} 
			Since $\min_{(a, x)\in[0, a_2]\times\overline{\Omega}}\phi(a, x)>0$ and 
			$$ 
			\norm{\frac{D}{\gamma^m}\left[\int_{\Omega}J_\gamma(\cdot-y)\phi(a, y)dy-\phi(a, \cdot)\right]}_{C(\overline{\Omega})}\rightarrow0 \;\text{ as } \gamma\rightarrow\infty 
			$$
			there is $\gamma_\epsilon>0$ such that \eqref{Dgammam}$\geq0$ in $[0, a_2]\times\overline{\Omega}$ for all $\gamma\geq\gamma_\epsilon$. It then follows that 
			$$
			-\mathcal{A}_{\gamma, m, \Omega}(0, \phi)+(\vartheta+\epsilon)(0, \phi)\geq(0, 0) \text{ in }[0, a_2]\times\overline{\Omega}, 
			$$
			which by the definition of $\lambda_p'(\mathcal{A}_{\gamma, m, \Omega})$ implies that 
			$$ 
			s(\mathcal{A}_{\gamma, m, \Omega})=\lambda_p'(\mathcal{A}_{\gamma, m, \Omega})\leq s(B_1+\mathcal{C})+\epsilon. 
			$$ 
			The arbitrariness of $\epsilon$ then yields (i) with \eqref{liminf} for $m>0$. 
			
			Now we prove the result in the cases $m=0$. Remark \ref{sasb} ensures that $s(\mathcal{A}_{\gamma, m, \Omega})\geq s(\mathcal{B}_1+\mathcal{C})=s(B_1+\mathcal{C})-D$. It remains to show that 
			\begin{eqnarray}\label{s-D}
				\limsup_{\gamma\rightarrow\infty}s(\mathcal{A}_{\gamma, m, \Omega})\leq s(B_1+\mathcal{C})-D. 
			\end{eqnarray}
			Let $\phi$ be a solution of $\eqref{supsolution}$ as above. For any $\epsilon>0$, we have for $(a, x)\in[0, a_2]\times\overline{\Omega}$ that
			\begin{eqnarray}
				&&\left[-\mathcal{A}_{\gamma, 0, \Omega}(0, \phi)+(\vartheta+\epsilon)(0, \phi)\right](a, x) \nonumber\\
				&&\quad\quad= \left(\phi(0, x)-\int_{0}^{a_2}\beta(a, x)\phi(a, x)da, \right. \nonumber\\ 
				&&\quad\quad\quad \left. \partial_a\phi(a, x)-D\left[\int_{\Omega}J_\gamma(x-y)\phi(a, y)dy-\phi(a, x)\right]+\mu(a, x)\phi+(\vartheta+\epsilon)\phi \right).\nonumber
			\end{eqnarray}
			Next we have
			\begin{eqnarray}
				&&\partial_a\phi(a, x)-D\left[\int_{\Omega}J_\gamma(x-y)\phi(a, y)dy-\phi(a, x)\right]+\mu(a, x)\phi+(\vartheta+\epsilon)\phi\nonumber\\
				&&\quad =-D\left[\int_{\Omega}J_\gamma(x-y)\phi(a, y)dy-\phi(a, x)\right]+\epsilon\phi+(\vartheta-\alpha(x))\phi\nonumber\\
				&&\quad \geq-D\left[\int_{\Omega}J_\gamma(x-y)\phi(a, y)dy-\phi(a, x)\right]+\epsilon\phi. 
				\label{Dgammam2}
			\end{eqnarray} 
			Hence for $\epsilon>0$, there holds
			$$
			-\mathcal{A}_{\gamma, 0, \Omega}(0, \phi)+(\vartheta+\epsilon-D)(0, \phi)\geq \left(0, -D\int_{\Omega}J_\gamma(x-y)\phi(a, y)dy+\epsilon\phi\right),\; \text{ in }[0, a_2]\times\overline{\Omega}. 
			$$
			As $\norm{\int_{\Omega}J_\gamma(\cdot-y)\phi(a, y)dy}_{C(\overline{\Omega})}\rightarrow0$ uniformly in $[0, a_2]$ when $\gamma\rightarrow\infty$, we can follow the arguments in the case $m>0$ to conclude \eqref{s-D}.
			
			$(ii)$ {Let $\phi=\phi(a,x)>0$ be the solution of \eqref{key} with $D=0$ which is defined for $x\in \R^N$ with normalization 
				$$
				\int_0^{a_2}\beta(a,x)\phi(a,x)da=1,\;\forall x\in\R^N. 
				$$
				Next we claim that the map $x\to((0, 
				\phi(\cdot, x)), \al(x))$ is of class $C^2$ from $\R^N$ into $\{0\}\times C([0,  a_2])\times\R$. The proof is given in Appendix, see Lemma \ref{C^4}.}
			
			For any $\epsilon>0$, similar argument as in \eqref{Dgammam} leads to 
			\begin{eqnarray}
				&&\partial_a\phi(a, x)-\frac{D}{\gamma^m}\left[\int_{\Omega}J_\gamma(x-y)\phi(a, y)dy-\phi(a, x)\right]+\mu(a, x)\phi+(\vartheta+\epsilon)\phi\nonumber\\
				&&\quad \geq-\frac{D}{\gamma^m}\left[\int_{\Omega}J_\gamma(x-y)\phi(a, y)dy-\phi(a, x)\right]+\epsilon\phi\nonumber\\
				&&\quad \geq -\frac{D}{\gamma^m}\left[\int_{\mathbb{R}^N}J_\gamma(x-y)\phi(a, y)dy-\phi(a, x)\right]+\epsilon\phi\nonumber\\
				&&\quad =-\frac{D}{\gamma^m}\left[\int_{\mathbb{R}^N}J(z)\phi(a, x+\gamma z)dz-\phi(a, x)\right]+\epsilon\phi \; \text{ in }[0, a_2]\times\overline{\Omega}.\nonumber
			\end{eqnarray}
			Then by Taylor expansion we have
			\begin{eqnarray}
				\frac{D}{\gamma^m}\left[\int_{\mathbb{R}^N}J(z)\phi(a, x+\gamma z)dz-\phi(a, x)\right]=D\gamma^{2-m}\sum_{|\nu|=2}\int_{\mathbb{R}^N}R_\nu(a, x, z)J(z)z^\nu dz, \nonumber
			\end{eqnarray}
			where $\nu=(\nu_1, \cdots, \nu_N)$ is the usual multiple index, and  
			$$
			R_\nu(a, x, z)=\frac{2}{\nu!}\int_0^1(1-s)\partial^\nu\phi(a, x+s\gamma z)ds,
			$$
			and we used the symmetry of $J$ with respect to each component. 
			
			Since $\phi\in C^2(\mathbb{R}^N, C([0, a_2]))$ and $J$ is compactly supported, there holds the boundedness of the function $x\to \sum_{|\nu|=2}\int_{\mathbb{R}^N}R_\nu(a, x, z)J(z)z^\nu dz$ on $\overline\Omega$ uniformly in $0\le a\le a_2$ and $0\le\ga\le 1$. It follows from the assumption $m\in[0, 2)$ that 
			$$
			\partial_a\phi(a, x)-\frac{D}{\gamma^m}\left[\int_{\Omega}J_\gamma(x-y)\phi(a, y)dy-\phi(a, x)\right]+\mu(a, x)\phi+(\vartheta+\epsilon)\phi\ge0
			$$ 
			in $[0, a_2]\times\overline{\Omega}$ for sufficiently small $\gamma$. This implies that
			$$ 
			-\mathcal{A}_{\gamma, m, \Omega}(0, \phi)+(\vartheta+\epsilon)(0, \phi)\geq(0, 0) \;\text{  in } [0, a_2]\times\overline{\Omega},\;0<\gamma\ll1, 
			$$
			which follows that
			$$ 
			\limsup_{\gamma\to0^+}s(\mathcal{A}_{\gamma, m, \Omega})=\limsup_{\gamma\to0^+}\lambda_p'(\mathcal{A}_{\gamma, m, \Omega})\leq s(B_1+\mathcal{C}). 
			$$
			Now we show the reverse inequality, i.e.
			\begin{eqnarray}\label{limsup}
				\liminf_{\gamma\to0^+}s(\mathcal{A}_{\gamma, m, \Omega})\geq s(B_1+\mathcal{C}). 
			\end{eqnarray} 
			For any $\epsilon>0$ sufficiently small, there exists an open ball $B_\de$ of radius $\de=\de(\epsilon)$ such that $\alpha(x)+\epsilon\geq s(B_1+\mathcal{C})=:\vartheta$ in $B_\de\cap \overline{\Omega}$, where $\alpha(x)$ is from Proposition \ref{Galphax} for $D=0$. In fact, if $\de>\epsilon$, we can reduce the ball to $B_\epsilon$ such that $\alpha(x)+\epsilon\geq\vartheta$ in $B_\epsilon\cap\overline{\Omega}$. Thus without loss of generality, we assume $\de\le\epsilon$. 
			
			Next let {$\tilde{\phi}_\epsilon\in C^2(\R^N, C([0, a_2]))\cap W^{1, 1}((0, a_2), C(\R^N))$} be nonnegative and satisfy
			$$ 
			\tilde{\phi}_\epsilon=\begin{cases} \phi \text{ in }\,[0, a_2]\times(\overline{B}_\de\cap\overline{\Omega}),\\
			 \;0 \,\text{ in }\,[0, a_2]\times(\mathbb{R}^N\setminus (B_{2\de}\cap\overline{\Omega}))\end{cases}
			 \,\text{ and }\,\sup_{[0, a_2]\times\mathbb{R}^N}\tilde{\phi}_\epsilon\leq\sup_{[0, a_2]\times\mathbb{R}^N}\phi,
			$$
			where {$\phi\in C^2(\R^N, C([0, a_2]))\cap W^{1, 1}((0, a_2), C(\R^N))$} is the solution of \eqref{key} with $D=0$ provided by Lemma \ref{C^4}. Then we have for $(a, x)\in[0, a_2]\times (B_\de\cap\overline{\Omega})$ that
			$$ 
			-\mathcal{A}_{\gamma, m, B_\de\cap\overline{\Omega}}(0, \phi)+\left(\vartheta-\epsilon-\frac{1}{|\ln\epsilon|}\right)(0, \phi):=({\rm I_3, I_4}), 
			$$
			where for any $(a, x)\in[0, a_2]\times(B_\de\cap\overline{\Omega})$ one has
			$$ 
			{\rm I_3}(x)=\phi(0, x)-\int_{0}^{a_2}\beta(a, x)\phi(a, x)da=0 
			$$
			and 
			\begin{eqnarray}
				{\rm I_4}(a, x)&\!=\!&\partial_a\phi(a, x)\!-\!\frac{D}{\gamma^m}\left[\int_{B_\de\cap\overline{\Omega}}J_\gamma(x-y)\phi(a, y)dy \!-\!\phi(a, x)\right]\!+\!\left[\mu(a, x)\!+\!\vartheta\!-\!\epsilon\!-\!\frac{1}{|\ln\epsilon|}\right]\phi(a, x)\nonumber\\
				&\!=\!&\!-\!\frac{D}{\gamma^m}\left[\int_{B_\de\cap\overline{\Omega}}J_\gamma(x-y)\phi(a, y)dy \!-\!\phi(a, x)\right]\!+\!\left[-\alpha(x)\!+\!\vartheta\!-\!\epsilon\!-\!\frac{1}{|\ln\epsilon|}\right]\phi(a, x)\nonumber\\
				&\!\leq\!&\!-\!\frac{D}{\gamma^m}\left[\int_{B_\de\cap\overline{\Omega}}J_\gamma(x-y)\phi(a, y)dy \!-\!\phi(a, x)\right]\!-\!\frac{\phi(a, x)}{|\ln\epsilon|}\nonumber\\
				&\!=\!&\!-\!\frac{D}{\gamma^m}\left[\int_{\mathbb{R}^N}J_\gamma(x-y)\tilde{\phi}_\epsilon(a, y)dy \!-\!\tilde{\phi}_\epsilon(a, x)\!-\!\int_{(B_{2\de}\setminus B_\de)\cap\overline{\Omega}}J_\gamma(x-y)\tilde{\phi}_\epsilon(a, y)dy\right]\!-\!\frac{\phi(a, x)}{|\ln\epsilon|}.\nonumber
			\end{eqnarray}
We still based on Taylor expansion deal with the estimate of							
\begin{eqnarray}
				&&\frac{D}{\gamma^m}\left[\int_{\mathbb{R}^N}J_\gamma(x-y)\tilde{\phi}_\epsilon(a, y)dy-\tilde{\phi}_\epsilon(a, x)-\int_{(B_{2\de}\setminus B_\de)\cap\overline{\Omega}}J_\gamma(x-y)\tilde{\phi}_\epsilon(a, y)dy\right]+\frac{\phi(a, x)}{|\ln\epsilon|}\nonumber\\
				&& \quad = D\gamma^{2-m}\int_{\mathbb{R}^N}J(z)\sum_{|\nu|=2}\frac{2}{\nu!}\int_0^1(1-s)^1\partial^\nu\tilde{\phi}_\epsilon(a, x+s\gamma z)dsz^\nu dz\nonumber\\
				&& \qquad -\frac{D}{\ga^{m+N}}\int_{(B_{2\de}\setminus B_\de)\cap\overline{\Omega}}J\left(\frac{x-y}{\ga}\right)\tilde{\phi}_\epsilon(a, y)dy+\frac{\phi(a, x)}{|\ln\epsilon|}\nonumber\\
				&& \quad := \mathcal{I}^1_{\epsilon, \ga}(a, x)+\mathcal{I}^2_{\epsilon, \ga}(a, x)+\mathcal{I}^3_{\epsilon}(a, x).\nonumber
			\end{eqnarray}
			Now note that $\min_{[0, a_2]\times(\overline{B}_{\de}\cap\overline{\Omega})}\phi(a, x)\ge\min_{[0, a_2]\times\overline{\Omega}}\phi(a, x)>0$ for all $0<\epsilon\ll1$. Choosing $\epsilon=\gamma^k$ with $k=\frac{m+2N}{N},$ we then have for all $0<\ga\ll1$ that
			$$
			\sup_{[0, a_2]\times\R^N}|\mathcal{I}^1_{\epsilon, \ga}|\le C_1\ga^{2-m}, \; \sup_{[0, a_2]\times\R^N}|\mathcal{I}^2_{\epsilon, \ga}|\le C_2\ga^N, \; \inf_{[0, a_2]\times\R^N}|\mathcal{I}^3_{\epsilon}|\ge\frac{C_3}{|\ln(\ga^k)|},
			$$
			where $C_i>0$ are constants independent on $\ga$ for $1\le i\le3$. As $\lim\limits_{\ga\to0^+}\ga^\beta|\ln\ga|=0$ for any $\beta>0$, the term $\mathcal{I}^3_{\epsilon}$ dominates $\mathcal{I}^1_{\epsilon, \ga}$ and $\mathcal{I}^2_{\epsilon, \ga}$ for small $\ga$. Thus we have
			$$ 
			-\mathcal{A}_{\gamma, m, B_\de\cap\overline{\Omega}}(0, \phi)+\left(\vartheta-\epsilon-\frac{1}{|\ln\epsilon|}\right)(0, \phi)\leq(0, 0) \; \text{ in } [0, a_2]\times (B_{\de}\cap\overline{\Omega}),\; 0<\gamma\ll1. 
			$$
			It then follows from the generalized principal eigenvalue and Proposition \ref{principleequal} that 
			$$ 
			s(\mathcal{A}_{\gamma, m, B_{\de}\cap\overline{\Omega}})=\lambda_p(\mathcal{A}_{\gamma, m, B_{\de}\cap\overline{\Omega}})\geq s(B_1+\mathcal{C})-\gamma^k-\frac{1}{|\ln\gamma^k|},\;0<\gamma\ll1. 
			$$
			By Proposition \ref{property}-$(iv)$, we have $s(\mathcal{A}_{\gamma, m, \Omega})\geq s(\mathcal{A}_{\gamma, m, B_{\de}\cap\overline{\Omega}})$, which yields that
			$$ 
			s(\mathcal{A}_{\gamma, m, \Omega})\geq s(B_1+\mathcal{C})-\gamma^k-\frac{1}{|\ln\gamma^k|},\;0<\gamma\ll1. 
			$$
			Letting $\gamma\rightarrow0$, we have \eqref{limsup}. Thus the result is desired.
			
			$(iii)$ Recall that $\mu$ is a radially non-decreasing function of $x$. For $\ga_1\ge\ga_2$, we show $s(\mathcal{A}_{\ga_1, 0, \Omega})\le s(\mathcal{A}_{\ga_2, 0, \Omega})$. It is equivalent to show $\la_p'(\mathcal{A}_{\ga_1, 0, \Omega})\le\la_p'(\mathcal{A}_{\ga_2, 0, \Omega})$.
			
			To this aim, set $\Omega_\ga=\frac{1}{\ga}\Omega$ and $\mu_\ga(a, x)=\mu(a, \ga x)$ for $a\in[0, a_2]$ and $x\in\Omega_\ga$. Clearly, $\la_p'(\mathcal{A}_{\ga, 0, \Omega})=\la_p'(\mathcal{B}^{\mu_\ga}_{1, 0, \Omega_\ga}+\mathcal{C})$. Therefore, we need to show that
			$$
			\la_p'(\mathcal{B}^{\mu_{\ga_1}}_{1, 0, \Omega_{\ga_1}}+\mathcal{C})\le\la_p'(\mathcal{B}^{\mu_{\ga_2}}_{1, 0, \Omega_{\ga_2}}+\mathcal{C}).
			$$
			It suffices to prove the inequality $\la_p'(\mathcal{B}^{\mu_{\ga_1}}_{1, 0, \Omega_{\ga_1}}+\mathcal{C})\le\la$ for any $\la>\la_p'(\mathcal{B}^{\mu_{\ga_2}}_{1, 0, \Omega_{\ga_2}}+\mathcal{C})$.
			
			Fix such a $\la$. By Proposition \ref{property}, there exists a function $\phi\in W^{1, 1}((0, a_2), X_{\ga_2})$ with $X_{\ga_2}=C(\overline\Omega_{\ga_2})$ or $X_{\ga_2}=L^1(\Omega_{\ga_2})$ satisfying $\phi>0$ in $[0, a_2]$ such that
			$$
			\left(-\mathcal{B}^{\mu_{\ga_2}}_{1, 0, \Omega_{\ga_2}}-\mathcal{C}\right)(0, \phi)+\la(0, \phi)\ge(0, 0), \text{ in }[0, a_2].
			$$
			Since $\Omega$ contains the origin, there hold $\Omega_{\ga_1}\subset\Omega_{\ga_2}$. Moreover, $\mu_{\ga_1}(a, x)\ge \mu_{\ga_2}(a, x)$ a.e. in $(0, a_2)\times\Omega_{\ga_1}$. Direct computations yields
			$$
			\left(-\mathcal{B}^{\mu_{\ga_1}}_{1, 0, \Omega_{\ga_1}}-\mathcal{C}\right)(0, \phi)+\la(0, \phi)\ge\left(-\mathcal{B}^{\mu_{\ga_2}}_{1, 0, \Omega_{\ga_2}}-\mathcal{C}\right)(0, \phi)+\la(0, \phi)\ge(0, 0), \text{ in }[0, a_2]\times\Omega_{\ga_1}.
			$$
			This implies $\la_p'(\mathcal{B}^{\mu_{\ga_1}}_{1, 0, \Omega_{\ga_1}}+\mathcal{C})\le\la$. Thus the proof is complete.
		\end{proof}
		
\begin{remark}
{\rm 
(i) Note that when $\beta(a, x)\equiv\beta(a)$ and $\mu(a, x)\equiv\mu(a)$, the age-structure and nonlocal diffusion can be decoupled, then the spectrum of $\mathcal{A}$ is quite clear, see Kang et al. \cite{Kang2020Age}. Thus the limiting properties of the principal eigenvalue of $\mathcal{A}$ is fully clear and is only determined by the one of nonlocal diffusion. Hence we omit the case.
				
(ii) Note that we did not discuss the case when $m=2$ and $\gamma\rightarrow0$. We conjecture that the principal eigenvalue for age-structured models with nonlocal diffusion converges to the one for age-structured models with Laplace diffusion. Actually, without age-structure, the autonomous nonlocal diffusion operator has a $L^2$ variational structure which can be used to show the convergence, see Berestycki et al. \cite{berestycki2016definition} and Su et al. \cite{su2020asymptotic}. While for the time-periodic nonlocal diffusion operators, Shen and Xie \cite{shen2015approximations,Shen2015principal} used the idea of a solution mapping to show the convergence, where they employed the spectral mapping theorem which is not valid in our case either, since we have a first order differential operator $\partial_a$ that is unbounded. However, when we add a nonlocal boundary condition to the birth rate $\beta$, it can be proved that the semigroup generated by solutions is eventually compact where spectral mapping theorem holds. Thus we can use it to show the desired convergence, see Kang and Ruan \cite{kang2021Approximation}.
			}
\end{remark}
		
		\section{Strong Maximum Principle}\label{Strong Maximum Principle}
		In this section via the sign of principal eigenvalue we establish the strong maximum principle for the operator $\mathcal{A}$ defined in \eqref{A} without kernel scaling, which is of fundamental importance and independent interest. We let Assumptions \ref{irreducible} and \ref{beta} hold, which are rewritten as follows.
		
		\begin{assumption}\label{both}
			{\rm There exists $a_2\in(0, \hat{a})$ such that $\beta\equiv0$ on $[a_2, \hat{a})\times\overline{\Omega}$ and $\int_a^{a_2}\underline\beta(l)dl>0, \forall a\in[0, a_2)$.}
		\end{assumption}
		
		\begin{definition}[Strong Maximum Principle]
			{\rm We say that $\mathcal{A}$ admits the \textit{strong maximum principle} if for any function $u\in W^{1, 1}((0, a_2), C(\overline\Omega))$ satisfying
				\begin{eqnarray}\label{SMP}
						\mathcal{A}(0, u)\leq (0, 0) &\;\text{ in } [0, a_2]\times\Omega,
				\end{eqnarray}	
				there must hold $u>0$ in $[0, a_2]\times\Omega$ unless $u\equiv0$ in $[0, a_2]\times\Omega$.}
		\end{definition}
		
		\begin{theorem}\label{MP}
			Let Assumption \ref{both} hold. Assume that $\mathcal{A}$ posses a principal eigenvalue $\la_1(\mathcal{A})$, then $\mathcal{A}$ admits the strong maximum principle if and only if $\la_1(\mathcal{A})<0$.
		\end{theorem}
		
		\begin{proof}
			If $\la_1(\mathcal{A})$ is the principal eigenvalue of $\mathcal{A}$ associated with an eigenfunction $(0, \phi)$ with $\phi\in W^{1, 1}((0, a_2), C(\overline{\Omega}))$ satisfying $\phi>0$, then
			$$ 
			\mathcal{A}(0, \phi)-\la_1(\mathcal{A})(0, \phi)=(0, 0); 
			$$
			that is
			\begin{equation}\label{bc}
				\begin{cases}
					-\partial_a\phi+D\left[\int_{\Omega}J(x-y)\phi(a, y)dy-\phi(a, x)\right]-\mu(a, x)\phi-\la_1(\mathcal{A})\phi=0,\\
					\phi(0, x)-\int_{0}^{a_2}\beta(a, x)\phi(a, x)da=0.
				\end{cases}
			\end{equation}
			For the sufficiency, that is $\la_1(\mathcal{A})<0$ implies the strong maximum principle, let $u\in W^{1, 1}((0, a_2), C(\overline\Omega))$ be nonzero and satisfy \eqref{SMP}. Assume by contradiction that there exists $(a_0, x_0)\in[0, a_2]\times\overline{\Omega}$ such that $u(a_0, x_0)=\min_{[0, a_2]\times\overline{\Omega}}u\le0$. Then consider the set
			$$ 
			\Gamma:=\{\epsilon\in\mathbb{R}: u+\epsilon\phi\geq0\,\text{in}\,[0, a_2]\times\overline\Omega\}. 
			$$
			Denote by $\epsilon_0=\min\Gamma$ and $\psi=u+\epsilon_0\phi$. It is clear that $\epsilon_0\geq0$ by the assumption of $u(a_0, x_0)\le0$ and that $\psi\geq0$. 
			
			Now if $\epsilon_0>0$, by simple computations, we have
			\begin{eqnarray}
				\begin{cases}
					\partial_a\psi-D\left[\int_{\Omega}J(x-y)\psi(a, y)dy-\psi(a, x)\right]+\mu(a, x)\psi\geq-\epsilon_0 \la_1(\mathcal{A}) \phi>0,&(a, x)\in(0, a_2)\times\overline{\Omega},\nonumber\\
					\psi(0, x)\geq\int_{0}^{a_2}\beta(a, x)\psi(a, x)da,&x\in\overline{\Omega}.\nonumber
				\end{cases}
			\end{eqnarray}
			That is,
			\begin{eqnarray}\label{>}
				\begin{cases}
					\partial_a\psi>D\left[\int_{\Omega}J(x-y)\psi(a, y)dy-\psi(a, x)\right]-\mu(a, x)\psi,&(a, x)\in(0, a_2)\times\overline{\Omega},\\
					\psi(0, x)\geq\int_{0}^{a_2}\beta(a, x)\psi(a, x)da,&x\in\overline{\Omega}.
				\end{cases}
			\end{eqnarray}
			It follows from the first inequality in \eqref{>} that $\psi(a, \cdot)>\mathcal{U}(0, a)\psi(0, \cdot)\geq0$ for $(a, x)\in(0, a_2]\times\Omega$. Plugging it into the second inequality, we have $\psi(0, \cdot)>0,$ which by comparison principle implies that $\psi$ is strictly positive in $[0, a_2]\times\Omega$. This contradicts the fact that $\epsilon_0$ is the infimum of $\Gamma$. 
			
			If $\epsilon_0=0$, it follows that $u\ge0$ and thus $u(a_0, x_0)=0$. Then if $a_0>0$, recalling again the constant of variation formula \eqref{U}, one has
			\begin{eqnarray}\label{u}
				u(a, x)\ge e^{-Da}\pi(0, a, x)u(0, x)+D\int_0^ae^{-D(a-l)}\pi(l, a, x)[Ku](l, x)dl.
			\end{eqnarray}
			Considering the above inequality at $(a_0, x_0)$, it follows that for any $l\in[0, a_0]$, one has $[Ku](l, x_0)=0$ and thus $u(l, x_1)=0$ for all $x_1\in B(x_0, r)$. Next consider \eqref{u} at $(l, x_1)$, one has $u(l, x_2)=0$ for all $x_2\in B(x_1, r)$. Then continue this process as we did in the proof of Theorem \ref{sim}, we can get $u(l, \cdot)\equiv0$ in $\overline{\Omega}\cap B(x_0, nr)$ with some $n\in\N$ large enough for all $l\in[0, a_0]$. On the other hand, by the nonlocal equation, the solution starting at $u(a_0, \cdot)\equiv0$ will be zero; i.e., $u(l, \cdot)\equiv0$ when $l>a_0$, which implies $u\equiv0$. This contradicts the fact that $u$ is nonzero.
			
			If $a_0=0$; that is, $u(0, x_0)=0$, then the integral boundary condition implies that
			$$
			\int_{0}^{a_2}\beta(a, x_0)u(a, x_0)da\leq u(0, x_0)=0
			$$ 
which shows that $u(\cdot, x_0)=0$ somewhere in $[0, a_2]$. By Assumption \ref{both}, we can choose a point $\widetilde a\in [0, a_2]$ and $\widetilde a\neq 0$ such that $u(\widetilde a, x_0)=0$. Considering the equation \eqref{u} at $(\widetilde a, x_0)$, we have the same contradiction as above. Hence $u>0$ in $[0, a_2]\times\Omega$, which concludes the desired result.
			
			For the necessity, that is, strong maximum principle implies $\la_1(\mathcal{A})<0$, the proof is almost identical to that of Shen and Vo \cite[Theorem F]{shen2019nonlocal} once noting the boundary condition is kept invariant, i.e. $\int_0^{a_2}\beta(a, \cdot)\phi(a)da=\phi(0)$, thus is omitted here.
		\end{proof}
		
		\section{Discussions}\label{Discussions}
		Age-structured models with nonlocal diffusion could be used to characterize the spatio-temporal dynamics of biological species and transmission dynamics of infectious diseases in which the age structure of the population is a very important factor and the dispersal is in long distance. There are very few theoretical studies on the dynamics of such equations. In this paper, we studied the spectrum theory for age-structured models with nonlocal diffusion. First we gave sufficient conditions on the existence of principal eigenvalues and presented a counterexample in which no principal eigenvalue exists. Then we used the generalized principal eigenvalue to characterize the principal eigenvalue and applied it to discuss the effects of the diffusion rate on the principal eigenvalue. Finally we established the strong maximum principle for such age-structured models with nonlocal diffusion. In our forthcoming paper \cite{Ducrot2022Age-structuredII} we will investigate the existence, uniqueness and stability of such equations with monotone type of nonlinearity on the birth rate. 
		
		We expect that the results on the principal eigenvalue and the construction of sub- and super-solutions can be applied to study traveling or pulsating wave solutions and spreading speeds of age-structured models with nonlocal diffusion (see Ducrot \cite{ducrot2007travelling}, Ducrot et al. \cite{ducrot2009travelling,ducrot2010travelling,ducrot2011travelling} with random diffusion) and we leave this for future consideration.

		\renewcommand{\thesection}{\Alph{section}}
		\setcounter{theorem}{0}
		\setcounter{equation}{0}
		\setcounter{section}{0}
		\section{Appendix}
		
		\subsection{Resolvent Positive Operators Theory}							
		In this Appendix we recall the theory of resolvent positive operators, the readers can refer to Thieme \cite{thieme1998remarks,thieme2009spectral} for details. A linear operator $A: Z_1\rightarrow Z$, defined on a linear subspace $Z_1$ of $Z$, is called \textit{positive} if $Ax\in Z_+$ for all $x\in Z_1\cap Z_+$ and $A$ is not the $0$ operator, where $Z_+$ is a closed convex cone that is normal and generating, i.e. $Z=Z_+-Z_+, Z^*=Z^*_+-Z^*_+$.
		\begin{definition}
			{\rm A closed operator $A$ in $Z$ is called \textit{resolvent positive} if the resolvent set of $A$, $\rho(A)$, contains a ray $(\omega, \infty)$ and $(\lambda I - A)^{-1}$ is a positive operator (i.e. maps $Z_+$ into $Z_+$) for all $\lambda>\omega$.}
		\end{definition}
		
\begin{definition}
{\rm We define the \textit{spectral bound} of a closed operator $A$ by
				$$ 
				s(A)=\sup\{{\rm Re}\lambda\in\mathbb{R}; \lambda\in\sigma(A)\}
				$$
				and the \textit{real spectral bound} of $A$ by
				$$ 
				s_{\mathbb{R}}(A)=\sup\{\lambda\in\mathbb{R}; \lambda\in\sigma(A)\}. 
				$$
				Moreover, if $A$ is a bounded linear operator, its \textit{spectral radius} $r(A)$ is given by
				$$
				r(A)=\sup\{|\la|\in\R; \la\in\sigma(A)\}.
				$$
			}
		\end{definition}
		\begin{definition}\label{growth bound}
			{\rm A semigroup $\{S(t)\}_{t\geq0}$ is said to be \textit{essentially compact} if its \textit{essential growth bound} $\omega_1(S)$ is strictly smaller than its \textit{growth bound} $\omega(S)$, where the growth bound and essential growth bound are defined respectively by
				\begin{eqnarray}\label{A1}
					&&\omega(S):=\lim_{t\to\infty}\frac{\log\norm{S(t)}}{t},\;\omega_1(S):=\lim_{t\to\infty}\frac{\log\alpha[S(t)]}{t},
				\end{eqnarray}
				and $\alpha$ denotes the \textit{measure of noncompactness}, which is defined by
				$$ 
				\alpha[L]=\inf\{\epsilon>0, L(B)\,\text{can be covered by a finite number of balls of radius}\,\leq\epsilon\}, 
				$$ 
				where $L$ is a closed and bounded linear operator in $Z$ and $B$ is the unit ball of $Z$.}
		\end{definition}
		By the formulas 
		$$ 
		r_e(S(t))=e^{\omega_1(S)t},\; r(S(t))=e^{\omega(S)t}, 
		$$
		we can see that equivalently $r_e(S(t))$ (the essential spectral radius of $S(t)$) is strictly smaller than $r(S(t))$ (the spectral radius of $S(t)$) for one (actually for all) $t>0$.
		
		Denote the {\it part} of $A$ in $\overline{{\rm dom}(A)}$ by $A_0$ and the {\it part} of $B$ in $\overline{{\rm dom}(B)}$ by $B_0$, respectively. Let $A_0$ and $B_0$ generate positive $C_0$-semigroups $\{S_{A_0}(t)\}_{t\geq0}$ and $\{T_{B_0}(t)\}_{t\geq0}$, respectively. If $Z$ is an abstract L space (that is, a Banach lattice for which the norm is additive on the positive cone $Z_+$) and $A$ and $B$ are resolvent positive, then by \cite[Proposition 2.4]{thieme1998positive} we have
		$$ 
		s(A)=s(A_0)=\omega(S), \; s(B)=s(B_0)=\omega(T). 
		$$
		If $B$ is a resolvent positive operator and $C: {\rm dom}(B)\rightarrow Z$ is a positive linear operator, then $A=B+C$ is called a \textit{positive perturbation} of $B$. If $B+C$ is a positive perturbation of $B$ and $\lambda>s(B)$, then $C(\lambda I -B)^{-1}$ is automatically bounded (without $C$ being necessarily closed). This is a consequence of $Z_+$ being normal and generating. 
		
		\begin{theorem}[Thieme {\cite[Theorem 3.5]{thieme1998remarks}}]\label{sR}
			Let the cone $Z_+$ be normal and generating and $A$ be a resolvent positive operator in $Z$. Then $s(A)=s_{\mathbb{R}}(A)<\infty$ and $s(A)\in\sigma(A)$ whenever $s(A)>-\infty$; further there is a constant $c>0$ such that
			$$ \norm{(\lambda I - A)^{-1}}\leq c\norm{({\rm Re}\lambda I - A)^{-1}}\text{ whenever } {\rm Re}\lambda>s(A). $$
		\end{theorem}
		
		\begin{corollary}[Thieme {\cite[Corollary 3.6]{thieme1998remarks}}]\label{spr}
			Let the cone $Z_+$ be normal and generating and $A$ be a resolvent positive operator in $Z$ with $\la>s(A)$. Then
			$$
			r((\la I -A)^{-1})=(\la -s(A))^{-1}.
			$$
		\end{corollary}
		\begin{theorem}[Thieme {\cite[Theorems 3.4 and 4.9]{thieme1998positive}}]\label{EC}
			If $C$ is a compact perturbator of $B$, $S_{A_0}(t)-T_{B_0}(t)$ is a compact operator for $t\geq0$. Moreover, if $\omega(T)<\omega(S)$, then $\{S_{A_0}(t)\}_{t\geq0}$ is an essentially compact semigroup.
		\end{theorem}
		
		Now define a positive resolvent output family for $B$ by
		\begin{equation}\label{Flambda}
			F_\lambda=C(\lambda I - B)^{-1}, \; \lambda>s(B).
		\end{equation}
		
		\begin{theorem}[Thieme {\cite[Theorem 3.6]{thieme2009spectral}}]\label{rFlambda}
			Let $Z$ be an ordered Banach space with normal and generating cone $Z_+$ and let $A=B+C$ be a positive perturbation of $B$. Then $r(F_\lambda)$ is a decreasing convex function of $\lambda>s(B)$, and exactly one of the following three cases holds:
			\begin{itemize}
				\item [(i)] if $r(F_\lambda)\geq1$ for all $\lambda>s(B)$, then $A$ is not resolvent positive;
				\item [(ii)] if $r(F_\lambda)<1$ for all $\lambda>s(B)$, then $A$ is resolvent positive and $s(A)=s(B)$;
				\item [(iii)] if there exists $\nu>\lambda>s(B)$ such that $r(F_\nu)<1\leq r(F_\lambda)$, then $A$ is resolvent-positive and $s(B)<s(A)<\infty$; further $s=s(A)$ is characterized by $r(F_s)=1$.
			\end{itemize}
		\end{theorem}
		
		\begin{definition}
			{\rm The operator $C: {\rm dom}(B)\rightarrow Z$ is called a \textit{compact perturbator} of $B$ and $A=B+C$ a \textit{compact perturbation} of $B$ if
				$$ (\lambda I - B)^{-1}F_\lambda: \overline{{\rm dom}(B)}\rightarrow\overline{{\rm dom}(B)} \text{ is compact for some } \lambda>s(B) $$
				and $$ (\lambda I - B)^{-1}(F_\lambda)^2: Z\rightarrow Z \text{ is compact for some } \lambda>s(B). $$
				$C$ is called an \textit{essentially compact perturbato}r of $B$ and $A=B+C$ an \textit{essentially compact perturbation} of $B$ if there is some $n\in\mathbb{N}$ such that $(\lambda I - B)^{-1}(F_\lambda)^n$ is compact for all $\lambda>s(B)$.}
		\end{definition}
		
		\begin{definition}\label{CSP}
			{\rm Let $F_\lambda$ be a positive resolvent output family for $B$.	
				A vector $x\in X_+$ is called \textit{conditionally strictly positive} if the following holds:
				\begin{itemize}
					\item[] If $x^*\in Z_+^*$ and $F^*_\lambda x^*\neq0$ for some (and then for all) $\lambda>s(B)$, then $\langle x, x^*\rangle>0$. 
				\end{itemize}
				Similarly a functional $x^*\in Z_+^*$ is said to be \textit{conditionally strictly positive} if the following holds:
				\begin{itemize}
					\item[] If $x\in Z_+$ and $F_\lambda x\neq0$ for some (and then for all) $\lambda>s(B)$, then $\langle x, x^*\rangle>0$.
				\end{itemize}
			}
		\end{definition}
		
		\begin{theorem}[Thieme {\cite[Theorems 4.7 and 4.9]{thieme1998remarks}}]\label{compactperturbation}
			Assume that $C$ is an essentially compact perturbator of $B$. Moreover assume that there exists $\lambda_2>\lambda_1>s(B)$ such that $r(F_{\lambda_1})\geq 1> r(F_{\lambda_2})$. Then $s(B)<s(A)<\infty$ and the following hold:
			\begin{itemize}
				\item [(i)] $s(A)$ is an eigenvalue of $A$ associated with positive eigenvectors of $A$ and $A^*$, has finite algebraic multiplicity, and is a pole of the resolvent of $A$. If $C$ is a compact perturbator of $B$, then all spectral values $\lambda$ of $A$ with ${\rm Re}\lambda\in(s(B), s(A)]$ are poles of the resolvent of $A$ and eigenvalues of $A$ with finite algebraic multiplicity;
				\item [(ii)] $1$ is an eigenvalue of $F_{s(A)}$ and is associated with an eigenvector $w\in Z$ of $F_{s(A)}$ such that $(\lambda I - B)^{-1}w\in Z_+$ and with an eigenvector $v^*\in Z_+^*$ of $F^*_{s(A)}$. Actually $s(A)$ is the largest $\lambda\in\mathbb{R}$ for which $1$ is an eigenvalue of $F_\lambda$.
			\end{itemize}
			Moreover, if $Z$ is a Banach lattice and there exists a fixed point of $F^*_s$ in $Z_+^*$ that is conditionally strictly positive, then the following hold:
			\begin{itemize}
				\item [(iii)] $s=s(A)$ is associated with a positive eigenvector $v$ of $A$ such that $w=(s(A) I -B)v$ is a positive fixed point of $F_{s(A)}$;
				\item [(iv)] $s$ is the only eigenvalue of $A$ associated with a positive eigenvector.
			\end{itemize}
			Finally, we assume in addition that all positive non-zero fixed points of $F_s$ are conditionally strictly positive. Then the following hold:
			\begin{itemize}
				\item [(v)] $s=s(A)$ is a first order pole of the resolvent of $A$;
				\item [(vi)] The eigenspace of $A$ associated with $s(A)$ is one-dimensional and spanned by a positive eigenvector $v$ of $A$. The eigenspace of $A^*$ associated with $s(A)$ is also spanned by a positive eigenvector $v^*$.
			\end{itemize}
		\end{theorem}
		
		\subsection{Theorem \ref{I} when $X=L^1(\Omega)$}
		\begin{proof}
			Note that in the proof of Theorem \ref{I}, the arguments are still valid before \eqref{com}. Thus we only show the latter part after \eqref{com}. We denote the area of $\Omega$ by $|\Omega|$. 
			
			As a result, 
			\begin{eqnarray}
				\norm{(\mathcal{B}_2(\alpha I -\mathcal{B}_1-\mathcal{C})^{-1})^n}\ge  \frac{1}{|\Omega|}\int_{\Omega}\cdots\int_{\Omega}\prod_{m=1}^{n}\left[J(x_{m-1}-x_m)\frac{DM(\alpha, D, \theta)}{1-G_\alpha(x_m)}\right]dx_n\cdots dx_0,\nonumber
			\end{eqnarray}
			which implies that for any $x_0\in\overline{\Omega}$ and $\delta>0$, 
			\begin{eqnarray}
				&&\norm{(\mathcal{B}_2(\alpha I -\mathcal{B}_1-\mathcal{C})^{-1})^n}\nonumber\\
				&&\quad \geq \frac{1}{|\Omega|}\int_{\Omega} \int_{\Omega\cap B(x_0, \delta)}\cdots\int_{\Omega\cap B(x_0, \delta)}\prod_{m=1}^{n}\left[J(x_{m-1}-x_m)\frac{DM(\alpha, D, \theta)}{1-G_\alpha(x_m)}\right]dx_n\cdots dx_1 dx_0\nonumber\\
				&&\quad \geq \frac{1}{|\Omega|}\int_{\Omega}\left[\inf_{x\in\Omega\cap B(x_0, \delta)}\int_{\Omega\cap B(x_0, \delta)}J(x-y)\frac{DM(\alpha, D, \theta)}{1-G_\alpha(y)}dy\right]^n dx_0,
			\end{eqnarray}
			where $B(x_0, \delta)$ is the open ball in $\mathbb{R}^N$ centered at $x_0$ with radius $\delta$. We can use \eqref{r<1} and Gelfand's formula for the spectral radius of a bounded linear operator to find that 
			\begin{eqnarray}\label{11}
				|\Omega|\geq \int_{\Omega} \inf_{x\in\Omega\cap B(x_0, \delta)}\int_{\Omega\cap B(x_0, \delta)}J(x-y)\frac{DM(\alpha, D, \theta)}{1-G_\alpha(y)}dy dx_0:=\int_{\Omega}I(x_0, \delta, \alpha, D)dx_0
			\end{eqnarray}
			for all $\delta>0$. {Note that $x_0\to I(x_0, \delta, \alpha, D)$ is continuous since the integrand is continuous.}
			
			Since $J$ is continuous and $J(0)>0$, there exist $r>0$ and $c_0>0$ such that $J\geq c_0$ on $B(0, r)$, the open ball in $\mathbb{R}^N$ centered at $0$ with radius $r$. Hence,
			\begin{eqnarray}
				I(x_0, \delta, \alpha, D) &\geq& \inf_{x\in\Omega\cap B(x_0, \delta)}\int_{\Omega\cap B(x_0, \delta)\cap B(x, r)}J(x-y)\frac{DM(\alpha, D, \theta)}{1-G_\alpha(y)}dy\nonumber\\
				&\geq& c_0\inf_{x\in\Omega\cap B(x_0, \delta)}\int_{\Omega\cap B(x_0, \delta)\cap B(x, r)}\frac{DM(\alpha, D, \theta)}{1-G_\alpha(y)}dy\nonumber\\
				&=& c_0\int_{\Omega\cap B(x_0, \delta)}\frac{DM(\alpha, D, \theta)}{1-G_\alpha(y)}dy\label{c0}
			\end{eqnarray}
			provided $2\delta\leq r$ so that $B(x_0, \delta)\subset B(x, r)$ whenever $x\in\overline{B(x_0, \delta)}$. In particular, for any $x_0\in\overline{\Omega}$, 
			$$ 
			I(x_0, r/2, \alpha, D)\geq c_0\int_{\Omega\cap B(x_0, r/2)}\frac{DM(\alpha, D, \theta)}{1-G_\alpha(y)}dy. 
			$$
			Now we fix this $\de$. Since $\frac{1}{1-G_{\alpha^{**}}}\notin L^1_{loc}(\overline{\Omega})$, there exists $x_*\in\overline{\Omega}$ such that
			$$ 
			\frac{1}{1-G_{\alpha^{**}}}\notin L^1(\overline{\Omega}\cap B(x_*, r/2)), 
			$$
			which implies the existence of some $\epsilon>0$ small enough, such that
			\begin{eqnarray}\label{de}
			c_0\int_{\Omega\cap B(x_*, r/2)}\frac{DM(\alpha^{**}+\epsilon, D, \theta)}{1-G_{\alpha^{**}+\epsilon}(y)}dy\geq \frac{2|\Omega|}{|\Omega\cap B(0, \de)|}. 
		    \end{eqnarray}
			It follows from \eqref{c0} that
			\begin{eqnarray}
				\int_{\Omega}I(x, \de, \alpha^{**}+\epsilon, D)dx&\ge& c_0\int_{\Omega\cap B(x_*, r)}\int_{\Omega\cap B(x, \de)}\frac{DM(\alpha^{**}+\epsilon, D, \theta)}{1-G_{\alpha^{**}+\epsilon}(y)}dydx\nonumber\\
				&=&c_0\int_{\Omega\cap B(x_*, r)}\int_{\Omega\cap B(0, \de)}\frac{DM(\alpha^{**}+\epsilon, D, \theta)}{1-G_{\alpha^{**}+\epsilon}(y+x)}dydx\nonumber\\
				&=&c_0\int_{\Omega\cap B(0, \de)}\int_{\Omega\cap B(x_*, r)}\frac{DM(\alpha^{**}+\epsilon, D, \theta)}{1-G_{\alpha^{**}+\epsilon}(y+x)}dxdy.\nonumber
			\end{eqnarray}
		    Next define $\Phi: (\Omega\cap B(x_*, r))\times(\Omega\cap B(0, \de))\to \R^{2N}$ by $\Phi(x, y)=(x+y, y):=(u, v)$. It follows that the Jacobian determinant $\left|\frac{\partial(u, v)}{\partial(x, y)}\right|=1$. Since $2\de\le r$, one has 
		    $$
		    (\Omega\cap B(x_*, r/2))\times(\Omega\cap B(0, \de))\subset \Phi\left((\Omega\cap B(x_*, r))\times(\Omega\cap B(0, \de))\right).
		    $$
		    Thus by the change of variable formula for double integrals and \eqref{de} we have
		    \begin{eqnarray}
		    	&&c_0\int_{\Omega\cap B(0, \de)}\int_{\Omega\cap B(x_*, r)}\frac{DM(\alpha^{**}+\epsilon, D, \theta)}{1-G_{\alpha^{**}+\epsilon}(y+x)}dxdy\nonumber\\
		    	&\ge&c_0\int_{\Omega\cap B(0, \de)}\int_{\Omega\cap B(x_*, r/2)}\frac{DM(\alpha^{**}+\epsilon, D, \theta)}{1-G_{{\alpha^{**}+\epsilon}}(u)}dudv\nonumber\\
		    	&\ge&2|\Omega|.\nonumber
		    \end{eqnarray}
			But	this contradicts \eqref{11}. Thus our proof is complete.
		\end{proof}
		
		\subsection{$C^k$ regularity of eigenfunctions}
		For any $x\in\R^N$, we define two operators $A(x):\{0\}\times W^{1, 1}(0, a_2)\to \R\times L^1(0, a_2)$ and $F(x): \{0\}\times C([0, a_2])\to \R\times C([0, a_2])$ respectively as follows:
		$$
		A(x)(0, u):=\left(-u(0), -\partial_au-\mu(a, x)u\right),\; F(x)(0, u):=\left(\int_0^{a_2}\beta(a, x)u(a)da, 0\right).
		$$
		{Next we denote the principal eigenfunction of $A(x)+F(x)$ associated with principal eigenvalue $\al(x)$ for $x\in\R^N$ by $(0, \phi(\cdot,x))$ with normalization as follows,
		\begin{equation}\label{normalization}
		\int_0^{a_2}\beta(a, x)\phi(a,x)da=1,\;\forall x\in\R^N.
		\end{equation}
		Next for $x\in\R^N$ define a map $\mathcal{H}:\{0\}\times C([0,  a_2])\times(\al^{**}+D, \infty)\times\R^N\to \{0\}\times C([0, a_2])\times\R$ by
		\begin{equation}
			\mathcal H((0, u), \al, x)=\left((\al I -A(x))^{-1}F(x)(0, u)-(0, u), \int_0^{ a_2}\beta(a, x)u(a)da-1\right),
		\end{equation} 
		where $\al^{**}$ is from Proposition \ref{rGalpha} with $D=0$. In this subsection, we prove the following lemma.
		\begin{lemma}\label{C^4}
			Assume $\mu, \beta\in C^k(\R^N, L^\infty_+(0, a_2))$ with $k\ge0$, then the map $x\to((0, \phi(\cdot,x)), \al(x))$ is of $C^k$ from $\R^N$ to $\{0\}\times C([0,  a_2])\times\R$.
		\end{lemma}}
		\begin{proof}
	
	To prove the above result, first note by Proposition \ref{Galphax} with $D=0$ that one has
	\begin{equation*}
	\mathcal H\bigl((0, \phi(\cdot,x)), \al(x), x\bigr)=\left((0,0),0\right),\;\forall x\in\R^N.
	\end{equation*}	
		Hence the smoothness of $\phi$ and $\alpha$ will follow from the implicit function theorem applied to the map $\mathcal H$.
		
		To that aim we fix $x_0\in\R^N$ and set 
		$$
		\al_0=\alpha(x_0)\text{ and }\phi_0(a)=\phi(a,x_0).
		$$ 
		Since $\mathcal H$ is $C^k-$smooth, to apply the implicit function theorem and to prove the smooth dependence with respect to $x$, it suffices to show that 
			\begin{eqnarray}
				&&D_{((0, u), \al)}\mathcal H((0, \phi_0), \al_0, x_0)((0, v), \eta)\nonumber\\
				&& \quad = \left((\al_0 I -A(x_0))^{-1}F(x_0)(0, v)-(0, v)-\eta(\al_0 I -A(x_0))^{-2}F(x_0)(0, \phi_0), \int_0^{a_2}\beta(a, x_0)v(a)da\right)\nonumber
			\end{eqnarray}
			is invertible as a linear mapping from	$\{0\}\times C([0, a_2])\times \R$ into $\{0\}\times C([0, a_2])\times \R$.
			
			To this end, given $((0, f), \psi)\in \{0\}\times C([0, a_2])\times \R$, we need to prove the existence and uniqueness of $((0, v), \eta)\in \{0\}\times C([0, a_2])\times \R$ such that
			\begin{equation}\label{ex}
				\begin{cases}
					(\al_0 I -A(x_0))^{-1}F(x_0)(0, v)-(0, v)-\eta(\al_0 I -A(x_0))^{-2}F(x_0)(0, \phi_0)=(0, f),\\
					\int_0^{a_2}\beta(a, x_0)v(a)da=\psi.
				\end{cases}
			\end{equation} 
			Note that $F(x_0)$ is finite rank and thus a compact operator. It follows that the operator $(\al_0 I -A(x_0))^{-1}F(x_0)$ is compact from $\{0\}\times C([0, a_2])$ to $\{0\}\times C([0, a_2])$. Hence, $(\al(x_0) I -A(x_0))^{-1}F(x_0)-I$ is a Fredholm operator with index $0$.
			
			Next we compute the adjoint operator of $T:=(\al_0 I -A(x_0))^{-1}F(x_0)$, which is denoted by $T^*$. Let us first clarify the dual space of $C([0, a_2])$, denoted by $C^*([0, a_2])$, which collects all the Radon measures in $[0, a_2]$, with the dual product given as follows,
			$$
			\langle w^*, w\rangle:=\int_0^{a_2}w(s) w^*(ds), \;\forall w^*\in C^*([0, a_2]), w\in C([0, a_2]).
			$$

			Now by definition, for any $w\in C([0, a_2])$ and $w^*\in C^*[0, a_2]$, we have $\langle (0, w^*), T(0, w)\rangle=\langle T^*(0, w^*), (0, w)\rangle$; that is,
			$$
			\int_0^{a_2}e^{-\al_0a}\pi(0, a, x_0)\int_0^{a_2}\beta(s, x_0)w(s)dsw^*(da)=\int_0^{a_2}w(s)\beta(s, x_0)\int_0^{a_2}e^{-\al_0a}\pi(0, a, x_0)w^*(da)ds.
			$$
			It follows that
			$$
			T^*(0, w^*)=\left(0,\; \beta(\cdot, x_0)\int_0^{a_2}e^{-\al_0a}\pi(0, a, x_0)w^*(da)\right),\;\forall w^*\in C^*([0, a_2]).
			$$
			Now by Fredholm Alternative, the first equation of \eqref{ex} has a unique solution $(0, v_1)\in\{0\}\times C([0, a_2])$ if and only if 
			\begin{eqnarray}
				&&(0, f)+\eta(\al_0 I -A(x_0))^{-2}F(x_0)(0, \phi_0)\in N(I-T^*)^\perp\nonumber\\
				\text{ with }&&
				N(I-T^*)^\perp:=\left\{(0, h)\in \{0\}\times C([0, a_2]): \int_0^{a_2}h(a)dw^*(a)da=0, \forall (0, w^*)\in N(I-T^*)\right\},\nonumber
			\end{eqnarray}
			where $N(T)$ denotes the kernel of $T$. Moreover, from the relation $T^*(0, w^*)=(0, w^*)$,  one obtains 
			\begin{equation}\label{perp}
				N(I-T^*)=\text{span}\{a\to\beta(a, x_0)da\}\subset C^*([0, a_2]).
			\end{equation}
			On the other hand, by some computations and \eqref{normalization}, one obtains
			\begin{eqnarray}
				(\al_0 I -A(x_0))^{-1}F(x_0)(0, v)&=&\left(0, e^{-\al_0a}\pi(0, a, x_0)\int_0^{a_2}\beta(s, x_0)v(s)ds\right),\nonumber\\
				(\al_0 I -A(x_0))^{-2}F(x_0)(0, \phi_0)&=&\left(0, ae^{-\al_0a}\pi(0, a, x_0)\right).\nonumber
			\end{eqnarray}
			It follows from \eqref{perp} that 
			\begin{equation}\label{eta}
				\int_0^{a_2}\beta(a, x_0)f(a)da+\eta\int_0^{a_2}\beta(a, x_0)ae^{-\al_0a}\pi(0, a, x_0)da=0.
			\end{equation}
			Observe that $\eta$ can be uniquely solved from the above equation as follows,
			$$
			\eta=-\frac{\int_0^{a_2}\beta(a, x_0)f(a)da}{\int_0^{a_2}\beta(a, x_0)ae^{-\al_0a}\pi(0, a, x_0)da}.
			$$
			Now observe that 
			$$
			\{0\}\times C([0, a_2])=\text{span}\{(0, \phi_0)\}\oplus \text{Range}(I-T)=\text{span}\{(0, \phi_0)\}\oplus N(I-T^*)^\perp.
			$$ 
			Then Fredholm Alternative will give us a unique solution $(0, v_1)$ of \eqref{ex}, which is in $N(I-T^*)^\perp$. Finally, let us show that the complete solution of \eqref{ex} in $\{0\}\times C([0, a_2])$. 
			If $v\in \{0\}\times C([0, a_2])$, it can be decomposed as $v=\vartheta \phi_0+v_1$, where $v_1\in N(I-T^*)^\perp$ can be solved as above. From the second equation of \eqref{ex} we can figure out $\vartheta$. Once it is done, the complete solution of \eqref{ex} exists and is unique. Indeed, one has from \eqref{ex} and \eqref{normalization} that
			\begin{eqnarray}
				\psi=\int_0^{a_2}\beta(a, x_0)v(a)da&=&\vartheta\int_0^{a_2}\beta(a, x_0)\phi_0(a)da+\int_0^{a_2}\beta(a, x_0)v_1(a)da\nonumber\\
				&=&\vartheta+\int_0^{a_2}\beta(a, x_0)v_1(a)da, \nonumber
			\end{eqnarray}
			which implies that $\vartheta=\psi-\int_0^{a_2}\beta(a, x_0)v_1(a)da$.
			
			Finally, bounded inverse theorem applies to the linear map $D_{((0, u), \al)}\mathcal H((0, \phi(\cdot,x_0)), \al(x_0), x_0)$ and concludes that its inverse is also linear and bounded, and thus we can use implicit function theorem to conclude the $C^k-$smoothness of the principal eigenpair as stated in the result.
		\end{proof}
		
		\bibliography{hulk}
		\bibliographystyle{plain}
		\addcontentsline{toc}{section}{\refname}
		
	\end{document}